\documentclass[10pt]{article}
\usepackage{fullpage}
\usepackage{color, tikz}
\usepackage{amsfonts}
\usepackage{amssymb}
\usepackage{changepage}
\usepackage{amsthm}
\usepackage{amsmath}
\usepackage{bm}
\usepackage{esint}
\usepackage{lscape}
\usepackage{hyperref}
\usepackage{verbatim}
\usepackage{authblk}
\allowdisplaybreaks

\makeatletter
\DeclareRobustCommand\widecheck[1]{{\mathpalette\@widecheck{#1}}}
\def\@widecheck#1#2{%
    \setbox\z@\hbox{\m@th$#1#2$}%
    \setbox\tw@\hbox{\m@th$#1%
       \widehat{%
          \vrule\@width\z@\@height\ht\z@
          \vrule\@height\z@\@width\wd\z@}$}%
    \dp\tw@-\ht\z@
    \@tempdima\ht\z@ \advance\@tempdima2\ht\tw@ \divide\@tempdima\thr@@
    \setbox\tw@\hbox{%
       \raise\@tempdima\hbox{\scalebox{1}[-1]{\lower\@tempdima\box
\tw@}}}%
    {\ooalign{\box\tw@ \cr \box\z@}}}
\makeatother

\newcommand{\sym}{\mathop{{\rm sym}\!}}

\newtheorem{theorem}{Theorem}[section]
\newtheorem{proposition}[theorem]{Proposition}

\newtheorem{lemma}[theorem]{Lemma}

\newtheorem{remark}[theorem]{Remark}

\begin{document}

\title{\sc Sharp operator-norm asymptotics for thin elastic plates with rapidly oscillating periodic properties}

\author[1]{Kirill Cherednichenko}
\author[2]{Igor Vel\v{c}i\'{c}\,}
\affil[1]{Department of Mathematical Sciences, University of Bath, Claverton Down, Bath,\qquad\qquad  BA2 7AY, United Kingdom (Email: k.cherednichenko@bath.ac.uk)}
\affil[2]{Faculty of Electrical Engineering and Computing, University of Zagreb, Unska 3,\qquad\qquad 10000 Zagreb, Croatia (Email: igor.velcic@fer.hr)}
\maketitle

\begin{abstract}
 
  We analyse a system of partial differential equations describing the behaviour of an elastic plate with periodic moduli in the two planar directions, in the  asymptotic regime when the period and the plate thickness are of the same order.  Assuming that the displacement gradients of the points of the plate are small enough for the equations of linearised elasticity to be a suitable approximation of the material response, such as the case in {\it e.g.} acoustic wave propagation,
 we derive a class of ``hybrid'', homogenisation dimension-reduction, norm-resolvent estimates for the plate, under different energy scalings with respect to the plate thickness.
 
 
 
 \vskip 0.5cm


\vskip 0.5cm


\end{abstract}

\section{Introduction}

\label{intro_section}

The present work is a contribution to the analysis of the asymptotic behaviour of solutions $U=U^\varepsilon$ to systems of parameter-dependent partial differential equations with periodic coefficients, of the form 
\begin{equation}
{\mathcal D}^*\bigl(A(\cdot/\varepsilon){\mathcal D}U\bigr)+U=F,\qquad U\in X,\ \ F\in X^*,\qquad\varepsilon>0,
\label{general_eq}
\end{equation}
where ${\mathcal D}$ is a suitably defined differential operator, $X$ is the normed space of square-integrable functions  $U$ on a domain $\Pi^h \subset{\mathbb R}^3$ with values in ${\mathbb R}^3$ that have finite energy norm $\Vert{\mathcal D}U\Vert_{L^2(\Pi^h, {\mathbb R}^{3\times 3})},$ $X^*$ is the space of bounded linear functionals on $X,$ and $A$ is a periodic function taking values in the space of positive definite fourth-order tensors.   

\subsection{Periodic homogenisation: warm-up}

Since the early days of the mathematical theory of homogenisation, problems of the form (\ref{general_eq}) and their scalar counterparts, where $U$ is 
an ${\mathbb R}$-valued function and $A$ is matrix-valued, have served as the first step in understanding the macroscopic behaviour of multiscale media, see  \cite{BLP}, \cite{BP}, and references therein. This r\^{o}le has been motivated by the relatively straightforward structure of the basic functions in perturbation series with respect the parameter representing the ratio between the micro- and macroscale wavelengths in a given problem, as well as by the availability of multiscale compactness statements for solution sequences, with their limits usually having the form of the mentioned basic functions. A typical asymptotic series used in the related analysis for the solution $U^\varepsilon$ has the form
\begin{equation}
U^\varepsilon\sim U^{(0)}\biggl(x, \frac{x}{\varepsilon}\biggr)+
U^{(1)}\biggl(x, \frac{x}{\varepsilon}\biggr)+ 
U^{(2)}\biggl(x, \frac{x}{\varepsilon}\biggr)+\dots,\ \quad U^{(j)}=O(\varepsilon^j),\ \  j=0,1,2,\dots,\quad\quad\varepsilon\to0,
\label{two_scale_series}
\end{equation}
where the functions $U^{(j)},$ $j=0,1,2,\dots,$ 
are periodic with respect to their second argument, with the same period as the tensor of coefficients $A$ in the original equation (\ref{general_eq}).

The series (\ref{two_scale_series}) proves to be an effective tool for estimating the convergence error for solutions to $\varepsilon$-dependent families (\ref{general_eq}) with a fixed set of data, such as the density of the applied forces (represented by the right-hand side $F$). In this case the assumption of periodic dependence of the terms $U^{(j)},$ $j=0,1,2,\dots,$ on the ``fast'' variable $x/\varepsilon$ in (\ref{two_scale_series}) suffices for purposes of the analysis. However, the constants in the convergence estimates obtained in this way may blow up for sequences of data with bounded $L^2$-norm, which makes them unusable for controlling the behaviour of the spectra or eigenspaces of the corresponding operators.

\subsection{Operator estimates}

 In order to gain error bounds that are uniform with respect to the data
(and hence obtain a sharp quantitative description of the asymptotic behaviour of the spectrum), one can replace the series (\ref{two_scale_series}) by a family of power-series expansions parametrised by the ``quasimomentum'' $\chi,$ which represents variations over intermediate scales of the length of several periods, see  \cite{ChCoARMA}, \cite{Quasi_Cooper}. The corresponding asymptotic procedure can be viewed as a combination of the classical perturbation theory with ``matched asymptotics'' on the domain of the quasimomentum. In this approach, the control of the resolvent in the sense of the operator norm is obtained by means of a careful analysis of the remainder estimates for the power series, taking advantage of the related Poincar\'{e}-type inequalities (or Korn-type inequalities for vector problems) that bound the $L^2$-norm of the solution by its energy norm.  Importantly, in order to provide the required uniform estimates, such inequalities must reflect the fact that the lowest eigenvalue of the 
$\chi$-parametrised ``fibre" operator tends to zero as $\vert\chi\vert\to0,$ and hence the $L^2$-norm of the corresponding eigenfunction with unit energy blows up.

In the present work we tackle a problem where the above kind of uniform estimates needs to be controlled with respect to an additional length-scale parameter, which represents one of the overall dimensions of the medium, namely the thickness of a thin plate in our case. 

\subsection{Homogensation of thin plates}

The presence of a small parameter in the thin-plate context introduces a geometric anisotropy into the problem, which requires simultaneous error analysis of in-plane and out-of-plane components of the solution field. This leads to the presence of non-trivial invariant subspaces in the asymptotically equivalent problem obtained in the course of the analysis, yielding an asymptotic decomposition of the spectrum of the original problem into several distinct components. There are two such components in the case of a plate: they are responsible for different modes of wave propagation (for plate sizes large enough so the notion of wave propagation is meaningful.) The information about the values of the quasimomentum for which such propagating modes can be identified allows one to draw conclusions about the corresponding wavelengths (as, roughly, the inverses of the quasimomenta). This, in turn, provides asymptotics for the corresponding density of states (for a definition thereof, see {\it e.g.} \cite[Section XIII.16]{ReedSimon}) and hence allows one to evaluate the energy contribution of the associated propagating modes to an arbitrary wave packet.

In order to be able to implement the above ideas, we consider an infinite plate, which we view as a model of a plate of finite size when the ratio between the plate thickness and all of its in-plane dimensions is small. While our approach leads to new qualitative and quantitative results already in the case of homogeneous plates, here we consider plates whose material constants are periodic in the planar directions, so that the analysis can also be interpreted as simultaneous dimension reduction and homogenisation. From the perspective of dimension reduction, the estimates we obtain complement and strengthen the results on asymptotic two-dimensional linear plate models obtained in the 1980s through the work of Ciarlet and his collaborators, see \cite{ciarlet} and references therein. In the approach described in \cite{ciarlet}, where the displacements  are set to vanish on a part of the boundary of the plate, error estimates are derived for an equilibrium problem, for a given density of applied body forces. 
Our asymptotic estimates are of the operator-norm type, {\it i.e.} the approximation rate is proportional to the norm of the function representing the applied forces. In particular, by localising these estimates with respect to the quasimomentum $\chi$ and considering suitable path integrals of the 
$\chi$-parametrised resolvents as functions of the spectral parameter, they imply 
the operator-norm asymptotics of the spectral projections on bounded intervals of the real line. Furthermore, either suitably adapted to the parabolic setting or by a careful control of the dependence of the constant in the error estimates, our method readily yields estimates for the time-evolution semigroups in the context of thin plates. Finally, a further refinement of the approach allows one to treat the unitary groups for hyperbolic problems. Our strategy in addressing these contexts would follow the approaches of  \cite{Suslina_parabolic}, \cite{Suslina_parabolic_corrector}, \cite{Meshkova_Suslina}, \cite{BirmanSuslina_hyperbolic}, \cite{Meshkova_hyperbolic_Math_Notes}, \cite{Meshkova_hyperbolic_full}. We postpone the detailed analysis of these settings to a future publication.

\subsection{Comparison with other approaches to operator-norm estimates and simultaneous homogenisation and dimension reduction in linear elasticity}

A research programme similar to the above, although outside the context of thin structures and using a different analytical approach, has been pursued by Birman, Suslina, and subsequently by Suslina and her students, starting with \cite{BirmanSuslina}, \cite{BirmanSuslina_corrector}. At the heart of their technique is the notion of a spectral germ for operator pencils, which quantifies the leading order of frequency dispersion of waves in a heterogeneous medium near the bottom of the spectrum of the associated differential operator with periodic coefficients. Complemented with the analysis of a Cauchy integral for a suitable operator-valued function of the spectral parameter, the analysis of the spectral germ allows one to obtain sharp operator-norm estimates for the resolvents in the direct integral representing the original operator via the standard Floquet-Bloch-Gelfand decomposition (parametrised by the quasimomentum $\chi,$ as we mention above).  

Our approach is different in the way we approximate the operator functions entering the mentioned Cauchy integral and is closer in spirit to the classical asymptotics via power series, where it is convenient to use the quasimomentum $\chi$ as the expansion parameter. In this respect the technique we develop for plates also differs from the method introduced in \cite{ChCoARMA} outside the thin-structure context, where the expansion parameter is the period 
$\varepsilon$ (albeit it could certainly be carried out in terms of the quasimomentum,
following the strategy of the present work). Like in \cite{ChCoARMA}, and in contrast to the approach of Birman and Suslina, the key analytical ingredient in our analysis are $\chi$-uniform estimates for the remainders in these expansions. The estimates are based on appropriate Korn-type inequalities  (see Section \ref{Korn_section}), and any information concerning spectral behaviour is obtained by studying suitable Rayleigh quotients (see Section \ref{structure}) and is used for identifying elements of the asymptotic procedure rather than for spectral analysis as such. 

The approach initiated by Birman and Suslina has proved fruitful in obtaining operator-norm and energy estimates for a number of related problems:  boundary-value operators \cite{Suslina_Dirichlet}, \cite{Suslina_Neumann}, parabolic semigroups \cite{Suslina_parabolic}, \cite{Suslina_parabolic_corrector}, \cite{Meshkova_Suslina}, hyperbolic groups \cite{BirmanSuslina_hyperbolic}, \cite{Meshkova_hyperbolic_Math_Notes}, \cite{Meshkova_hyperbolic_full}, perforated domains \cite{Suslina_perforated}. The key technical milestones for this progress are boundary-layer analysis for bounded domains (as in \cite{Suslina_Dirichlet}, \cite{Suslina_Neumann}) and two-parametric operator-norm estimates \cite{Suslina_two_parametric}. It seems natural to conjecture that similar developments could be pursued in the context of thin plates, both infinite and bounded, by taking either the spectral germ approach or the one we introduce in the present article.  (See, however, the discussion at the end Section \ref{structure}
as well as Remark \ref{f3_omitted}.)

An overview of the existing approaches to obtaining operator-norm estimates would not be complete without mentioning also the works \cite{Griso_2006}, \cite{ZhikovPastukhova}, \cite{Kenig}, whose methods could also be considered in the context of thin structures. However, here we refrain from pursuing the related discussion.



In terms of the specific plates context, quantitative results for the spectrum on a bounded plate of thickness $h$ have been obtained in \cite{Dauge}, for the case of an isotropic homogeneous elasticity tensor. 
The authors separate the problem into two invariant subspaces and identify spectra of orders $h^2$ and $1.$
However, their estimates (due to the scaling of the spectrum in one of the two invariant subspaces) imply that the error of the approximation explodes on any compact frequency interval as $h\to0$. In Section \ref{normresolvent} we obtain an operator-norm resolvent approximation for the infinite plate, 
which allows us to provide more information concerning the behaviour of the subspace that \cite{Dauge} is unable to control, by scaling the elasticity operator with an appropriate power of $h,$ see also Remark \ref{expl3} below.

A variety of problems concerning simultaneous homogenisation and dimension reduction have been studied:
 in the finite-domain setting for plates and rods by \cite{Panasenko_book}, see also references therein, where the error of the asymptotics depends in a non-uniform way on the problem data, see also \cite{Caill}, \cite{Cioranescu_Damlamian_Griso}; in the context of general homogenisation for a linear plate by \cite{BV}, \cite{Dam}; in the non-linear context by \cite{BFF}  (membrane energy scaling),  \cite{NV_vK} (von K\`{a}rm\`{a}n plate), \cite{CC}, \cite{HNV} (bending energy).


In the present paper we focus on the most interesting
case when the period $\varepsilon$ of the composite structure is of the same order as the plate thickness $h.$ The analysis of the operator-norm resolvent asymptotics in other asymptotic regimes, when the two passages to the limit as $\varepsilon\to0$ and $h\to0$ are carried out sequentially, is outside the scope of our work, although the two-scale limit in the strong resolvent sense, {\it cf.} (\ref{two_scale_series}), is likely to involve either homogenisation of the order-four Kirchhoff-Love plate equations in the case when $h\ll\varepsilon$ or the thin-plate limit of a homogeneous plate when $h\gg\varepsilon.$ We conjecture that these also give the corresponding operator-norm resolvent asymptotics.



\subsection{Structure of the paper}

Here we briefly outline the structure of the manuscript. In Section \ref{formulation} we give problem formulation and the main statements --- the full set of results and their proofs are postponed to Section \ref{normresolvent}. Section \ref{strategynovelties} provides basic properties of Floquet transform, outlines our strategy for the proofs to follow, and describes the novel aspects of our approach.  In Section \ref{aux_sec} we give some auxiliary statements that are needed for Section \ref{asymptotic_proc}, where the main estimates are obtained through an asymptotic expansion. Section \ref{apriori_sec} is devoted to the derivation of {\it a priori} estimates for the solutions of the resolvent formulations we analyse in the present work --- this section is not needed for the proofs, but is helpful in understanding what is the form of approximation we are looking for.

\subsection{Notation}

We conclude the introductory section by listing some notation we use throughout the paper.

--- The notation $\cdot$ is used for the standard Euclidean inner product in 
${\mathbb R}^3$ or ${\mathbb C}^3$ (where the choice is always clear from the context);  

--- We use the following inner product $:$ in the space of $(3\times3)$-matrices with 
${\mathbb R}$- or ${\mathbb C}$-valued entries:
$$
M:N={\rm tr}\bigl(M^\top N\bigr),\qquad M, N\in {\mathbb R}^{3\times3}\ \bigl(M, N\in {\mathbb C}^{3\times3}\bigr);
$$

--- For a $(d\times d)$-matrix $M,$ $d=2,3,$
 its symmetrisation is defined by 
${\rm sym}\,M:=2^{-1}(M+M^\top);$

--- By $L^2(X,Y),$ $H^1(X,Y)$ we denote, respectively, the Lebesgue and Sobolev spaces of functions defined on $X$ and taking values in $Y;$ 

--- A similar notation $C_0^\infty(X, Y)$ is used for the set of infinitely smooth functions with compact support in $X;$ 

--- We often denote by $C$ constants in various estimates, whenever the precise value of the constant is unimportant;

--- Summation over the set $\{1,2,3\}$ is assumed for repeated Latin indices and over the set $\{1,2\}$ for repeated Greek indices; 

--- Vectors are always treated as columns, so that $(x_1, x_2)^\top$ and $(x_1, x_2, x_3)^\top$ stand for the vectors with components $x_1,$ $x_2$ and $x_1,$ $x_2,$ $x_3,$ respectively; 

--- We denote by $\nabla=(\partial_1, \partial_2, \partial_3)$ and  $\widetilde{\nabla}=(\partial_1, \partial_2)$ the gradients in three and two dimensions, respectively, and by $\widetilde{\nabla}^2$ the second gradient in two dimensions, understood as a $(2\times2)$-matrix differential expression; 

--- For $\chi\in[-\pi, \pi)^2,$ the quasiperiodic exponential function $\exp({\rm i}\chi\cdot y),$ $y\in[0,1)^2,$ is denoted by $e_\chi.$ It is often treated as a function of $(y, x_3)\in[0,1)^2\times(-1/2, 1/2),$ constant with respect to the variable $x_3.$

\section{Problem formulation and main results}
\label{formulation}


Denote $Q_{\rm r}:=[0,1)^2\subset{\mathbb R}^2$ and 
consider a function on $Q_{\rm r}$ with values in the space of fourth-order tensors, {\it i.e.}
\[
A(y)=\bigl\{A_{ijkl}(y)\bigr\}_{i, j,k,l=1}^3,\quad y\in Q_{\rm r}.
\] 
We assume that $A$ is measurable, symmetric, bounded, and uniformly positive definite: for a.e. $y\in Q_{\rm r}$
\[
A_{ijkl}(y)=A_{klij}(y)=A_{ijlk}(y),\quad i,j,k,l=1,2,3,
\]
\begin{align*}
&\nu^{-1}\xi:\xi\ge A(y)\xi:\xi\ge\nu\xi:\xi,\quad\nu>0,\qquad \forall\xi\in{\mathbb R}^{3\times3}\ {\rm such that}\  \xi=\xi^\top,\\[0.35em]
&\bigl(A(y)\xi\bigr)_{ij}:=A_{ijkl}(y)\xi_{kl},\quad i, j=1,2,3,
\end{align*}
where $\nu$ does not depend on $y.$ In what follows we also assume that $A$ is extended to the whole of ${\mathbb R}^2$ by $Q_{\rm r}$-periodicity.  For all $h>0$ we denote $\Pi^h:={\mathbb R}^2\times (-h/2, h/2),$ and for each $\varepsilon>0,$ $h>0$ suppose that the tensor of elastic moduli at any point
$(x_1, x_2, x_3)\in\Pi^h$ is given by $A(x_1/\varepsilon, x_2/\varepsilon).$ 

For given ``body-force densities'' $F^h\in L^2(\Pi^h, {\mathbb R}^3),$ 
we  
study the behaviour of the following family of resolvent problems for $U^{h,\varepsilon}\in H^1(\Pi^h, {\mathbb R}^3)$ as $\varepsilon\to0,$ $h\to0:$
\begin{equation}
\int_{\Pi^h} A\biggl(\frac{x_1}{\varepsilon}, \frac{x_2}{\varepsilon}\biggr) \,{\rm sym} \nabla U^{h,\varepsilon}:{\rm sym}\nabla\Phi+\int_{\Pi^h}U^{h,\varepsilon}\cdot\Phi=\int_{\Pi^h}F^h\cdot\Phi\quad\quad\forall\Phi\in H^1(\Pi^h, {\mathbb R}^3),
\label{original_identity}
\end{equation}



We shall assume throughout the paper that $h=\varepsilon.$ Our analysis is valid if we assume that the thickness $h$ is a function of the period $\varepsilon$ in the asymptotic regime $h=h(\varepsilon) \sim \varepsilon$, {\it i.e.} in the case when the quantity $h/\varepsilon$ is bounded above and below by $\varepsilon$-independent positive constants. For adapting the analysis to this more general case, see Section \ref{rgamma}. 
For each $\varepsilon>0$ we denote by  ${\mathcal A}^\varepsilon$ the operator in $L^2(\Pi^h , {\mathbb R}^3)$ defined by the bilinear form ({\it cf.} (\ref{original_identity}))
\begin{equation*}
\int_{\Pi^h } A\biggl(\frac{x_1}{\varepsilon}, \frac{x_2}{\varepsilon}\biggr) \,{\rm sym} \nabla U:{\rm sym}\nabla\Phi,\qquad  U,\Phi\in H^1(\Pi^h , {\mathbb R}^3).
\end{equation*}
Next, denote $I:=(-1/2, 1/2),$  $Q:=Q_{\rm r}\times I$ and consider the following symmetric tensor $\mathcal{L}$ of order 4:
\begin{equation}
\begin{aligned} 
\mathcal{L}(M_1,M_2) : (M_1,M_2):= \inf_{\psi \in H^1_{\#}(Q,\mathbb{R}^3)}\int_Q A\bigl({\mathfrak I}( M_1-x_3M_2)&+{\rm sym}\nabla \psi\bigr):\bigl( {\mathfrak I}(M_1-x_3 M_2)+{\rm sym}\nabla \psi\bigr),\\[0.1em]
&M_1,M_2 \in \mathbb{R}^{2 \times 2},
\end{aligned}
\label{L_definition}
\end{equation}
where ${\mathfrak I}$ is the mapping
\begin{equation}
{\mathfrak I}: \mathbb{R}^{2 \times 2}\ni M\mapsto \left(\begin{matrix} M &\begin{matrix}0\\0\end{matrix}\\ \begin{matrix}0&0\end{matrix}&0\end{matrix}\right)\in\mathbb{R}^{3\times3}.
\label{iota_def}
\end{equation}
We also employ the following tensors $\mathcal{L}^1$, $\mathcal{L}^2$ on $\mathbb{R}^{2 \times 2},$ representing the out-of-plane and in-plane ``truncations" of the tensor ${\mathcal L}:$
$$ 
\mathcal{L}^1 M:M :=\mathcal{L} (0,M):(0,M), \qquad  \mathcal{L}^2 M:M :=\mathcal{L} (M,0):(M,0),\qquad M\in\mathbb{R}^{2 \times 2},
$$
and the differential operators
\begin{equation}
\mathcal{A}^{{\rm hom}, 1}:=
\bigl(\widetilde{\nabla}^2\bigr)^{*} \mathcal{L}^1\widetilde{\nabla}^2,\qquad 
\mathcal{A}^{{\rm hom}, 2}:=
\bigl({\rm sym}\widetilde{\nabla}\bigr)^{*} \mathcal{L}^2\,{\rm sym}\widetilde{\nabla},\qquad 
\mathcal{A}^{\rm hom}:=\bigl({\rm sym}\widetilde{\nabla}, \varepsilon\widetilde{\nabla}^2\bigr)^{*} \mathcal{L}\bigl({\rm sym}\widetilde{\nabla}, \varepsilon\widetilde{\nabla}^2\bigr),
\label{Ahoms}
\end{equation} 
defined in the standard way ({\it e.g.} via appropriate bilinear forms \cite{Davies}) in $L^2({\mathbb R}^2, {\mathbb R}),$ $L^2({\mathbb R}^2, {\mathbb R}^2),$ and $L^2({\mathbb R}^2, {\mathbb R}^3),$ respectively.  In (\ref{Ahoms}) we use the shorthand $({\rm sym}\widetilde{\nabla}, \varepsilon\widetilde{\nabla}^2)$ for the pair of $(2\times 2)$-matrix differential expressions
\[
(u,v)\mapsto\bigl({\rm sym}\widetilde{\nabla}u, \varepsilon\widetilde{\nabla}^2v\bigr),\qquad u=\bigl(u_1(x_1, x_2), u_2(x_1, x_2)\bigr),\quad v=v(x_1, x_2).
\]

\begin{remark}
\label{different_scalings}
It is not unusual in the theory of plates that the out-of-plane component of the displacement is scaled in relation to the in-plane components (see {\it e.g.} \cite{ciarlet}), hence the presence of the parameter $\varepsilon$ in the expression (\ref{Ahoms}) for ${\mathcal A}^{\rm hom}.$ It is well known that in this context one also has a ``separation" of the asymptotic equations of state in the thin-plate limit \cite[Theorem 1.10-1]{ciarlet}: one equation describes the out-of-plane component of the leading-order displacement and a part of its horizontal components, while the other describes the remaining part of 
the horizontal components. In the case when material properties oscillate in-plane or vary out-of-plane, this kind of separation happens only under an additional symmetry assumption on the elasticity tensor.
\end{remark}

Before stating the main result, we recall once again the existing literature on simultaneous homogenization and dimension reduction in linear elasticity: in \cite{Caill}, bounded linearly elastic plates were studied, under the assumption of periodic oscillations in material properties and $h=\varepsilon$; in \cite{Dam}, a model of a linearly elastic plate was derived without periodicity assumption and invoking additional symmetries of the material, by the $H$-convergence approach (the assumption of material symmetries is removed in \cite{BV}, which uses 
an original approach in the spirit of $\Gamma$-convergence and additionally proves locality of the $\Gamma$-closure, i.e., the fact that every energy density obtained by non-periodic homogenization can be approximated pointwise by energy densities obtained by periodic homogenization); in \cite{Panasenko_book} the full asymptotic expansion with error estimates (as well as 
a boundary layer analysis) for bounded plates and rods is given, although the estimates are not written in the operator norm (i.e. the  constants involved depend on the problem data.) 

For $k=1,2,3,$ denote by ${\mathcal F}_\varepsilon$ the 
  Floquet transform\footnote{The transform (\ref{Floquet_scalar}) is a bounded extension of the mapping defined on $C_0^\infty(\Pi, {\mathbb C}^k)$ by the same formula, see Section \ref{strategynovelties}, where the transform defined by (\ref{Floquet_scalar}) is extended to $L^2(\Pi, {\mathbb C}^k).$}
 \begin{equation}
L^2(\Pi, {\mathbb C}^k)\ni U\mapsto\frac{\varepsilon^2}{2\pi}\sum_{n\in\mathbb{Z}^{2}}\overline{e_\chi(n)}
U\bigl(\varepsilon(y+n), x_3\bigr),\quad (y, x_3)\in Q,
\quad\chi\in Q'_{\rm r}:=[-\pi, \pi)^2.
\label{Floquet_scalar}
 \end{equation}
Furthermore, we define the operators $P_j:L^2(\Pi^h , \mathbb{R}^3)\to L^2(\Pi^h , \mathbb{R})$ of orthogonal projection on the $j$-th coordinate axis, $j=1,2,3,$ as well as the ``smoothing" operators for momenta and forces:

\begin{equation*}
S_\alpha F={\mathcal F}_\varepsilon^*\biggl({\rm i}e_{\chi} \chi_{\alpha}\int_Q \overline{e_{\chi}} x_3 {\mathcal  F}_\varepsilon F\biggr),\quad\alpha=1,2,\qquad
SF:={\mathcal F}_\varepsilon^*\biggl(e_{\chi}\int_Q\overline{e_{\chi}}{\mathcal F}_\varepsilon F\biggr),\qquad F\in L^2(\Pi , {\mathbb R}).
\end{equation*}

The appearance of smoothing operators is standard in homogenisation, as they guarantee that the corrector is in an appropriate function space \cite{BirmanSuslina_corrector}, \cite{ZhikovPastukhova}.  
In some situations it is possible to eliminate them
 from the final results, provided the relevant correctors satisfy certain regularity assumptions. In our setting the situation turns out to be more complex, due to the existence of a bending invariant subspace, where the approximating operator is of fourth order. Although in the end it is possible to remove 
 $S_\alpha,$ $S$ from some of the final estimates, we refrained from doing so, since it cannot be done for all estimates at once. 

Finally, we introduce the 
operators 
$R^h: L^2(\Pi^h,\mathbb{R}^k)\to
L^2(\Pi, \mathbb{R}^k),$ $k=1, 2, 3,$ 
defined on $u\in L^2(\Pi^h,\mathbb{R}^k)$ for each $h$ by the formula 
\[
(R^hu)(x_1, x_2, x_3)=\sqrt{h}u(x_1, x_2, hx_3). 
\]

 We shall omit the label $k$ in the notation for these operators, as the value of $k$ in a given expression will always be clear from the context.
As is common in the analysis of thin structures, we will express the error on the ``canonical" domain $\Pi$ rather than on physical domain $\Pi^h.$ One reason for presenting the results in this way
is that it allows us to scale the displacement gradients ``isotropically", according to the same scaling rule in the in-plane and out-of-plane directions.
 
\begin{theorem}
\label{main_result_first} 
Suppose that
$\gamma\ge-2,$ $\delta\ge0$ and the following ``planar symmetry" conditions on the elasticity tensor $A$ hold:
\begin{equation}
\label{symmetry}
A_{\alpha_1\alpha_2\alpha_33}=0,\quad\quad 
A_{\alpha_1333}=0,\quad\quad\quad
\alpha_j=1,2,\quad j=1,2,3,
\end{equation}
and that the body-force density $F^h$ in (\ref{original_identity})
has the form 
\[
F^h=\bigl(\varepsilon^{-\delta}\widecheck{F\,}\!\!_1,\varepsilon^{-\delta}\widecheck{F\,}\!\!_2, \widehat{F}_3\bigr)^\top,
\] where the first two components are odd and the third component is even in the variable across the plate.

There exists $C=C(A)>0,$ independent of $F^h,$ such that for all $\varepsilon>0$ the following estimates hold:
\begin{equation}
 \begin{aligned}
 \Bigl\|R^h P_{\alpha}&\left(\varepsilon^{-\gamma} \mathcal{A}^{\varepsilon}+I\right)^{-1}F^h\\[0.5em]
 &+\Re\Bigl\{\varepsilon x_3\partial_{\alpha}\left(  \varepsilon^{-\gamma+2}\mathcal{A}^{{\rm hom}, 1} +I \right)^{-1}\left(SR^h\widehat{F}_3+ \varepsilon^{-\delta}{S}_1R^h\widecheck{F\,}\!\!_1+\varepsilon^{-\delta}{S}_2R^h\widecheck{F\,}\!\!_2\right)\Bigr\}\Bigr\|_{L^2(\Pi)} \\[0.5em] 
 &  \hspace{+10ex}  \leq C \varepsilon^{(\gamma+2)/2} \max\bigl\{\varepsilon^{(\gamma+2)/4-\delta},1\bigr\}\bigl\|R^h(\widecheck{F}_1,\widecheck{F}_2, \widehat{F}_3)^\top\bigr\|_{L^2(\Pi,{\mathbb R}^3)} , \quad  \alpha=1,2,\\[1.1em]
 \Bigl\|R^hP_3&\left(\varepsilon^{-\gamma} \mathcal{A}^{\varepsilon}+I\right)^{-1}F^h
 -\Re\Bigl\{\left(  \varepsilon^{-\gamma+2}\mathcal{A}^{{\rm hom}, 1} +I \right)^{-1} \left(SR^h\widehat{F}_3+\varepsilon^{-\delta}S_1R^h\widecheck{F\,}\!\!_1+\varepsilon^{-\delta}S_2 R^h\widecheck{F\,}\!\!_2\right)\Bigr\}\Bigr\|_{L^2(\Pi)} \\[0.5em] 
 & \hspace{+10ex}  \leq C \varepsilon^{(\gamma+2)/4} \max\bigl\{\varepsilon^{(\gamma+2)/4-\delta},1\bigr\}\bigl\|R^h(\widecheck{F\,}\!\!_1,\widecheck{F\,}\!\!_2, \widehat{F}_3)^\top\bigr\|_{L^2(\Pi,{\mathbb R}^3)}. 
 \end{aligned}
\label{first_subspace_est}
\end{equation}

\end{theorem}

A simple analysis of the exponents in the estimates (\ref{first_subspace_est}) shows that the approximation error in (\ref{first_subspace_est}) goes to zero as $\varepsilon\to0$ as long as $\gamma+2>2\delta.$

\begin{remark}
1. In the case $\gamma=2,$ $\delta=1,$ the formulae the approximating fields in (\ref{first_subspace_est}) represent the well-known Kirchhoff-Love ansatz of the theory of plates, see {\it e.g.} \cite{ciarlet}. In particular, they reflect the different scalings of in-plane and out-of-plane components with respect to the parameter $\varepsilon$, see also Remark \ref{different_scalings}. Also, the dependence of the approximating deformation on the moments of in-plane forces is standard, see \cite[Theorem 1.10-1]{ciarlet}.

2. In Theorem \ref{main_result_first}, we establish operator-norm asymptotics for a variety of scaling exponents 
$\gamma,$ not restricting ourselves to the case $\gamma=2,$ $\delta=1$ discussed in \cite{ciarlet} in the context of strong convergence (see also \cite{Miara} showing that the traditional power series asymptotics dictates the choice 
$\gamma=2$). This flexibility in the choice of scaling of the elastic energy (equivalently, the choice of scaling of the spectrum of the original ``unscaled" operator), combined with the fact that the error estimates are given in the operator-norm topology, provides useful quantitative information for evolution problems for elastic plates, under a wider range of time scalings (corresponding to the choice of $\gamma$) that have been considered so far. The scaling exponent of the 
force density translates directly to a similar scaling of the force density in the associated evolution problems. We refer the reader to the preprint \cite{BuzChVZ}, where the scalings $\gamma=0$, $\gamma=2$ and their consequences for the evolution equation are analysed in the context of bounded plates with high-contrast inclusions.
\end{remark}

\begin{theorem}
\label{main_result_second} 
Suppose that 
the tensor $A$ has the symmetry properties (\ref{symmetry})
and that in (\ref{original_identity}) the first two components of 
$F^h=\bigl(\widehat{F}_1,\widehat{F}_2, \widecheck{F\,}\!\!_3\bigr)$ 
are even and the third component is odd in $x_3\in hI.$
Then for each $\gamma\ge-2$ there exists $C=C(A)>0,$ independent of $F^h,$ such that for all $\varepsilon>0$ the following estimate holds:

\begin{equation}
\left\| R^h\left(\varepsilon^{-\gamma} \mathcal{A}^{\varepsilon}+I\right)^{-1}F^h
-\left(\begin{array}{c}\left(\varepsilon^{-\gamma}\mathcal{A}^{{\rm hom}, 2} +I \right)^{-1}SR^h\bigl(\widehat{F}_1, \widehat{F}_2 \bigr)^\top\\[0.4em]0\end{array}\right)\right\|_{L^2(\Pi , \mathbb{R}^3)}
\leq C \varepsilon^{(\gamma+2)/2} \bigl\| R^hF^h
\bigr\|_{L^2(\Pi ,\mathbb{R}^3)}.
\label{second_subspace_est}
\end{equation}

\end{theorem}

Notice that the approximation error in (\ref{second_subspace_est}) is small as $\varepsilon\to0$ as long as $\gamma>-2.$ In particular, for $\gamma=0$ we obtain a version of the norm-resolvent estimates \cite{BirmanSuslina} for the two-dimensional ``zero-thickness'' plate. The operator ${\mathcal A}^{\rm hom, 2}$ in this case is related to the two-dimensional version of the``spectral germ'' of \cite{BirmanSuslina}. However, the estimates (\ref{first_subspace_est}) in the ``bending subspace" consisting of functions with the first two components odd and the third component even do not have an analogue in the existing literature and indeed require a more delicate analysis than the estimates (\ref{second_subspace_est}), which we carry out below under the umbrella of the general approach we introduce in this paper.

In the next statement we do not impose the symmetry conditions (\ref{symmetry}) on $A.$
\begin{theorem}
\label{main_result_general}
For each $\gamma\ge-2$ there exists $C=C(A)>0,$ independent of $F^h=(F_1, F_2, F_3),$ such that for all $\varepsilon>0$ the following estimates hold:
\begin{equation}
\begin{aligned}
&\Bigl\|R^h P_{\alpha}\left(\varepsilon^{-\gamma} \mathcal{A}^{\varepsilon}+I\right)^{-1}F^h\\[0.4em]
&-\Re\Bigl\{\left(P_{\alpha}-\varepsilon x_3\partial_{\alpha}P_3\right)\left(  \varepsilon^{-\gamma}\mathcal{A}^{\rm hom} +I \right)^{-1}\bigl(SR^hF_1, S R^h F_2, S R^h F_3+S_1 R^h F_1+S_2 R^h F_2  \bigr)^\top\Bigr\}\Bigr\|_{L^2(\Pi )}\\[0.7em]
& \hspace{+15ex}\leq C \varepsilon^{(\gamma+2)/2} \bigl\|R^h F^h\bigr\|_{L^2(\Pi,{\mathbb R}^3)}, \quad \alpha=1,2, \\[0.7em]
&\Bigl\| R^h P_{3}\left(\varepsilon^{-\gamma} \mathcal{A}^{\varepsilon}+I\right)^{-1}F^h\\[0.4em]
&-\Re\Bigl\{P_3\left(\varepsilon^{-\gamma}\mathcal{A}^{\rm hom} +I \right)^{-1}\bigl(S R^h F_1, S R^h F_2, S R^h F_3+S_1 R^hF_1+S_2 R^hF_2\bigr)^\top\Bigr\}\Bigr\|_{L^2(\Pi)}
\\[0.7em]
& \hspace{+15ex}
\leq C \varepsilon^{(\gamma+2)/4} \bigl\|R^h F^h\bigr\|_{L^2(\Pi,{\mathbb R}^3)}. 
\end{aligned}
\label{general_est}
\end{equation}

\end{theorem} 

When $A$ satisfies the condition (\ref{symmetry}), we are able to give a more precise estimate in each invariant subspace and include the dependence on $\delta,$ which is important for the plate theory. In the case of bounded plates one usually takes $\delta=1$, see Remark \ref{expl3} for the meaning of $\gamma$, $\delta.$

\section{Strategy for the proofs and the novelty of the approach} 
\label{strategynovelties} 
In this chapter we introduce the Floquet transform, briefly discuss our strategy of the proofs of the main results stated in the previous section, and describe in what way our approach is new. 

The Floquet transform allows one to represent the elasticity operator (more precisely, its resolvent) in terms of simpler operators, which can then be analysed. Namely, we rewrite (\ref{original_identity}) on the scaled domain $\Pi$
with the solution $\widetilde{U}\in H^1(\Pi, {\mathbb R}^3),$ as follows:
\begin{align}
\int_\Pi A\biggl(\dfrac{x_1}{\varepsilon}, \dfrac{x_2}{\varepsilon}\biggr) \,{\rm sym}\nabla^\varepsilon \widetilde{U}:{\rm sym}\nabla^\varepsilon \Psi+\int_\Pi\widetilde{U}\cdot\Psi
&=\int_\Pi\widetilde{F}^\varepsilon\cdot\Psi\quad\quad\forall\Psi\in H^1(\Pi, {\mathbb R}^3),\label{original2}
\end{align}
where 
for each $\varepsilon$ the force density $\widetilde{F}^\varepsilon$ is given by  $\widetilde{F}^\varepsilon(x)=F^\varepsilon(x_1, x_2, \varepsilon x_3),$ the solution 
$\widetilde{U}$ is related to the solution $U^{h, \varepsilon}$ of (\ref{original_identity}) by a similar formula $\widetilde{U}(x)=U^{h, \varepsilon}(x_1, x_2, \varepsilon x_3),$
and $\nabla^\varepsilon=(\partial_{1}, \partial_{2}, \varepsilon^{-1}\partial_{3})$ is the corresponding rescaling of the gradient.

For $\chi\in Q'_{\rm r},$ we define the space $H^1_\chi(Q, {\mathbb C}^3)$ as the closure in $H^1(Q, {\mathbb C}^3)$ of the set of smooth functions $u$ on $\overline{\Pi}$ that are $\chi$-quasiperiodic with respect to $y_1,$ $y_2,$ {\it i.e.} 
\[
u(y+l_\alpha, x_3)=u(y, x_3)\exp({\rm i}\chi_\alpha),\ \ \ \ \ l_\alpha:=(\delta_{\alpha j})_{j=1}^3,\ \ \ \ \ \alpha=1,2,\qquad {y\in{\mathbb R}^2},\quad x_3\in I,
\]
where $\delta_{ij},$ $i,j=1,2,3,$ is the Kronecker delta. It is easily seen that $\chi$-quasiperiodic functions on $Q,$ are products of $Q_{\rm r}$-periodic functions and the exponent $e_\chi$ (see the end of Section \ref{intro_section}).

The Floquet transform ${\mathcal F}_\varepsilon$ defined by (\ref{Floquet_scalar}) has the following properties: 
\begin{enumerate} 
\item  It is a unitary transform on $L^2(\Pi, {\mathbb C}^3)$ with values in $L^2(Q \times Q_r',\mathbb{C}^3)$; 
\item  For fixed $\chi \in Q_r'$, the function ${\mathcal F}_\varepsilon \widetilde{U}(\cdot,\cdot, \chi)$ is quasiperiodic with respect to $y_1,y_2$;
\item  
For every $\widetilde{U} \in H^1 (\Pi,{\mathbb C}^3),$ one has
$$({\mathcal F}_\varepsilon) (\partial_{\alpha} \widetilde{U})(y,x_3,\chi)=\varepsilon^{-1}  \partial_{y_{\alpha}} ({\mathcal F}_\varepsilon \widetilde{U})(y,x_3,\chi), \ \alpha=1,2, \qquad  ({\mathcal F}_\varepsilon) (\partial_{x_3} \widetilde{U})(y,x_3,\chi)= \partial_{x_3} ({\mathcal F}_\varepsilon \widetilde{U})(y,x_3,\chi).$$  
\item   
The $\varepsilon$-problems can  be decomposed (after transformation $\mathcal{F}_{\varepsilon}$), using von Neumann direct integral. For every $\gamma \in \mathbb{R},$ one has
\begin{equation}
R^h\left(\varepsilon^{-\gamma} \mathcal{A}^{\varepsilon}+I\right)^{-1}={\mathcal F}_\varepsilon^*\left(\oplus \int_{Q'_{\rm r}}\left(\varepsilon^{-\gamma-2} \mathcal{A}_\chi+I\right)^{-1}d\chi\right) {\mathcal F}_\varepsilon R^h,
\label{vonNeumann1}
\end{equation}
where the operator $\mathcal{A}_{\chi}$ is defined 
in $L^2(Q,\mathbb{C}^3)$ by 
the sesquilinear form ({\it cf.} (\ref{general_problemrevision}))
\begin{equation}
{{\mathfrak b}_\chi}(u,\varphi):=\int_QA\,\,{\rm sym}\nabla u:\overline{{\rm sym}\nabla\varphi},\qquad u, \varphi\in H^1_{\chi} (Q,\mathbb{C}^3).
\label{bform}
\end{equation}
\end{enumerate}

Applying the transform ${\mathcal F}_\varepsilon$ to (\ref{original2}) yields a family of problems on 
$Q$
parametrised by $\varepsilon$ and ``quasimomentum'' $\chi\in Q'_{\rm r}:$
\begin{align}
\frac{1}{\varepsilon^2}\int_Q  A \ {\rm sym} \nabla u:\overline{ {\rm sym} \nabla\varphi}+ \int_Q  u\cdot\overline{\varphi}
=\int_Qf\cdot\overline{\varphi}\quad\quad\forall\varphi\in H^1_\chi(Q, {\mathbb C}^3).
\label{general_problemrevision}
\end{align}
The right-hand side $f\in L^2(\Pi, {\mathbb C}^3)$ in (\ref{general_problemrevision}) and the solution $u\in H^1_\chi(Q, {\mathbb C}^3)$ are the Floquet transforms of the right-hand side $\widetilde{F}^\varepsilon$ and the solution $\widetilde{U}$ in (\ref{original2}), {\it i.e.} $f={\mathcal F}_\varepsilon\widetilde{F}^\varepsilon,$ $u={\mathcal F}_\varepsilon\widetilde{U},$ where for brevity we do not indicate the dependence of $f,$ $u$ on the parameter $\varepsilon.$

Next we will describe our approach to proving the results formulated in Section \ref{formulation}. First, in Sections \ref{Korn_section}, \ref{structure} we show that the smallest eigenvalue of the operator $\mathcal{A}_{\chi}$ is of order $|\chi|^4$, the next two eigenvalues are of order $|\chi|^2,$ and the remainin eigenvalues are of order one, see Lemma \ref{revlemma}. The discrepancy in the order of the smallest eigenvalues is the consequence of the anisotropy of the space $H^1_\chi(Q, {\mathbb C}^3)$ (as the periodicity conditions are imposed only in $y$ variable). The same observation shows why the plate problem cannot be tackled by the techniques developed in \cite{BirmanSuslina}, where the regularity of the so called "spectral germ" implies that the smallest eigenvalues of the corresponding $\chi$-problem are all of order $|\chi|^2$, see Remark \ref{revrem} below. In the present paper we develop an original approach for obtaining operator norm-resolvent estimates, by invikong $\chi$-dependent asymptotic expansions (see Section \ref{casegeneltens}, where our approach is compared to that of \cite{ChCoARMA}). Our technique can also be used for the problems analysed in \cite{BirmanSuslina} in a relatively simple way, see part 2 of Remark \ref{f3_omitted} below. We also point out that the presence of an eigenvalue of order $|\chi|^4$ in addition to eigenvalues of order $|\chi|^2$ makes the computations more difficult (see Remark \ref{rem53} below) and requires that the final estimates are given with respect to the parameters $\gamma$ and $\delta$, which does not fit into the framework of \cite{BirmanSuslina} (see Remark \ref{expl3}).  

Once the problem is decomposed into simpler problems using the formula \eqref{vonNeumann1}, it is important to notice that the eigenspaces associated with order one eigenvalues of the operator $\mathcal{A}_{\chi}$ are not important, if one is interested in the $\varepsilon^2$-precision of the resolvent. This is a consequence of the scaling (see also the first term in \eqref{general_problemrevision}). Thus, for every $\chi$ we only need to approximate the space spanned by the first three smallest eigenvalues of 
 $\mathcal{A}_{\chi}$ (see also how  the estimates obtained in Section \ref{asymptotic_proc} are used in Section \ref{normresolvent} to estimate the resolvent in the operator norm). Also, the limit homogenised operator (looked at for every value of the parameter $\chi$) acts on a finite-dimensional space, see  \eqref{sinisa1000revision},\eqref{sinisa3*second}, \eqref{korekcija11111} below. This is precisely what makes the Floquet transform useful in the analysis of these problems.  It is also important to notice that not all values of $\chi$ play the same role in the estimates of the resolvent. Namely, again due to scaling, only the ``fibres'' $\mathcal{A}_{\chi}$ for $\chi$ in a neighbourhood of origin contribute to the resolvent estimate as for $|\chi|>c$, for some $c>0,$ the smallest eigenvalue of $\mathcal{A}_{\chi}$ is of order one, see also Section \ref{normresolvent}. 

In order to estimate the resolvent of $\mathcal{A}_{\chi}$ we use asymptotic procedure in Section \ref{asymptotic_proc}. However, it would be very cumbersome (and unnatural) to develop this asymptotics directly in a uniform way with respect to the parameter $\varepsilon$. Instead, the asymptotics is approached by scaling the operator $\mathcal{A}_{\chi}$ with $|\chi|^{-4}$ and then with $|\chi|^{-2}$, see \eqref{jdba10_new}, \eqref{jdba10_newsecond}, as well as \eqref{scalingjedan} and \eqref{scalingdva}. This is natural as a consequence of apriori estimates on the least eigenvalues of $\mathcal{A}_{\chi}$.  As a result, the error itself is then also dependent on $|\chi|,$ which however does not prevent us from obtaining uniform estimates on the resolvent in Section \ref{normresolvent}, by virtue of the fact that the values of the resolvent of $\mathcal{A}_{\chi}$ for different $\chi$ do not have the same importance in the uniform estimate of the resolvent of $\mathcal{A}^{\varepsilon}$. 

In our asymptotic procedure we single out the case when the elasticity tensor satisfies planar symmetries (\ref{symmetry}). 
Although it is not necessary from the technical point of view, this assumption 
makes the asymptotic procedure more straightforward and provides an idea on how to proceed in the general case. Moreover, as already mentioned, under the assumption (\ref{symmetry}) one obtains sharper estimates, as the two invariant subspaces can be treated separately. This, in turn, allows us to express the final result in one of them, namely the ``bending" subspace, with the forces scaled in the way it is usually done in linear plate theory, see \cite{ciarlet}.  

We emphasize the fact that the asymptotic procedure would be more straighforward if all the lowest eigenvalues of the operator $\mathcal{A}_{\chi}$ were of the same order $|\chi|^2$ as in \cite{BirmanSuslina}, and the analysis presented in Section \ref{section_drugaasimptotika} would suffice. 
We also emphasize the fact that we do not require any smoothness of the coefficients of the elasticity tensor (they are only assumed to be in $L^{\infty}$). 

We next summarise the main steps of our approach, which can be used for solving other problems (see also \cite{CDVZ} the remark 3 in Section \ref{concl_sec}): 
\begin{enumerate}
\item Using Korn inequalities for quasiperiodic function in Section \ref{Korn_section}, we obtain the estimates on the smallest eigenvalues for the operator $\mathcal{A}_{\chi}$, see Proposition \ref{propcorrection}; 

\item  By applying an asymptotic development in $|\chi|$ we obtain optimal estimates for the resolvents of the operators $|\chi|^{-2}\mathcal{A}_{\chi}$ and $1/|\chi|^{-4}\mathcal{A}_{\chi}$ with respect to the approximating operator that corresponds to the standard linear Kirchoff plate model. This operator acts on an finite-dimensional subspace of $H^1_{\chi}(Q,\mathbb{C}^3)$. This is implemented in Section \ref{asymptotic_proc}; 

\item Using $\chi$- dependent estimates for the resolvents of the operators  $|\chi|^{-2}\mathcal{A}_{\chi}$ and $|\chi|^{-4} \mathcal{A}_{\chi}$ and the property 4 above, we obtain an estimate for the resolvent of the operator $\varepsilon^{-\gamma} \mathcal{A}^{\varepsilon}$. This is carried out in Section \ref{normresolvent}. The estimates are given in terms of the $L^2 \to L^2$ and $L^2 \to H^1$ norms. 
\end{enumerate} 	
We also  include Section \ref{apriori_sec}, which is not necessary to obtain the resolvent estimates but is helpful for a better understanding of the approximating operator and the asymptotic procedure in Section \ref{asymptotic_proc}, see Remark \ref{expl2} and Section \ref{rgamma}, where the estimates on the constants appearing on the right-hand side of the resolvent estimate are discussed. It may be worth making a technical remark that while in our analysis we chose to use the Floquet version of the transform representing the operator of the original equation (\ref{general_eq}) as a direct integral, see (\ref{vonNeumann1}), it could be similarly carried out in the framework of the Gelfand transform \cite{Gelfand}, with appropriate straightforward modifications of the Korn inequalities of Section \ref{Korn_section} and the asymptotic recurrence procedure of Section \ref{asymptotic_proc}.


\section{Auxiliary results}
\label{aux_sec}

\subsection{Korn inequalities}
\label{Korn_section}
We next establish several Korn-type inequalities, which will inform us in Section \ref{apriori_sec} about the structure of solutions to the problem (\ref{general_problemrevision}).
\begin{lemma}
\label{neN_est}
There exists a constant $C>0$ such that for all $\chi\in Q'_{\rm r}$ and $u=(u_1, u_2, u_3)^\top\in H^1_\chi(Q, {\mathbb C}^3),$ 
there are $a_1, a_2, c_1, c_2, c_3 \in \mathbb{C}$ satisfying the estimates
\begin{align} 
\label{prva} \|u_\alpha- a_\alpha x_3-c_\alpha \|_{H^1(Q, {\mathbb C})} &\leq  C\bigl\| {\rm sym}\nabla u\bigr\|_{L^2(Q, {\mathbb C}^{3\times3})},\qquad \alpha=1,2,\\[0.4em]
\label{treca} \|u_3+ a_1y_1+a_2y_2-c_3 \|_{H^1(Q, {\mathbb C})} &\leq  C\| {\rm sym}\nabla u \|_{L^2(Q, {\mathbb C}^{3\times3})}, \\[0.6em]
\label{cetvrta}\max\bigl\{ |a_1|, |a_2|, |c_1|, |c_2|\bigr\} &\leq  C\vert\chi\vert^{-1}
\bigl\| {\rm sym}\nabla u\bigr\|_{L^2(Q, {\mathbb C}^{3\times3})}, \\[0.5em]
\label{peta}| c_3| &\leq C 
\vert\chi\vert^{-2}\bigl\| {\rm sym}\nabla u\bigr\|_{L^2(Q, {\mathbb C}^{3\times3})}, \\[0.5em]
\label{sesta}\bigl|\left(\exp({\rm i}\chi_\alpha)-1\right) c_3+a_\alpha\bigr| &\leq  C\bigl\| {\rm sym}\nabla u\bigr\|_{L^2(Q, {\mathbb C}^{3\times3})},\quad\alpha=1,2.
\end{align} 
\end{lemma}
\begin{proof}
The estimates \eqref{prva}--\eqref{treca} are obtained by using Korn's inequality, as follows. The standard ``second'' Korn inequality, see {\it e.g.} \cite[Theorem 2.5]{OShY}, gives 
\begin{eqnarray} 
	\|u_1- a_1x_3-dy_2-c_1 \|_{H^1(Q, {\mathbb C})} &\leq & C\bigl\| \textrm{sym}\nabla u\bigr\|_{L^2(Q, {\mathbb C}^{3\times3})}, \label{one}\\[0.5em]
	\|u_2- a_2x_3+dy_1-c_2 \|_{H^1(Q, {\mathbb C})} &\leq & C\bigl\| \textrm{sym}\nabla u\bigr\|_{L^2(Q, {\mathbb C}^{3\times3})}, \nonumber\\[0.5em]
	\|u_3+ a_1y_1+a_2y_2- c_3 \|_{H^1(Q, {\mathbb C})} &\leq & C\bigl\| \textrm{sym}\nabla u\bigr\|_{L^2(Q, {\mathbb C}^{3\times3})},\label{three} 
\end{eqnarray}
where 
\begin{equation*}
\begin{aligned}
a_\alpha&:=\int_Q(\partial_3 u_\alpha-\partial_\alpha u_3), \quad \alpha=1,2,
 \qquad d:=\int_Q(\partial_2 u_1-\partial_1 u_2),\\[0.4em] 
c_1&:=\int_Qu_1-d,\quad c_2:=\int_Qu_2+d,\quad c_3=\int_Qu_3+a_1+a_2.
\end{aligned}   
\end{equation*}           
Furthermore, by a variation of the proof of the ``first'' Korn inequality, see 
{\it e.g.} \cite[Theorem 2.1]{OShY}, the following estimate is valid for quasiperiodic functions on $Q_{\rm r}:$
\begin{equation} 
\label{kornquasi} 
\bigl\| \nabla u\bigr\|_{L^2(Q_{\rm r}, {\mathbb C}^{2\times 2})} \leq C\bigl\| {\rm sym} \nabla u\bigr\|_{L^2(Q_{\rm r}, {\mathbb C}^{2\times2})},\quad u\in H^1_\chi(Q_{\rm r}, {\mathbb C}^2),\quad \chi\in Q'_{\rm r},
\end{equation}
where the constant $C$ is independent of the quasimomentum $\chi,$ and $H^1_\chi(Q_{\rm r}, {\mathbb C}^2)$ is defined similarly to $H^1_\chi(Q, {\mathbb C}^3).$ Noting that 
\[
d= \int_{Q_{\rm r}}\left\{ \partial_2 \int_I u_1d\,x_3-\partial_1 \int_I u_2\,dx_3 \right\}dy_1dy_2,
\]
and applying \eqref{kornquasi} to the vector  
\[
\overline u:=\left(\int_I u_1\,dx_3,\int_I u_2\,dx_3 \right)^\top: Q_{\rm r} \to \mathbb{R}^2,
\]
we infer that 
\[
|d| \leq C\bigl\| {\rm sym} \nabla \overline{u}\bigr\|_{L^2(Q_{\rm r}, {\mathbb C}^{3\times3})} \leq C\bigl\|{\rm sym} \nabla u\bigr\|_{L^2(Q, {\mathbb C}^{3\times 3})}, 
\]
which in combination with (\ref{one})--(\ref{three}) yields (\ref{prva})--(\ref{treca}). Notice that  the same argument proves 
the estimates (\ref{prva})--(\ref{sesta}) with $a_\alpha,$ $\alpha=1,2,$ as defined above, as long as they hold with 
\[ 
a_\alpha=\int_{Q}\partial_3 u_\alpha,\quad \alpha=1,2.
\]

The estimates (\ref{cetvrta})--(\ref{sesta}) are derived by invoking the quasiperodicity in the directions $y_1, y_2,$ as follows. Notice first that from  (\ref{prva})--(\ref{treca}), using the trace inequality and the fact that $u$ is $\chi$-quasiperodic, one has
\begin{align} 
&\int\limits_{\{(y_1, y_2, x_3)\in Q: y_\beta=1\}}\Bigl|\bigl(\exp{({\rm i}\chi_\beta)} -1\bigr) (a_\alpha x_3+c_\alpha)\Bigr|^2dy_2dy_3\leq C\bigl\| {\rm sym} \nabla u\bigr\|^2_{L^2(Q, {\mathbb C}^{3\times3})},\quad \alpha,\beta=1,2,
\label{dvan}\\[0.6em]
&\int\limits_{\{(y_1, y_2, x_3)\in Q: y_1=1\}}\Bigl|\bigl(\exp{({\rm i}\chi_1)} (a_2y_2- c_3)-(a_1+a_2y_2- c_3)\bigr)\Bigr|^2dy_2dy_3\leq  C\bigl\| {\rm sym} \nabla u\bigr\|^2_{L^2(Q,  {\mathbb C}^{3\times3})}, \label{sesn}
\\[0.6em]
&\int\limits_{\{(y_1, y_2, x_3)\in Q: y_2=1\}}\Bigl|\bigl(\exp{({\rm i}\chi_2)} (a_1y_1- c_3)-(a_1y_1+a_2-c_3)\bigr)\Bigr|^2dy_1dy_3\leq  C\bigl\| {\rm sym} \nabla u\bigr\|^2_{L^2(Q,  {\mathbb C}^{3\times3})}.
 \label{sedamn}
\end{align}
Furthermore, there exist constants $C_1,C_2>0$ such that 
 \begin{equation*} 
 C_1|\chi_{\alpha}| \leq\bigl|\exp{({\rm i}\chi_{\alpha})} -1\bigr| \leq C_2 |\chi_{\alpha}|, \quad\quad 
 \alpha=1,2. 
 \end{equation*}
The estimates \eqref{cetvrta}--\eqref{sesta} are now obtained as a direct consequence of \eqref{dvan}--\eqref{sedamn}. 
\end{proof}

\subsection{Structure of the leading-order field}
\label{structure}

Taking into account 
\eqref{peta}, 
we infer that \eqref{sesta}
 are equivalent to the estimates
\begin{equation}
\label{jdba1}
|{\rm i}\chi_\alpha c_3+a_\alpha| \leq C\bigl\| {\rm sym} \nabla u\bigr\|_{L^2(Q, {\mathbb C}^{3\times 3})}, 
\qquad\alpha=1,2,
\end{equation}
In particular, from \eqref{prva}--\eqref{treca}, \eqref{jdba1}
we obtain
\begin{align}
\label{janprva}  \| u_\alpha-(c_\alpha-{\rm i}\chi_\alpha c_3 x_3)e_\chi\|_{H^1(Q, {\mathbb C})}&\leq C\bigl\| {\rm sym} \nabla u\bigr\|_{L^2(Q, {\mathbb C}^{3\times3})},\qquad\alpha=1,2,\\[0.65em]
\label{jantreca} \| u_3-c_3e_\chi\|_{H^1(Q, {\mathbb C})}&\leq  C\bigl\| {\rm sym} \nabla u\bigr\|_{L^2(Q, {\mathbb C}^{3\times3})},
\end{align}
where $c_j,$ $j=1,2,3,$ satisfy (\ref{peta})--(\ref{sesta}).

\begin{remark}
Notice that the leading-order term $e_\chi(c_1-{\rm i}\chi_1c_3x_3, c_2-{\rm i}\chi_2c_3x_3, c_3)^{\top}$ for $u$ 
 is the Floquet-transformed representation of the classical linear Kirchhoff-Love ansatz \cite[Theorem 1.10-1]{ciarlet}. 
\end{remark}


Next, notice that for $\chi\in Q'_{\rm r}$ 
\begin{equation}
 \label{antonio2}
\begin{aligned}
&
\bigl\| {\rm sym}\nabla (c_1,c_2,0)^\top e_{\chi}\bigr\|_{L^2(Q, {\mathbb C}^{3\times3})}=\frac{1}{2}\Bigl(|c_1\chi_1+c_2\chi_2|+|\chi| \bigl|(c_1,c_2)^\top\bigr|\Bigr)\le  |\chi| \bigl|(c_1,c_2)^\top\bigr|
\qquad \forall c_1, c_2\in{\mathbb C},\\[0.4em]
& \bigl\| {\rm sym}\nabla (-{\rm i}\chi_1x_3,-{\rm i}\chi_2 x_3, 1)^\top e_{\chi}\bigr\|_{L^2(Q, {\mathbb C}^{3\times3})}=\frac{|\chi|^2}{\sqrt{24}}.
\end{aligned}
\end{equation}	

	It follows from the estimates \eqref{janprva}--\eqref{antonio2}, by examining appropriate Rayleigh quotients (see (\ref{max_min}) below), that the smallest eigenvalue of $\mathcal{A}_{\chi}$ is of order $|\chi|^4,$ there are also two eigenvalues of order $|\chi|^2,$ and the remaining eigenvalues are of order $1$ or higher. 
Indeed, as a consequence of \eqref{cetvrta}, \eqref{peta}, \eqref{janprva}--\eqref{antonio2} as well as coercivity and boundedness of the tensor $A,$ the following proposition holds. 
\begin{proposition}
\label{propcorrection}
There exists $C>0$ 
such that for all $\chi\in Q'_{\rm r}$ the following bounds hold:
 \begin{align} 
 	& \label{bbb0} {\mathfrak b}_{\chi} (u,u)\ge C\nu|\chi|^4 \| u\|^2_{L^2(Q,\mathbb{C}^3)} \quad \forall u \in H^1_{\chi}(Q,\mathbb{C}^3), \\[1.0em]
 	& {\mathfrak b}_{\chi}\left( e_\chi(-{\rm i}\chi_1c_3x_3,-{\rm i}\chi_2 c_3 x_3,c_3)^\top, e_\chi(-{\rm i}\chi_1c_3x_3,-{\rm i}\chi_2 c_3 x_3,c_3)^\top e_{\chi} \right)\leq(24\nu)^{-1}|\chi|^4 |c_3|^2\quad \quad \forall c_3 \in \mathbb{C}, \label{bbb1}\\[1.0em]
        & {\mathfrak b}_{\chi} (u,u)\ge C\nu|\chi|^2 \| u\|^2_{L^2(Q,\mathbb{C}^3)} \quad \forall u \in H^1_{\chi}(Q,\mathbb{C}^3)\quad
        \textrm{such that }\ u\cdot e_\chi(-{\rm i}\chi_1 x_3,-{\rm i}\chi_2 x_3,1)^{\top}=0, \label{bbb15} \\[1.0em]
 &{\mathfrak b}_{\chi}\left(e_\chi(c_1-{\rm i}\chi_1 x_3 c_3,c_2-{\rm i}\chi_2 x_3 c_3,c_3)^\top, e_\chi(c_1-{\rm i}\chi_1 x_3 c_3,c_2-{\rm i}\chi_2 x_3 c_3,c_3)^\top \right)\nonumber
 \\[0.7em] 
&  \hspace{+30ex} \leq \nu^{-1}|\chi|^2\left(\bigl|(c_1,c_2)^\top\bigr|^2
+|\chi|^2|c_3|^2/24 \right) \quad\forall c_1,c_2,c_3 \in \mathbb{C}, \label{bbb3}  \\[1.0em]
 & {\mathfrak b}_{\chi}(u,u)\geq C\|u\|^2_{L^2(Q, \mathbb{C}^3)} \quad \forall u \in  H^1_{\chi}(Q,\mathbb{C}^3) \nonumber \\[0.7em]
& \label{bbb4} \hspace{+5ex} \textrm{ such that }\ u\cdot e_\chi(-{\rm i}\chi_1 x_3,-{\rm i}\chi_2 x_3,1)^{\top}=0,\quad
u \cdot (e_{\chi}, 0, 0)^\top =0,\quad u \cdot (0, e_{\chi}, 0)^\top=0.
  \end{align}	
  \end{proposition}
Arranging the eigenvalues of the operator $\mathcal{A}_{\chi}$ in non-decreasing order, the $k$-th eigenvalue $\lambda_\chi^{(k)}$ is given by (see {\it e.g.} \cite[Section 4.5]{Davies})
\begin{equation}
\lambda_\chi^{(k)}=\inf_{{\mathfrak U}\subset H^1_{\chi} (Q,\mathbb{C}^3),\, \dim {\mathfrak U}=k}\ \ \sup_{u\in {\mathfrak U}\setminus\{0\}}\frac{ {\mathfrak b}_{\chi} (u,u)}{\|u\|^2_{L^2(Q,\mathbb{C}^3)}},\qquad k\in{\mathbb N}.
\label{max_min} 
\end{equation}

\begin{lemma}\label{revlemma} 
There exist $c_1,c_2>0$, independent of $\varepsilon$ such that
\begin{equation} 
c_1 \leq \frac{\lambda^{(1)}_{\chi}}{|\chi|^4}, \frac{\lambda^{(2)}_{\chi}}{|\chi|^2}, \frac{\lambda^{(3)}_{\chi}}{|\chi|^2} \leq c_2, \quad \lambda^{(4)}_{\chi} \geq c_1 \quad \forall \chi \in Q_r'. 
\end{equation} 		
\end{lemma}
\begin{proof} 
By virtue of (\ref{max_min}) with $k=1,$ it  follows from \eqref{bbb0}  that $\lambda_\chi^{(1)} \geq   C\nu|\chi|^4$ and \eqref{bbb1} implies that $\lambda_\chi^{(1)} \leq 2 \nu^{-1}|\chi|^4$. Next, using  \eqref{bbb15} and the fact that every two-dimensional subspace 
of $H^1_{\chi} (Q, \mathbb{C}^3)$ contains a vector orthogonal to $(-{\rm i}\chi_1 x_3,-{\rm i}\chi_2 x_3,1)^{\top}e_{\chi},$ we infer that $\lambda_\chi^{(2)} \geq C\nu|\chi|^2$. Furthermore, combining \eqref{bbb3} and (\ref{max_min}), where $k=3$ and ${\mathfrak U}$ is taken to be the three-dimensional space 
\[
\bigl\{(c_1-{\rm i}\chi_1 x_3 c_3,c_2-{\rm i}\chi_2 x_3 c_3,c_3)^\top e_{\chi},\ c_1, c_2, c_3\in{\mathbb C}\bigr\},
\]
yields $\lambda_\chi^{(3)}\leq 2\nu^{-1}|\chi|^2$. Finally, using \eqref{bbb4} with $k=4$ and the fact that every four-dimensional subspace
of $ H^1_{\chi}(Q,\mathbb{C}^3)$ contains a vector orthogonal to  the three vectors $(-{\rm i}\chi_1 x_3,-{\rm i}\chi_2 x_3,1)^{\top}e_{\chi},$  $(e_\chi, 0, 0)^\top,$ and  $(0, e_\chi, 0)^\top,$ 
we infer that $\lambda_\chi^{(4)}\geq C$. 	
\end{proof}




	
\begin{remark}\label{revrem}  As already explained in Section \ref{strategynovelties}, 	the norm-resolvent approximations of Theorems \ref{main_result_first}, \ref{main_result_second}, \ref{main_result_general}
	are equivalent to estimating the subspaces corresponding to the lowest three eigenvalues: in the special case $\chi=0$ the eigenvalue zero has a three-dimensional invariant subspace, which is perturbed for $\chi \neq 0$ into a one-dimensional subspace with an eigenvalue of order $|\chi|^4$ and two one-dimensional subspaces with eigenvalues of order $|\chi|^2$. 
	This makes the thin-plate problem fundamentally different from the homogenisation problems considered by the approach of Birman and Suslina, see {\it e.g.} \cite{BirmanSuslina}. In particular, the assumption of regularity of the ``spectral germ" in \cite{BirmanSuslina} implies that the finite-dimensional subspace that appears for $\chi \neq 0$ as a perturbation of the subspace corresponding to the eigenvalue zero for $\chi=0$ has all eigenvalues of the same order $|\chi|^2,$ which is clearly false in the case of plates.  
\end{remark} 


\section{A priori estimates for solutions of (\ref{general_problemrevision})}
\label{apriori_sec}



We consider separately the case when the elasticity tensor $A$ is planar-symmetric and the general, not necessarily symmetric, case: in the former, one is able to separate the study of (\ref{original_identity}) into two different problems, due to the fact that one can identify two invariant subspaces for the operator ${\mathcal A}_\chi,$ in which resolvent estimates (see Sections \ref{formulation}, \ref{normresolvent}) can be made more  
straightforward and instructive for the purpose of considering the general situation. 

\begin{remark}
Although we do not assume that the elasticity tensor depends on $x_3,$ one could do so. In the planar-symmetric case one then would have to assume additionally that the moduli are even with respect to $x_3$ (otherwise we would lose the decomposition into two invariant subspaces), while in the general case the dependence on $x_3$ can be arbitrary.  
\end{remark}

 The results of this section are not used in the proof of the error estimates that we carry out in Section \ref{asymptotic_proc}. However, they are significant in that they give us an idea of what we should look for in the asymptotic approximation.
	
	In what follows we assume that $\chi\neq0,$ which does not affect the derivation of the results announced in Section \ref{formulation}.
	
\subsection{
The case of planar-symmetric elasticity tensor}
\label{planar_symmetric}

In this section we work under the assumption (\ref{symmetry}) on the elasticity tensor $A$ in (\ref{original_identity}). We derive estimates 
that provide an informed guess about the asymptotic behaviour of solutions to (\ref{general_problemrevision})
as $\varepsilon\to0$ for each of two 
subspaces $L^2(Q, {\mathbb C}^3)$ invariant with respect to the operator ${\mathcal A}_\chi\!\!:$
the one consisting of vector functions whose first two components are odd and the third component is even in $x_3,$ and its orthogonal complement, whose elements have the first two components even and the third component odd in $x_3$. 



\subsubsection{First (``bending") invariant subspace} 
\label{first_subspace}
It follows from the analysis of Section \ref{structure} that the first invariant subspace described above contains the eigenspace for the eigenvalue $\lambda_\chi^{(1)}$ of ${\mathcal A}_\chi,$ since it contains the vector $e_\chi(-{\rm i}\chi_1x_3,-{\rm i}\chi_2 x_3, 1)^\top$ on which the Rayleigh quotient attains a value of order $|\chi|^4,$ see (\ref{bbb1}). Guided by this observation, we scale with $|\chi|^{-4}$ the operator ${\mathcal A}_\chi$ (equivalently, the form ${{\mathfrak b}_\chi}$ with $\chi$-dependent domain, see (\ref{bform})) 
 and with $|\chi|^{-1}$ the horizontal components of the force density on its right-hand side, so the identity \eqref{general_problemrevision} is replaced with
\begin{equation}
\frac{1}{|\chi|^4}{{\mathfrak b}_\chi}(u, \varphi)
+ \int_Q u\cdot\overline\varphi
=\frac{1}{|\chi|}\int_Q \widecheck{f\,\,}\!\!\!_{\alpha}\overline{\varphi_\alpha}+\int_Q \widehat{ f}_3\overline{\varphi_3}\quad\quad\forall\varphi\in H^1_\chi(Q, {\mathbb C}^3),
\label{jdba10_new}
\end{equation}
where the form ${{\mathfrak b}_\chi}$ is defined by (\ref{bform}). Here, in line with the above,
we assume that the components $f_\alpha=\widecheck{f\,}\!\!_{\alpha},$ $\alpha=1,2,$ are odd in the $x_3$ variable, while $f_3=\widehat{f}_3$ is even. The estimates \eqref{cetvrta}--\eqref{peta}, \eqref{janprva}--\eqref{jantreca} imply that
\begin{equation}
\begin{aligned}
	\bigl\|(u_1, u_2)^\top\bigr\|_{H^1(Q, {\mathbb C}^2)}
	&\leq  C|\chi|^{-1}
	\bigl\| \textrm{sym}\nabla u
	\bigr\|_{L^2(Q, {\mathbb C}^{3\times3})},\\[0.5em]
	\|u_3\|_{H^1(Q, {\mathbb C})} &\leq  C|\chi|^{-2}
	\bigl\| \textrm{sym}\nabla u
	\bigr\|_{L^2(Q, {\mathbb C}^{3\times3})}.
\end{aligned}
\label{revision1}
\end{equation}
Setting $\varphi=u$ in (\ref{jdba10_new}) and applying (\ref{revision1})
to the right-hand side of the resulting equation, we obtain
\begin{equation}
\bigl\| \textrm{sym}\nabla u
\bigr\|_{L^2(Q, {\mathbb C}^{3\times3})} \leq C|\chi|^2 
\bigl\|(\widecheck{f\,}\!\!_{1},\widecheck{f\,}\!\!_{2}, \widehat{f}_3)^\top\bigr\|_{L^2(Q, {\mathbb C}^3)}.
\label{sym_nabla_est}
\end{equation}
Going back to \eqref{revision1},
we thus infer that
\begin{equation} 
\label{revision11}
\bigl\|(u_1, u_2)^\top\bigr\|_{H^1(Q, {\mathbb C}^2)}
\leq
C |\chi| \bigl\|(\widecheck{f\,}\!\!_{1},\widecheck{f\,}\!\!_{2}, \widehat{f}_3)^\top\bigr\|_{L^2(Q, {\mathbb C}^3)},\qquad
\|u_3\|_{H^1(Q, {\mathbb C})} 
\leq
C\bigl\|(\widecheck{f\,}\!\!_{1},\widecheck{f\,}\!\!_{2}, \widehat{f}_3)^\top\bigr\|_{L^2(Q, {\mathbb C}^3)}.
\end{equation}

Furthermore, as a consequence of the material symmetries (\ref{symmetry}), we have 
\begin{equation*} 
u_\alpha(\cdot, -x_3)=-u_\alpha(\cdot,x_3), \quad \alpha=1,2,\qquad 
 \quad u_3(\cdot,-x_3)=u_3 (\cdot,x_3), \qquad x_3\in I.    
\end{equation*}
In combination with \eqref{janprva}--\eqref{jantreca}, \eqref{sym_nabla_est} this implies the existence of $c_3 \in \mathbb{C}$ such that 
\begin{align*} 
\| u_\alpha+{\rm i}\chi_\alpha c_3 x_3e_\chi\|_{H^1(Q, {\mathbb C})}&\leq  C|\chi|^2 \bigl\|(\widecheck{f\,}\!\!_{1},\widecheck{f\,}\!\!_{2}, \widehat{f}_3)^\top\bigr\|_{L^2(Q, {\mathbb C}^3)},\qquad\alpha=1,2,\\[0.6em]
\| u_3-c_3e_\chi\|_{H^1(Q, {\mathbb C})}&\leq  C|\chi|^2\bigl\|(\widecheck{f\,}\!\!_{1},\widecheck{f\,}\!\!_{2}, \widehat{f}_3)^\top\bigr\|_{L^2(Q, {\mathbb C}^3)}.  
\end{align*}
which suggests that the solution to an appropriate asymptotic equation approximates $u$ 
 with an error of order $O(|\chi|^2)$ with respect to the $H^1$ norm. 
  In Section \ref{asymptotic_proc} we demonstrate that this is indeed so, see the first and fourth estimates in (\ref{korrre1}).

\subsubsection{Second (``membrane'') invariant subspace} 
\label{section_secondsubspace}

Similarly to the argument for the ``bending" invariant subspace, it follows from the analysis of Section \ref{structure} that the second invariant subspace 
described above contains the two-dimensional space spanned by the eigenfunctions of ${\mathcal A}_\chi$ corresponding to the eigenvalues $\lambda_\chi^{(2)},$ $\lambda_\chi^{(3)},$ since it contains vectors $(e_{\chi}, 0, 0)^\top,$ $(0, e_{\chi}, 0)^\top$ on which the Rayleigh quotient attains values of order $|\chi|^2.$
Guided by this observation, we scale the operator ${\mathcal A}_\chi$ with $|\chi|^{-2}$ and do not scale the force density $f,$ 
so the identity \eqref{general_problemrevision} is replaced with
\begin{align}
\frac{1}{|\chi|^2}{{\mathfrak b}_\chi}(u,\varphi)+\int_Q u\cdot\overline\varphi
= \int_Q \widehat{f\,}\!\!_{\alpha}\overline{\varphi_\alpha}+\int_Q\widecheck{f\,}\!\!_{3}\overline{\varphi_3}\quad\quad\forall\varphi\in H^1_\chi(Q, {\mathbb C}^3),
\label{jdba10_newsecond}
\end{align}
where we assume that $f_\alpha=\widehat{f\,}\!\!_{\alpha},$ $\alpha=1,2,$ are even and $f_3=\widecheck{f\,}\!\!_{3}$ is odd in the $x_3$ variable. Setting $\varphi=u$ in (\ref{jdba10_new}), 
we obtain
\begin{equation}
\|u\|_{L^2(Q, {\mathbb C}^3)}\leq C\bigl\|(\widehat{f}_{1}, \widehat{f}_2, \widecheck{f\,}\!\!_{3})^\top\bigr\|_{L^2(Q, {\mathbb C}^3)},\qquad \bigl\| \textrm{sym}\nabla u
\bigr\|_{L^2(Q, {\mathbb C}^{3\times3})} \leq C|\chi| 
\bigl\|(\widehat{f}_{1}, \widehat{f}_2, \widecheck{f\,}\!\!_{3})^\top\bigr\|_{L^2(Q, {\mathbb C}^3)}.
\label{init_est}
\end{equation}
The estimates \eqref{cetvrta}, \eqref{janprva} and the second estimate in (\ref{init_est})
imply ({\it cf.} \eqref{revision1}, (\ref{revision11}))
\begin{equation}
\bigl\|(u_1, u_2)^\top\bigr\|_{H^1(Q, {\mathbb C}^2)}\leq
C|\chi|^{-1}
\bigl\| \textrm{sym}\nabla u
\bigr\|_{L^2(Q, {\mathbb C}^{3\times3})}\leq C\bigl\|(\widehat{f}_{1}, \widehat{f}_2, \widecheck{f\,}\!\!_{3})^\top\bigr\|_{L^2(Q, {\mathbb C}^3)}.
\label{revision1second1}
\end{equation}
Furthermore, by virtue of the material symmetries (\ref{symmetry}), we have
\begin{equation} 
\label{symmetry2revisionintrosecond}
u_\alpha(\cdot, -x_3)=u_\alpha(\cdot,x_3),\quad\alpha=1,2,
 \qquad u_3(\cdot,-x_3)=-u_3 (\cdot,x_3), \qquad x_3\in I,    
\end{equation}
which, in combination with \eqref{jantreca} and the second estimate in (\ref{init_est}),
implies ({\it cf.} 
\eqref{revision11})
\begin{equation}
\|u_3\|_{H^1(Q, {\mathbb C})}\leq
C \bigl\| \textrm{sym}\nabla u
\bigr\|_{L^2(Q, {\mathbb C}^{3\times3})}\leq C|\chi|\bigl\|(\widehat{f}_{1}, \widehat{f}_{2}, \widecheck{f\,}\!\!_{3})^\top\bigr\|_{L^2(Q, {\mathbb C}^3)}.
\label{revision1second2}
\end{equation}
Finally, 
\eqref{janprva}--\eqref{jantreca},  \eqref{symmetry2revisionintrosecond}, 
and the second estimate in (\ref{init_est}) imply 
the existence of $c_1,c_2\in \mathbb{C}$ such that
\begin{equation}
\|u_\alpha-c_{\alpha} e_\chi\|_{H^1(Q, {\mathbb C}^3)}
\leq C|\chi|\bigl\|(\widehat{f}_1, \widehat{f}_2, \widecheck{f\,}\!\!_3)^\top\bigr\|_{L^2(Q, {\mathbb C}^3)} ,\qquad\alpha=1,2,
\label{prelim_est2}
\end{equation}
which suggests that the solution to an appropriate asymptotic equation approximates $u_{\alpha},$ $\alpha=1,2,$ with an error of order $O(|\chi|)$ in $H^1.$
We demonstrate the latter in Section \ref{asymptotic_proc}, see the first estimate in (\ref{korrre1oo}).

Notice also that 
setting $\widehat{f}_{\alpha}=0,$ $\alpha=1,2,$ in
 \eqref{jdba10_newsecond}, we obtain
\begin{equation}
\bigl\| \textrm{sym}\nabla u
\bigr\|_{L^2(Q, {\mathbb C}^{3\times3})}
\leq C|\chi|^2\bigl\|\widecheck{f\,}\!\!_{3}\bigr\|_{L^2(Q, {\mathbb C})}, 
\end{equation}
and, using \eqref{revision1second1}, \eqref{revision1second2}, it yields the bounds
\begin{equation}
\bigl\|(u_1, u_2)^\top\bigr\|_{H^1(Q, {\mathbb C}^2)}\leq
C |\chi|\bigl\| \widecheck{f\,}\!\!_{3}\bigr\|_{L^2(Q, {\mathbb C})},\quad 
\|u_3\|_{H^1(Q, {\mathbb C})}\leq 
 C|\chi|^2\bigl\| \widecheck{f\,}\!\!_{3}\bigr\|_{L^2(Q, {\mathbb C})}.
\label{additional_est}
\end{equation} 
As a result, the component $\widecheck{f\,}\!\!_{3}$ can be set to zero for purposes of obtaining the leading-order $H^1$ 
estimates (\ref{prelim_est2}),
see also Remark \ref{f3_omitted}.


 \subsection{
The case of a general elasticity tensor}
\label{casegeneltens}
 
In the general case, when we do not assume \eqref{symmetry}, consider two asymptotic procedures separately. The first of these corresponds to the problem ({\it cf.} (\ref{jdba10_new}))
\begin{equation} 
\label{scalingjedan} 
\frac{1}{|\chi|^4}{{\mathfrak b}_\chi}(u, \varphi)
+\int_Q u\cdot\overline\varphi
=\frac{1}{|\chi|}\int_Q f\,\!\!_{\alpha}\overline{\varphi_\alpha}+\int_Q  f_3\overline{\varphi_3}\quad\quad\forall\varphi\in H^1_\chi(Q, {\mathbb C}^3).
\end{equation}
In the same way as in Section 
\ref{first_subspace}, we obtain the estimates ({\it cf.} (\ref{sym_nabla_est}), (\ref{revision11}))
\begin{align*}
\bigl\| \textrm{sym}\nabla u\bigr\|_{L^2(Q, {\mathbb C}^{3\times3})} &\leq C|\chi|^2\|f\|_{L^2(Q, {\mathbb C}^3)}, \\[0.6em]
	\bigl\|(u_1, u_2)^\top\bigr\|_{H^1(Q, {\mathbb C}^2)}
	&\leq C |\chi| \|f\|_{L^2(Q, {\mathbb C}^3)},\quad
	\|u_3\|_{H^1(Q, {\mathbb C})}
	\leq C\|f\|_{L^2(Q, {\mathbb C}^3)},
\end{align*}
as well as the existence of  $c_1, c_2, c_3 \in \mathbb{C}$ such that
\begin{align}
 \| u_\alpha-(c_{\alpha}-{\rm i}\chi_\alpha c_3 x_3)e_\chi\|_{H^1(Q, {\mathbb C})}&\leq C|\chi|^2 \|f\|_{L^2(Q, {\mathbb C}^3)},\qquad\alpha=1,2,\label{approx11}
 \\[0.8em]
 \| u_3-c_3e_\chi\|_{H^1(Q, {\mathbb C})}&\leq  C|\chi|^2\|f\|_{L^2(Q, {\mathbb C}^3)}.\label{approx12}  
\end{align}

The second procedure corresponds to the problem ({\it cf.} (\ref{jdba10_newsecond}))
\begin{align}
\frac{1}{|\chi|^2}{{\mathfrak b}_\chi}(u,\varphi)+\int_Q u\cdot\overline\varphi
= \int_Qf_\alpha\overline{\varphi_\alpha}+\int_Qf_3\overline{\varphi_3}\quad\quad\forall\varphi\in H^1_\chi(Q, {\mathbb C}^3).
\label{scalingdva}
\end{align}
We follow the strategy of Section \ref{section_secondsubspace}. Namely, first setting $\varphi=u$ in (\ref{scalingdva}) we obtain ({\it cf.} (\ref{init_est}))
\begin{equation}
\|u\|_{L^2(Q, {\mathbb C}^3)} \leq C\|f\|_{L^2(Q, {\mathbb C}^3)},\qquad \bigl\| \textrm{sym}\nabla u
\bigr\|_{L^2(Q, {\mathbb C}^{3\times3})} \leq C|\chi| 
\bigl\|f\|_{L^2(Q, {\mathbb C}^3)}.
\label{init_est_gen}
\end{equation}
Second, \eqref{cetvrta},
\eqref{janprva},
and the second estimate in (\ref{init_est_gen})
imply ({\it cf.} \eqref{revision1second1})
\begin{equation*}
\bigl\|(u_1, u_2)^\top\bigr\|_{H^1(Q, {\mathbb C}^2)}
\leq
C|\chi|^{-1}
\bigl\| \textrm{sym}\nabla u
\bigr\|_{L^2(Q, {\mathbb C}^{3\times3})}\leq C\|f\|_{L^2(Q, {\mathbb C}^3)},
\end{equation*}
Finally, \eqref{janprva}--\eqref{jantreca} and the second estimate in (\ref{init_est_gen}) imply the existence of $c_1,c_2,c_3 \in \mathbb{C}$  such that
\begin{align}
\bigl\| u_\alpha-(c_{\alpha}-{\rm i}\chi_\alpha c_3 x_3)e_\chi\bigr\|_{H^1(Q, {\mathbb C})}&\leq  C|\chi|\|f\|_{L^2(Q, {\mathbb C}^3)},\qquad\alpha=1,2\label{approx21},\\[0.55em]
\| u_3-c_3e_\chi\|_{H^1(Q, {\mathbb C})}&\leq  C|\chi|\|f\|_{L^2(Q, {\mathbb C}^3)}.\label{approx22}  
\end{align}
It follows from (\ref{init_est_gen}) and (\ref{approx22}) that
\[
|c_3|\le C\|f\|_{L^2(Q, {\mathbb C}^3)}.
\]
Therefore ({\it cf.} \eqref{revision1second2})
\[
\|u_3\|_{H^1(Q, {\mathbb C})} \leq C\|f\|_{L^2(Q, {\mathbb C}^3)}
\]
and ({\it cf.} (\ref{approx21}))
\[
\bigl\| u_\alpha-c_{\alpha}e_\chi\bigr\|_{H^1(Q, {\mathbb C})}\leq  C|\chi|\|f\|_{L^2(Q, {\mathbb C}^3)},\qquad\alpha=1,2.
\]
The values of $c_j,$ $j=1,2,3,$ in (\ref{approx11})--(\ref{approx12}), (\ref{approx21})--(\ref{approx22}) will be determined in Section \ref{gen_est_section}.

\begin{remark}
\label{expl2} 
	The existence of two distinguished scalings for the operator ${\mathcal A}_\chi$ plays a key role in the proof of Theorem \ref{main_result_general}.
	 When adopting the first scaling, we approximate the eigenspace ${\mathcal A}_\chi$ corresponding to the eigenvalue of order $|\chi|^4$. The related normalised eigenfunction has in-plane components of order $|\chi|$ and the vertical component of order one, thus 
	 we achieve one more order of precision in $|\chi|$ for the first two components, {\it cf.} (\ref{approx11})--(\ref{approx12}). Under the second scaling, we approximate the space spanned by eigenfunctions of the two eigenvalues of ${\mathcal A}_\chi$ of order $|\chi|^2$. The related normalised eigenfunctions have components of order one and we approximate them with precision of order $|\chi|,$ {\it cf.} (\ref{approx21})--(\ref{approx22}).
	\end{remark}
	
	 In the next section we develop an asymptotic procedure for the ($\varepsilon$-independent) problems (\ref{jdba10_new}), (\ref{jdba10_newsecond}), (\ref{scalingjedan}) with respect to the modulus of the quasimomentum $\chi.$ In this way we obtain order-optimal resolvent estimates for (\ref{general_problemrevision}), equivalently for the fibres ${\mathcal A}_\chi,$ see (\ref{bform}). This allows us to derive norm-resolvent approximations for scaled versions of the original operator 
	 ${\mathcal A}^\varepsilon,$ see also the original formulation (\ref{original_identity}), which is a direct integral of ${\mathcal A}_\chi,$ $\chi\in Q'_{\rm r},$ by the Floquet transform 
	 \cite[Chapter 3]{Kuchment}, \cite[Section XIII.16]{ReedSimon}, \cite[Chapter 7]{BirmanSolomyak}.
	 	 It also makes our approach different to that presented in  \cite{ChCoARMA}, where an asymptotics as $\varepsilon\to 0$ uniform with respect to the quasimomentum, rather than an asymptotics in $|\chi|,$ was established for a multi-dimensional degenerate problem. Since our estimates are optimal with respect to $\chi,$
		 we do not need to separate the dual cell $Q'_{\rm r}$ into the ``inner" and ``outer" regions like it is done in \cite{ChCoARMA}.
		 In addition, we identify the correctors required for $L^2\to H^1$ estimates as well as for higher-precision $L^2\to L^2$ estimates, see Section \ref{normresolvent} for details. By adopting particular scalings for the operator of (\ref{original_identity}) and the force density, we thus also recover a version of the results of \cite{BirmanSuslina}.

From the {\it a priori} estimates of Section \ref{first_subspace} we see that the appropriate scaling of the horizontal components of the displacement with respect to the vertical one is by $|\chi|,$ not $\varepsilon$ as is usually done for bounded plates. This suggests that we should scale the horizontal force components with respect to the vertical ones with the pre-factor $|\chi|^{-1},$ in order that the corresponding parts of the work by external loads $\int_Q (f_1u_1+f_2u_2)$ and  
$\int_Qf_3u_3$ are of the same order. In our final results we allow arbitrary scalings  $\varepsilon^{-\delta}$ between horizontal and vertical forces, and the final approximation error then depends also on the scaling exponent $\delta$. 

In the asymptotic procedure we capture as many correctors in the expansion as is necessary to infer the error estimate from the equation. In the end these correctors are not required in the final result --- this is standard in homogenisation, see {\it e.g.} \cite[Chapter 7]{CD}. At each step of the asymptotic procedure we ensure that the related system has a solution --- this is done by the method of ``updating", which to the best of our knowledge is also novel. Repeating this process allows us to obtain the complete asymptotic expansion, up to any order in $|\chi|.$

\section{Solution asymptotics: recurrence relations and error estimates}

\label{asymptotic_proc}


In the asymptotic procedure developed in this section, we use the operator 
$X$ on $L^2(Q, {\mathbb C}^3)$ defined for each $\chi\in Q'_{\rm r}$ as the matrix multiplication
\begin{equation*}
X\varphi:=\left(\begin{array}{ccc} \chi_1 \varphi_1 & \dfrac{1}{2} (\chi_2 \varphi_1+\chi_1 \varphi_2) & \dfrac{1}{2} \chi_1 \varphi_3 \\[1.3em] \dfrac{1}{2} (\chi_2 \varphi_1+\chi_1 \varphi_2) & \chi_2 \varphi_2 & \dfrac{1}{2} \chi_2 \varphi_3 \\[1.3em]  \dfrac{1}{2} \chi_1 \varphi_3  & \dfrac{1}{2} \chi_2 \varphi_3 & 0
\end{array} \right),\qquad \varphi\in L^2(Q, {\mathbb C}^3).
\end{equation*}
It is easy to see that for all $\varphi\in L^2(Q, {\mathbb C}^3)$ one has
\begin{equation}
\label{bukal1000}
\bigl\|X\varphi\bigr\|_{L^2(Q, {\mathbb C}^3)} \geq \frac{1}{2}|\chi| \|\varphi\|_{L^2(Q, {\mathbb C}^3)},\quad \chi\in Q'_{\rm r}.
\end{equation}
We also define, for $m_1,m_2,m_3 \in \mathbb{C},$
\begin{align*}
\Xi(\chi, m_1,m_2)&:={\rm i}{\mathfrak I}\left(\begin{array}{ccc} \chi_1 m_1 & \dfrac{1}{2}(\chi_1 m_2+ \chi_2 m_1)\\[0.8em] 
\dfrac{1}{2}(\chi_1 m_2 +\chi_2 m_1)  & \chi_2 m_2\end{array} \right),\quad
\Upsilon(\chi, m_3):={\rm i}m_3{\mathfrak I}\left(\begin{array}{ccc} \chi_1^2& \chi_1 \chi_2 \\[0.8em] \chi_1 \chi_2 & \chi_2^2\end{array} \right),
\end{align*}
where ${\mathfrak I}$ is the mapping (\ref{iota_def}), and remark the following useful identity:
\begin{equation}
\label{ljulj1} {\rm i}X(D_1-{\rm i}\chi_1x_3 D_3,D_2-{\rm i}\chi_2x_3 D_3,0)^{\top}=\Xi (\chi, D_1,D_2)-{\rm i}x_3 \Upsilon(\chi, D_3)\qquad \forall D=(D_1, D_2, D_3)\in \mathbb{C}^3. 
\end{equation}	

Furthermore, for each $\chi\in Q'_{\rm r}$ we define a matrix $A^{\textrm{hom}}_{\chi}$ by the formula 
\begin{align*} 
&A^{\textrm{hom}}_{\chi}m\cdot\overline{d}
:=\int_Q A\bigl(\nabla N_m+\Xi(\chi,m_1,m_2)-{\rm i}x_3 \Upsilon(\chi, m_3)\bigr):\overline{\bigl(\Xi(\chi,d_1,d_2)-{\rm i}x_3 \Upsilon(\chi, d_3)\bigr)},\qquad m, d\in{\mathbb C}^3,
\end{align*} 
where $N_m,$ $m\in{\mathbb C}^3,$ satisfy
\begin{align*}
({\rm sym}\nabla)^{*} A\,{\rm sym}\nabla N_m=-({\rm sym}\nabla)^{*} A\bigl(\Xi(\chi,m_1,m_2)-{\rm i}x_3 \Upsilon(\chi, m_3)\bigr),\qquad
 N_m\in H^1_{\#}(Q, {\mathbb C}^3),\quad\int_Q N_m=0.
\end{align*}
Notice that
\begin{equation}
\begin{aligned} 
\max\bigl\{|m_1|, |m_2|\bigr\}+\max\bigl\{|\chi_1|,|\chi_2|\bigr\} |m_3|&
\leq  C\vert\chi\vert^{-1} 
\Bigl(\bigl|\Xi(\chi,m_1,m_2)\bigr|+\bigl|\Upsilon(\chi, m_3)\bigr|\Bigr),\\[0.5em] 
|m_3| &\leq C\vert\chi\vert^{-2} 
\bigl|\Upsilon(\chi,m_3)\bigr\vert,
\end{aligned}
\label{muic1}
\end{equation}
and in the case of the planar symmetries \eqref{symmetry} we have
$ N_m=N_m^{(1)}+N_m^{(2)}, $
where 
\begin{align*}
({\rm sym}\nabla)^{*}A\,{\rm sym}\nabla N_m^{(1)} &= ({\rm sym}\nabla)^{*}A\bigl({\rm i}x_3 \Upsilon(\chi, m_3)\bigr),\quad\quad N_m^{(1)}\in H^1_{\#}(Q, {\mathbb C}^3), \quad \quad\int_Q N_m^{(1)}=0.
\\[0.7em]
({\rm sym}\nabla)^{*}A\,{\rm sym}\nabla N_m^{(2)} &= -({\rm sym}\nabla)^{*}A\Xi(\chi,m_1,m_2),\quad\quad 
N_m^{(2)}\in H^1_{\#}(Q, {\mathbb C}^3), \quad\quad\int_Q N_m^{(2)}=0.
\end{align*}
By a symmetry argument, we have, for all $m\in{\mathbb C}^3,$ $x_3\in I,$ $j=1,2:$
\begin{equation*}
\bigl(N_{m}^{(j)}\bigr)_{\alpha}(\cdot,-x_3)=\mp\bigl(N_{m}^{(j)}\bigr)_{\alpha}(\cdot,x_3),\quad\alpha=1,2,\qquad
\bigl(N_{m}^{(j)}\bigr)_{3}(\cdot,-x_3)=\pm\bigl(N_{m}^{(j)}\bigr)_{3}(\cdot,x_3),
\end{equation*}
where the top and bottom signs in ``$\mp$" and ``$\pm$'' are taken for $j=1$ and $j=2,$ respectively.
Therefore,
\begin{equation*}
A^{\textrm{hom}}_{\chi} m\cdot\overline{d}
=A^{\textrm{hom},1}_{\chi}m_3\,\overline{d_3}+A^{\textrm{hom},2}_{\chi}(m_1,m_2)^\top\cdot \overline{(d_1,d_2)^\top},  
\end{equation*}
where, for $\chi\in Q'_{\rm r},$ $m_3, d_3\in{\mathbb C},$ $(m_1, m_2)^\top,$ $(d_1, d_2)^\top\in{\mathbb C}^2,$ we define
\begin{align*} 
	A^{\textrm{hom},1}_{\chi}m_3\,\overline{d_3}&:=\int_Q A\bigl(\nabla N_m^{(1)}-{\rm i}x_3 \Upsilon(\chi, m_3)\bigr): \overline{\bigl(-{\rm i}x_3 \Upsilon(\chi, d_3)\bigr)},
	\\[0.4em]  
	A^{\textrm{hom},2}_{\chi}(m_1,m_2)^\top\cdot\overline{(d_1,d_2)^\top}  &:=\int_Q A\bigl(\nabla N_m^{(2)}+\Xi(\chi,m_1,m_2)\bigr):\overline{ \Xi(\chi,d_1,d_2)}. 
\end{align*}

\subsection{Planar-symmetric elasticity tensor}
\label{planar} 
In this section we assume that the ``planar symmetry" conditions (\ref{symmetry}) on the elasticity tensor $A$ hold and, as in Section \ref{planar_symmetric}, analyse the bending and membrane invariant subspaces separately. We use the notation $H^1_\#(Q, {\mathbb C}^3)$ for the closure of smooth functions $Q_{\rm r}$-periodic 
in $y_1,$ $y_2$ in the  norm of $H^1(Q, {\mathbb C}^3).$ 

\subsubsection{First invariant subspace}
\label{section_prvaasimptotika}

\textbf{Asymptotic equation.} 
Here assume that the components $f_\alpha=\widecheck{f\,}\!\!_\alpha,$ $\alpha=1, 2,$ of the body force density are odd in the $x_3$ variable, while $f_3=\widehat{f\,}\!\!_3$ is even, {\it cf.} Section \ref{first_subspace}, and throughout the present section we write $f$ in place of $(\widecheck{f\,}\!\!_{1},\widecheck{f\,}\!\!_{2}, \widehat{f}_3)^\top.$ An approximating problem for the leading-order amplitude $m_3\in{\mathbb C}$ in the first invariant subspace, see the estimates (\ref{korrre1}), has the form\footnote{Note that the left-hand side of (\ref{sinisa1000revision}) can be simplified, by observing that
\begin{equation*}
\int_Q (-{\rm i}\chi_1 x_3 m_3, -{\rm i}\chi_2 x_3m_3,m_3)^\top \cdot \overline{(-{\rm i}\chi_1 x_3 d_3, -{\rm i}\chi_2 x_3d_3,d_3)^\top}=\biggl(\frac{|\chi|^2}{3}+1\biggr)m_3\overline{d_3},
\end{equation*}
see also (\ref{sinisa1500oo}), (\ref{sinisa1500}). For clarity of the argument, we keep the original form of the second term on the left-hand side of (\ref{sinisa1000revision}).
}
\begin{equation}
\begin{aligned}
\vert\chi\vert^{-4}A^{\textrm{hom},1}_{\chi}m_3\,\overline{d_3} 
&+\int_Q (-{\rm i}\chi_1 x_3 m_3, -{\rm i}\chi_2 x_3m_3,m_3)^\top \cdot \overline{(-{\rm i}\chi_1 x_3 d_3, -{\rm i}\chi_2 x_3d_3,d_3)^\top}
\\[0.2em] 
&=\int_Q|\chi|^{-1}\bigl(\widecheck{f\,}\!\!_1, \widecheck{f\,}\!\!_2\bigr)^\top \cdot\overline{e_\chi(-{\rm i}\chi_1 x_3 d_3, -{\rm i}\chi_2 x_3d_3)^\top}
+\int_Q \widehat{f\,}\!\!_3 \cdot \overline{e_\chi d_3} \quad\quad\forall d_3\in{\mathbb C},
\end{aligned}
\label{sinisa1000revision}
\end{equation}
so that the following estimate holds:
\begin{equation} 
\label{propertiesrevision}
\vert m_3\vert \leq C\| f
\|_{L^2(Q, {\mathbb C}^3)}. 
\end{equation}
\vskip 0.1cm
\noindent\textbf{Approximation error estimates in the first invariant subspace.} 
{\it Step 1.} We define ${\mathfrak u}_2 \in H^1_{\#}(Q,\mathbb{C}^3)$ that solves 
\begin{equation} \label{sinisa1001}
(\sym \nabla)^{*}A \sym \nabla {\mathfrak u}_2=(\sym \nabla)^{*}A\bigl({\rm i}x_3 \Upsilon(\chi,m_3)\bigr), \quad\quad \int_Q {\mathfrak u}_2=0,      
\end{equation}
so that, by the standard Poincar\'{e} inequality combined with the Korn inequality for periodic functions ({\it cf.} (\ref{kornquasi})), we have the bound
\begin{equation} 
\label{sinisa1002}
\|{\mathfrak u}_2\|_{H^1(Q, {\mathbb C}^3)} \leq C |\chi|^2\|f
\|_{L^2(Q, {\mathbb C}^2)}.
\end{equation}
In addition, the following symmetry properties\footnote{Similar symmetry properties hold for all terms in the asymptotics series, see {\it e.g.} (\ref{sinisa1005}).} hold: 
\begin{equation} 
\label{sinisa1003} 
({\mathfrak u}_{2})_\alpha(\cdot,-x_3)=-({\mathfrak u}_{2})_{\alpha}(\cdot,x_3), \quad \alpha=1,2, 
\qquad({\mathfrak u}_{2})_{3}(\cdot,-x_3)=({\mathfrak u}_{2})_{3}(\cdot,x_3),\quad x_3\in I.
\end{equation}

Next, we define ${\mathfrak u}_3^{(1)} \in H^1_{\#}(Q,\mathbb{C}^3)$ that satisfies 
\begin{equation}
\begin{aligned} 
(\sym \nabla)^{*}A \sym \nabla {\mathfrak u}_3^{(1)}&= {\rm i}\bigl\{X^{*}A \sym \nabla {\mathfrak u}_2-(\sym \nabla)^{*}  A(X{\mathfrak u}_2)
-X^{*}A\bigl({\rm i}x_3\Upsilon(\chi,m_3)\bigr)\bigr\}
\\[0.5em]
&+|\chi|^4 ({\rm i}\chi_1 x_3 m_3, {\rm i}\chi_2 x_3m_3,0)^\top  +|\chi|^3\overline{e_\chi}\bigl(\widecheck{f\,}\!\!_1,\widecheck{f\,}\!\!_2,0\bigr)^\top,\quad\quad \int_Q {\mathfrak u}_3^{(1)}=0.
\end{aligned}
\label{sinisa1004}
\end{equation}
In view of the symmetry properties \eqref{sinisa1003}, the problem (\ref{sinisa1004}) has a unique solution. Indeed, it follows from (\ref{sinisa1003}) that the right-hand side of the above equation for ${\mathfrak u}_3^{(1)}$ vanishes when tested with vectors $(D_1,D_2,0)^\top$. To verify that it also vanishes when tested with vectors $(0,0,D_3)^\top,$ we use the fact that 
\begin{equation}
\label{psi_id}
\begin{aligned}
\int_Q  A\bigl(-{\rm i}x_3 \Upsilon(\chi,m_3)+\sym \nabla {\mathfrak u}_2\bigr):\overline{{\rm i}X(0,0,D_3)^\top}&= 
\int_Q  A\bigl( -{\rm i}x_3 \Upsilon(\chi,m_3)+\sym \nabla {\mathfrak u}_2\bigr):\overline{\sym \nabla \psi}=0,\\[0.4em]
&\psi:={\rm i}D_3 x_3(\chi_1, \chi_2 ,0)^\top.
\end{aligned}
\end{equation}
By virtue of (\ref{propertiesrevision}), (\ref{sinisa1002}), the solution to (\ref{sinisa1004}) satisfies the estimate 
\begin{equation}
\label{julian101} 
\bigl\|{\mathfrak u}_3^{(1)}\bigr\|_{H^1(Q, {\mathbb C}^3)} \leq C |\chi|^3\|f
\|_{L^2(Q, {\mathbb C}^3)}. 
\end{equation}
Furthermore, using the assumption on the forces and the properties \eqref{sinisa1003} we obtain 
\begin{equation}
 \label{sinisa1005} 
\bigl({\mathfrak u}_{3}^{(1)}\bigr)_{\alpha}(\cdot,-x_3)=-\bigl({\mathfrak u}_{3}^{(1)}\bigr)_{\alpha}(\cdot,x_3), \quad\alpha=1,2,\qquad
\bigl({\mathfrak u}_{3}^{(1)}\bigr)_{3}(\cdot,-x_3)=\bigl({\mathfrak u}_{3}^{(1)}\bigr)_{3}(\cdot,x_3),\quad x_3\in I.
\end{equation}

Next, we seek ${\mathfrak u}_4^{(1)} \in H^1_{\#}(Q,\mathbb{C}^3)$ such that
\begin{equation}
\begin{aligned}
(\sym \nabla)^{*}A \sym \nabla {\mathfrak u}_4^{(1)} &= {\rm i}\bigl\{X^{*}A \sym \nabla {\mathfrak u}_3^{(1)}-(\sym \nabla)^{*}A\bigl(X{\mathfrak u}_3^{(1)}\bigr)\bigr\}
\\[0.5em] &-X^{*}A(X{\mathfrak u}_2)
-|\chi|^4\bigl(0,0,m_3-\widehat{f}_3\overline{e_\chi}\bigr)^\top,
\quad\quad \int_Q {\mathfrak u}_4^{(1)}=0.
\end{aligned}
\label{sinisa1010}
\end{equation}
The right-hand side of the problem for ${\mathfrak u}_4^{(1)}$ yields zero when tested with constant vectors of the form $(D_1,D_2,0)^\top$. To ensure it also vanishes when tested with vectors $(0,0,D_3)^\top,$ we observe that ({\it cf.} (\ref{psi_id}))
\begin{equation*}
\begin{aligned}
	& \int_Q A \bigl(\sym \nabla {\mathfrak u}_3^{(1)}+ {\rm i}X{\mathfrak u}_2 \bigr): \overline{{\rm i}X(0,0,D_3)^\top}
	=\int_Q A \bigl(\sym \nabla {\mathfrak u}_3^{(1)}+ {\rm i}X {\mathfrak u}_2 \bigr): \overline{ \sym \nabla \psi}\\[0.4em] 
	& 
	=A^{\textrm{hom},1}_{\chi}m_3\,\overline{D}_3
	+|\chi|^4\int_Q ({\rm i}\chi_1 x_3 m_3, {\rm i}\chi_2 x_3m_3,0)^\top \cdot \overline{({\rm i}\chi_1 x_3 D_3, {\rm i}\chi_2 x_3D_3,0)^\top}
	\\[0.4em] 
	& \hspace{+8ex}
	\qquad\quad\ \ +|\chi|^3\int_Q \bigl(\widecheck{f\,}\!\!_1,\widecheck{f\,}\!\!_2\bigr)^\top \cdot \overline{e_\chi({\rm i}\chi_1 x_3 D_3, {\rm i}\chi_2 x_3D_3)^\top}
	=-|\chi|^4 m_3\overline{D_3}+|\chi|^4 \int_Q \widehat{f}_3 \cdot \overline{e_{\chi} D_3},
\end{aligned}
\end{equation*}
where we use \eqref{sinisa1000revision}, \eqref{sinisa1004}. 

Similarly to the above argument for $\mathfrak{u}_2,$ $\mathfrak{u}_3^{(1)},$ the following estimate holds for the solution of (\ref{sinisa1010}):
\begin{equation*}
\bigl\|{\mathfrak u}_4^{(1)}\bigr\|_{H^1(Q, {\mathbb C}^3)} \leq C |\chi|^4\|f
\|_{L^2(Q, {\mathbb C}^3)}. 
\end{equation*}
\vskip 0.2cm
{\it Step 2.} 
We proceed by updating the leading-order term $(-\chi_1 x_3 m_3,-{\rm i}\chi_2 x_3 m_3,m_3)^\top$. To this end, we define a ``correction'' $m_3^{(1)}$ to the value $m_3$ as the solution to 
\begin{equation}
\begin{aligned} 
\label{sinisa1500oo}
A^{\textrm{hom},1}_{\chi}m_3^{(1)}\,\overline{d_3} 
&+|\chi|^4\int_Q\bigl(-{\rm i}\chi_1 x_3 m_3^{(1)}, -{\rm i}\chi_2 x_3m_3^{(1)},m_3^{(1)}\bigr)^\top \cdot \overline{(-{\rm i}\chi_1 x_3 d_3, -{\rm i}\chi_2 x_3d_3,d_3)^\top}
\\[0.4em]
& 
=-\int_Q A \bigl(\sym \nabla {\mathfrak u}_4^{(1)}+{\rm i}X{\mathfrak u}_3^{(1)}\bigr): \overline{{\rm i}X(0,0,d_3)^\top}\quad\quad\forall d_3\in{\mathbb C}.
\end{aligned}
\end{equation}
We then have the estimate
\begin{equation} \label{julian100} 
\bigl|m_3^{(1)}\bigr| \leq C|\chi|
\|f
\|_{L^2(Q, {\mathbb C}^3)}.
\end{equation}
Following the correction,
we repeat the procedure and first define ${\mathfrak u}_3^{(2)} \in H^1_{\#} (Q, \mathbb{C}^3)$ as the solution of ({\it cf.} (\ref{sinisa1002}))
\begin{equation} 
\label{sinisa1120}
(\sym \nabla)^{*} A \sym \nabla {\mathfrak u}_3^{(2)}=(\sym \nabla)^{*}A\bigl({\rm i}x_3 \Upsilon\bigl(\chi,m_3^{(1)}\bigr)\bigr),\quad\quad   \int_Q {\mathfrak u}_3^{(2)}=0,
\end{equation}
so that
\begin{equation} \label{julian102} 
\bigl\|{\mathfrak u}_3^{(2)}\bigr\|_{H^1(Q, {\mathbb C}^3)} \leq C |\chi|^3\|f
\|_{L^2(Q, {\mathbb C}^2)}.
\end{equation}
Second, consider ${\mathfrak u}_4^{(2)} \in  H^1_{\#} (Q, \mathbb{C}^3)$ satisfying ({\it cf.} (\ref{sinisa1004}))
\begin{equation}
\label{sinisa1104oo}
\begin{aligned}
(\sym \nabla)^{*}A \sym \nabla  {\mathfrak u}_4^{(2)}&= {\rm i}\bigl\{X^{*} \cdot A \sym \nabla {\mathfrak u}_3^{(2)}-(\sym \nabla)^{*} A\bigl(X{\mathfrak u}_3^{(2)}\bigr)
\\[0.7em]  & -X^{*} \cdot A\bigl({\rm i}x_3\Upsilon\bigl(\chi,m_3^{(1)}\bigr)\bigr)\bigr\}
+|\chi|^4\bigl({\rm i}\chi_1 x_3 m_3^{(1)}, {\rm i}\chi_2 x_3m_3^{(1)},0\bigr)^\top,\quad \int_Q {\mathfrak u}_4^{(2)}=0.
\end{aligned}
\end{equation}
In the same way as above, we conclude that (\ref{sinisa1104oo}) has a unique solution, and  
\begin{equation} \label{julian103} 
\bigl\|{\mathfrak u}_4^{(2)}\bigr\|_{H^1(Q, {\mathbb C}^3)} \leq C |\chi|^4\|f
\|_{L^2(Q, {\mathbb C}^3)}.
\end{equation}

Finally, we define ${\mathfrak u}_5^{(1)} \in  H^1_{\#} (Q, \mathbb{C}^3)$ as the solution to ({\it cf.} (\ref{sinisa1010}))
\begin{equation}
\label{sinisa1022oo} 
\begin{aligned}
(\sym \nabla)^{*}A \sym \nabla {\mathfrak u}_5^{(1)} &= {\rm i}\Bigl(X^{*}A \sym \nabla\bigl({\mathfrak u}_4^{(1)}+{\mathfrak u}_4^{(2)}\bigr)-(\sym \nabla)^{*}A\bigl\{X\bigl({\mathfrak u}_4^{(1)}+{\mathfrak u}_4^{(2)}\bigr)\bigr\}\Bigr)
\\[0.7em]
&-X^{*}A\bigl\{X\bigl({\mathfrak u}_3^{(1)}+{\mathfrak u}_3^{(2)}\bigr)\bigr\}
-|\chi|^4 (0,0,m_3^{(1)})^\top,
\quad\quad \int_Q {\mathfrak u}_5^{(1)}=0.
\end{aligned}
\end{equation}
The right-hand side yields zero when tested with vectors $D=(D_1,D_2,0)^\top$, by symmetry properties similar to (\ref{sinisa1003}) for the terms involved. To see that it yields zero when tested with vectors $(0,0,D_3)^\top,$ we use the same reasoning as above ({\it cf.} (\ref{psi_id})):
\begin{align*}
\int_Q A \bigl(\sym \nabla {\mathfrak u}_4^{(2)}+{\rm i}X{\mathfrak u}_3^{(2)}\bigr)&: \overline{{\rm i}X(0,0,D_3)^\top}
=\int_Q A \bigl(\sym \nabla {\mathfrak u}_4^{(2)}+ {\rm i}X{\mathfrak u}_3^{(2)}\bigr): \overline{ \sym \nabla \psi} \\[0.4em]
&=A^{\textrm{hom},1}_{\chi}m_3^{(1)}\,\overline{D}_3
+|\chi|^4\int_Q\bigl({\rm i}\chi_1 x_3 m_3^{(1)}, {\rm i}\chi_2 x_3m_3^{(1)},0\bigr)^\top \cdot \overline{({\rm i}\chi_1 x_3 D_3, {\rm i}\chi_2 x_3D_3,0)^\top}\\[0.4em]
&=-|\chi|^4 m_3^{(1)}\,\overline{D}_3-\int_Q A \bigl(\sym \nabla {\mathfrak u}_4^{(1)}+{\rm i}X{\mathfrak u}_3^{(1)}\bigr): \overline{{\rm i}X(0,0,D_3)^\top},
\end{align*}
where we use \eqref{sinisa1500oo}, \eqref{sinisa1104oo}. Thus, the problem \eqref{sinisa1022oo} has a unique solution, and 
\begin{equation} \label{julian104} 
\bigl\|{\mathfrak u}_5^{(1)}\bigr\|_{H^1(Q, {\mathbb C}^3)} \leq C |\chi|^5\|f
\|_{L^2(Q, {\mathbb C}^3)}.
\end{equation}
{\it Step 3.} 
We once again update the leading-order term $\bigl(-\chi_1 x_3 m_3^{(2)},-{\rm i}\chi_2 x_3 m_3^{(2)},m_3^{(2)}\bigr)^\top$. Defining $m_3^{(2)}$ so that
\begin{align} \label{sinisa1500}
A^{\textrm{hom},1}_{\chi}m_3^{(2)}\,\overline{d_3} 
&+|\chi|^4\int_Q\bigl(-{\rm i}\chi_1 x_3 m_3^{(2)}, -{\rm i}\chi_2 x_3m_3^{(2)},m_3^{(2)}\bigr)^\top \cdot \overline{(-{\rm i}\chi_1 x_3 d_3, -{\rm i}\chi_2 x_3d_3,d_3)^\top}
\\[0.3em] \nonumber & 
=-\int_Q A \bigl(\sym \nabla {\mathfrak u}_5^{(1)}+{\rm i}X{\mathfrak u}_4^{(1)}+{\rm i}X{\mathfrak u}_4^{(2)}\bigr): \overline{{\rm i}X(0,0,d_3)^\top}
\quad\quad\forall d_3\in{\mathbb C},
\end{align}
one has
\begin{equation}\label{julian105} 
\bigl|m_3^{(2)}\bigr| \leq C|\chi|^2\|f
\|_{L^2(Q, {\mathbb C}^3)}.
\end{equation}
As in Step 2, we repeat the procedure and first define ${\mathfrak u}_4^{(3)} \in H^1_{\#} (Q, \mathbb{C}^3)$ as the solution to ({\it cf.} (\ref{sinisa1120}))
\begin{equation*}
(\sym \nabla)^{*} A \sym \nabla {\mathfrak u}_4^{(3)}=(\sym \nabla)^{*}A\bigl({\rm i}x_3 \Upsilon\bigl(\chi,m_3^{(2)}\bigr)\bigr),\quad\quad   \int_Q {\mathfrak u}_4^{(3)}=0,
\end{equation*}
so that
\begin{equation} \label{julian110} 
\bigl\|{\mathfrak u}_4^{(3)}\bigr\|_{H^1(Q, {\mathbb C}^3)} \leq C |\chi|^4\|f
\|_{L^2(Q, {\mathbb C}^3)}.
\end{equation}
Second, consider ${\mathfrak u}_5^{(2)} \in  H^1_{\#} (Q, \mathbb{C}^3)$ that satisfies ({\it cf.} (\ref{sinisa1104oo}))
\begin{equation}
\begin{aligned}
(\sym \nabla)^{*}A \sym \nabla  {\mathfrak u}_5^{(2)}&= {\rm i}\bigl\{X^{*}A \sym \nabla {\mathfrak u}_4^{(3)}-(\sym \nabla)^{*} A\bigl(X{\mathfrak u}_4^{(3)}\bigr)
 -X^{*}A\bigl({\rm i}x_3\Upsilon\bigl(\chi,m_3^{(2)}\bigr)\bigr)\bigr\}
\\[0.7em]  &
+|\chi|^4\bigl({\rm i}\chi_1 x_3 m_3^{(2)}, {\rm i}\chi_2 x_3m_3^{(2)},0\bigr)^\top,\quad \int_Q {\mathfrak u}_5 ^{(2)}=0.
\end{aligned}
 \label{sinisa1104}
\end{equation}
In the same way as above, we infer that the system has a unique solution, and 
\begin{equation}\label{julian111} 
\bigl\|{\mathfrak u}_5^{(2)}\bigr\|_{H^1(Q, {\mathbb C}^3)} \leq C |\chi|^4\|f
\|_{L^2(Q, {\mathbb C}^3)}.
\end{equation}
Finally, we define ${\mathfrak u}_6 \in  H^1_{\#} (Q, \mathbb{C}^3)$ as the solution to ({\it cf.} (\ref{sinisa1022oo}))
\begin{equation}
\label{sinisa1022} 
\begin{aligned}
(\sym \nabla)^{*}&A \sym \nabla {\mathfrak u}_6= {\rm i}\Bigl(X^{*}A \sym \nabla\bigl({\mathfrak u}_5^{(1)}+{\mathfrak u}_5^{(2)}\bigr)-(\sym \nabla)^{*}A\bigl\{X\bigl({\mathfrak u}_5^{(1)}+{\mathfrak u}_5^{(2)}\bigr)\bigr\}\Bigr)
\\[0.5em]
& 
-X^{*}A\bigl\{X\bigl({\mathfrak u}_4^{(1)}+{\mathfrak u}_4^{(2)}+{\mathfrak u}_4^{(3)}\bigr)\bigr\}
-|\chi|^4\bigl(0,0,m_3^{(2)}\bigr)^\top-|\chi|^4 \mathfrak{u}_2,
\quad\quad \int_Q {\mathfrak u}_6=0.
\end{aligned}
\end{equation}
The right-hand side of the equation for ${\mathfrak u}_6$ yields zero when tested with vectors $D=(D_1,D_2,0)^\top$, by symmetry properties as above. To see that it yields zero when tested with vectors $(0,0,D_3)^\top,$ we use the same reasoning as above ({\it cf.} (\ref{psi_id})):
\begin{equation*}
\begin{aligned}
\int_Q A &\bigl(\sym \nabla {\mathfrak u}_5^{(2)}+{\rm i}X{\mathfrak u}_4^{(3)} \bigr): \overline{{\rm i}X(0,0,D_3)^\top}
=\int_Q A \bigl(\sym \nabla {\mathfrak u}_5^{(2)}+ {\rm i}X{\mathfrak u}_4^{(3)}\bigr): \overline{ \sym \nabla \psi} \\[0.4em]
&=A^{\textrm{hom},1}_{\chi}m_3^{(2)}\,\overline{D}_3
+|\chi|^4\int_Q\bigl({\rm i}\chi_1 x_3 m_3^{(2)}, {\rm i}\chi_2 x_3m_3^{(2)},0\bigr)^\top \cdot \overline{({\rm i}\chi_1 x_3 D_3, {\rm i}\chi_2 x_3D_3,0)^\top}\\[0.4em]
&=-|\chi|^4 m_3^{(2)}\,\overline{D}_3-\int_Q A \bigl(\sym \nabla {\mathfrak u}_5^{(1)}+{\rm i}X{\mathfrak u}_4^{(1)}+{\rm i}X{\mathfrak u}_4^{(2)}\bigr): \overline{{\rm i}X(0,0,D_3)^\top},
\end{aligned}
\end{equation*}
where we use \eqref{sinisa1500}, \eqref{sinisa1104}. Thus, the problem \eqref{sinisa1022} has a unique solution, and 
\begin{equation} \label{julian112} 
\bigl\|{\mathfrak u}_6\bigr\|_{H^1(Q, {\mathbb C}^3)} \leq C |\chi|^6\|f
\|_{L^2(Q, {\mathbb C}^3)}.
\end{equation}
\vskip 0.2cm
{\it Step 4.} To complete the proof, we define the approximate solution 
\begin{align*}
	{\mathfrak U}&=\bigl(-{\rm i}\chi_1 x_3\bigl(m_3+m_3^{(1)}+m_3^{(2)}\bigr), -{\rm i}\chi_2 x_3\bigl(m_3+m_3^{(1)}+m_3^{(2)}\bigr),m_3+m_3^{(1)}+m_3^{(2)}\bigr)^\top
	+{\mathfrak u}_2+{\mathfrak u}_3^{(1)}+{\mathfrak u}_3^{(2)}\\[0.6em]
	& +{\mathfrak u}_4^{(1)}+{\mathfrak u}_4^{(2)}+{\mathfrak u}_4^{(3)}+{\mathfrak u}_5^{(1)}+{\mathfrak u}_5^{(2)}+{\mathfrak u}_6, 
\end{align*}
which satisfies 
$$ (\sym \nabla+{\rm i}X)^{*}A (\sym \nabla+{\rm i}X){\mathfrak U}+|\chi|^4 {\mathfrak U}=|\chi|^4\overline{e_\chi}\bigl(\widecheck{f\,}\!\!_1,\widecheck{f\,}\!\!_2,\widehat{f}_3\bigr)^\top+R_7, $$
where 
$$\bigl\| R_7\bigr\|_{H^{-1}_{\#}(Q)} \leq C|\chi|^7\|f
\|_{L^2(Q, {\mathbb C}^3)}.$$
The error $z:=u\overline{e_\chi} 
-{\mathfrak U}$  satisfies 
\begin{equation} 
\label{revision10000} 
(\sym \nabla+{\rm i}X)^{*} A (\sym \nabla+{\rm i}X)z+|\chi|^4 z=R_7.
\end{equation}
It is easy to see that, due to the estimate (\ref{bukal1000}), there exists $\rho_1>0$ such that for all $\chi\in Q'_{\rm r},$ $\vert\chi\vert\le \rho_1,$ one has 
$$
\bigl\| A({\rm sym}\nabla +{\rm i}X)z\bigr\|^2_{L^2(Q, {\mathbb C}^{3\times3})} \geq \frac{1}{2}\bigl(\|{\rm sym}\nabla z\|^2_{L^2(Q, {\mathbb C}^{3\times3})}- \|z\|^2_{L^2(Q, {\mathbb C}^3)}\bigr),
$$
and thus 
\begin{equation*}
\bigl\| A({\rm sym}\nabla+{\rm i}X)z\bigr\|^2_{L^2(Q, {\mathbb C}^{3\times3})}+|\chi|^4 \|z\|^2_{L^2(Q, {\mathbb C}^3)} \geq 
\frac{1}{4}|\chi|^4\bigl(\|{\rm sym}\nabla z\|^2_{L^2(Q, {\mathbb C}^{3\times3})}+\|z\|^2_{L^2(Q, {\mathbb C}^3)}\bigr).     
\end{equation*}
Testing the equation \eqref{revision10000} with $z$ therefore yields
\begin{equation*} 
\|z\|_{H^1(Q, {\mathbb C}^3)} \leq C|\chi|^3\|f
\|_{L^2(Q, {\mathbb C}^2)}.
\end{equation*}
Finally, as a consequence of  \eqref{sinisa1002}, \eqref{julian101}, \eqref{julian100}, \eqref{julian102}, \eqref{julian103}, \eqref{julian104}, \eqref{julian105}, \eqref{julian110}, \eqref{julian111}, \eqref{julian112}, we obtain 
\begin{equation}
\begin{aligned}
\|\overline{e_\chi}  u_\alpha+{\rm i}\chi_\alpha m_3 x_3
\|_{H^1(Q, {\mathbb C}^3)} & \leq C|\chi|^2\|f
\|_{L^2(Q, {\mathbb C}^3)} ,\qquad\alpha=1,2,
\\[0.7em]
\| \overline{e_\chi}  u_3-m_3 
\|_{H^1(Q, {\mathbb C}^3)} & \leq  C|\chi|\|f
\|_{L^2(Q, {\mathbb C}^3)},
\\[0.5em]
\bigl\|\overline{e_\chi}  u_\alpha+{\rm i}\chi_\alpha\bigl(m_3+m_3^{(1)}\bigr)x_3
-(\mathfrak{u}_2)_{\alpha}\bigr\|_{H^1(Q, {\mathbb C}^3)} & \leq C|\chi|^3\|f
\|_{L^2(Q, {\mathbb C}^3)} ,\qquad\alpha=1,2,
\\[0.5em]
\bigl\|\overline{e_\chi}  u_3-\bigl(m_3+m_3^{(1)}\bigr)\bigr\|_{H^1(Q, {\mathbb C}^3)} & \leq C|\chi|^2\|f
\|_{L^2(Q, {\mathbb C}^3)},
\end{aligned}
\label{korrre1}
\end{equation}
where $m_3 \in \mathbb{C}$ solves 
\eqref{sinisa1000revision}, $m_3^{(1)} \in \mathbb{C}$ solves \eqref{sinisa1500oo}, and $\mathfrak{u}_2$ solves \eqref{sinisa1001}. For $|\chi|>\rho_1$ the estimates (\ref{korrre1}) hold automatically (where the constant $C$ depends on $\rho_1$).
\begin{remark} \label{korekcijajed} 	
	Denote by $\widetilde m _3 \in \mathbb{C}$ the solution to the identity
	\begin{equation*}
	\label{sinisa1000revision000} 
	\bigl(\vert\chi\vert^{-4}A^{\textrm{\rm hom},1}_{\chi}+1\bigr)\widetilde m _3\,\overline{d_3} 
	=\int_Q|\chi|^{-1}\bigl(\widecheck{f\,}\!\!_1, \widecheck{f\,}\!\!_2\bigr)^\top \cdot \overline{e_\chi(-{\rm i}\chi_1 x_3 d_3, -{\rm i}\chi_2 x_3d_3)^\top}+\int_Q \widehat{f\,}\!\!_3 \cdot \overline{e_\chi d_3} \quad\quad\forall d_3\in{\mathbb C}.
	\end{equation*}
	Setting $d_3=m_3-\widetilde m _3$ in \eqref{sinisa1000revision} and \eqref{sinisa1000revision000}, subtracting one from the other, and using \eqref{propertiesrevision}, we obtain
	$$ |m_3-\widetilde m _3 | \leq  C|\chi|^2\|f
	\|_{L^2(Q, {\mathbb C}^3)}. $$ 
	This implies that all four estimates in \eqref{korrre1}
	are valid with $m_3$ replaced by $\widetilde{m}_3.$ In the same way we can replace $m_3^{(1)}$ in \eqref{sinisa1500oo} with $\widetilde m_3^{(1)}$ that satisfies 
	\begin{equation} 
	\label{marita2} 
	\bigl(A^{\textrm{\rm hom},1}_{\chi}+|\chi|^4\bigr)\widetilde m_3^{(1)}\,\overline{d_3} 
	=-\int_Q A \bigl(\sym \nabla {\mathfrak u}_4^{(1)}+{\rm i}X{\mathfrak u}_3^{(1)}\bigr): \overline{{\rm i}X(0,0,d_3)^\top}\quad\quad\forall d_3\in{\mathbb C},
	\end{equation} 
	thereby making an error of order $|\chi|^2$ at most. This also does not affect any of the estimates \eqref{korrre1}.
\end{remark} 

\subsubsection{Second invariant subspace} 
\label{section_drugaasimptotika}
\textbf{Asymptotic equation.} Here assume that the components $f_\alpha=\widehat{f}_{\alpha},$ $\alpha=1,2,$ are even in the $x_3$ variable, while $f_3=\widecheck{f\,}\!\!_3$ is odd, {\it cf.} Section \ref{section_secondsubspace}.
An approximating problem for $(m_1, m_2)^\top$ in the estimates (\ref{korrre1oo}) then takes the form
\begin{equation}
\begin{aligned}
\bigl(A^{\textrm{hom},2}_{\chi} +|\chi|^2\bigr)(m_1,m_2)^\top\cdot\overline{(d_1,d_2)^\top}
=|\chi|^2\int_Q\bigl(\widehat{f}_1, \widehat{f}_2\bigr)^\top \cdot \overline{e_\chi(d_1,d_2)^\top}
\quad\forall(d_1, d_2)^\top\in{\mathbb C}^2,
\end{aligned}
\label{sinisa3*second}
\end{equation}
and the following estimate holds:
\begin{equation} \label{revision1111} 
\bigl|(m_1, m_2)^\top\bigr| \leq C \bigl\| (\widehat{f}_1, \widehat{f}_2)^\top\bigr\|_{L^2(Q, \mathbb{C}^2)}. 
\end{equation} 
\vskip 0.1cm
\noindent\textbf{Approximation error estimates in the second invariant subspace.} 
{\it Step 1.}
First, we define ${\mathfrak u}_1 \in H^1_{\#} (Q, \mathbb{C}^3)$ so that 
\begin{equation}\label{revision1003} 
({\rm sym}\nabla)^{*}A\,{\rm sym}\nabla {\mathfrak u}_1=-({\rm sym}\nabla)^{*}A\Xi(\chi, m_1,m_2),\quad\quad\int_Q {\mathfrak u}_1=0.           
\end{equation} 
As in the case of $\mathfrak{u}_2$ for the first invariant subspace, we infer from \eqref{revision1111} that 
\begin{equation}\label{sinisa10revision} 
\|{\mathfrak u}_1\|_{H^1(Q, {\mathbb C}^3)} \leq C|\chi| \bigl\|(\widehat{f}_1,\widehat{f}_2)^\top\bigr\|_{L^2(Q, {\mathbb C}^2)}. 
\end{equation}
Notice that ${\mathfrak u}_1 $ has the symmetry properties
\begin{equation} 
\label{symmetryuagain2} 
({\mathfrak u}_{1})_{\alpha}(\cdot,-x_3)=({\mathfrak u}_{1})_{\alpha} (\cdot,x_3), \quad\alpha=1,2,\qquad 
\quad({\mathfrak u}_{1})_{3}(\cdot,-x_3)=-({\mathfrak u}_{1})_{3} (\cdot,x_3), \qquad x_3\in I.
\end{equation} 

In the remainder of the present section we write $f$ in place of $(\widehat{f}_1,\widehat{f}_2,\widecheck{f\,}\!\!_3)^\top.$ Define ${\mathfrak u}_2\in H^1_{\#} (Q, \mathbb{C}^3)$ that satisfies 
\begin{equation}
\begin{aligned} 
({\rm sym}\nabla)^{*}A\,{\rm sym}\nabla  {\mathfrak u}_2^{(1)}={\rm i}\bigl\{X^{*}A\,{\rm sym}\nabla {\mathfrak u}_1&-({\rm sym}\nabla)^{*}A(X{\mathfrak u}_1)+X^{*} A \Xi (\chi,m_1,m_2)\bigr\}\\[0.7em] 
&-|\chi|^2 (m_1,m_2,0)^\top +|\chi|^2 \overline{e_\chi}f,
\quad\quad \int_Q {\mathfrak u}_2^{(1)}=0.
\end{aligned} 
\label{sinisa101pon2}
\end{equation}
It follows from \eqref{sinisa3*second} that the right-hand side of \eqref{sinisa101pon2} yields zero when tested with constant vectors $(D_1,D_2,0)^\top$. Furthermore, it also vanishes when tested with vectors $(0,0,D_3)^\top,$ by the same argument as before, {\it i.e.} as a consequence of the symmetry properties of $\mathfrak{u}_1$ and $f.$ 
We also have the estimate
\begin{equation}\label{revision11111}
\bigl\|{\mathfrak u}_2^{(1)}\bigr\|_{H^1(Q, {\mathbb C}^3)} \leq C|\chi|^2\|f
\|_{L^2(Q, {\mathbb C}^3)}.
\end{equation}
\vskip 0.2cm
\noindent{\it Step 2.}
Second, we update $(m_1,m_2)^\top$ with $(m_1^{(1)},m_2^{(1)})^\top\in{\mathbb C}^2,$ defined as the solution to
\begin{equation}
\begin{aligned}
\bigl(A^{\textrm{hom},2}_{\chi}+|\chi|^2\bigr)\bigl(m_1^{(1)}&, m_2^{(1)}\bigr)^\top\cdot\overline{(d_1,d_2)^\top}
\\[0.5em]
&=-\int_QA\bigl(\sym \nabla \mathfrak{u}_2^{(1)}+{\rm i}X\mathfrak{u}_1\bigr): \overline{{\rm i}X(d_1,d_2,0)^\top}
\quad\forall(d_1, d_2)^\top\in{\mathbb C}^2.
\end{aligned}
\label{sinisa3*secondoo}
\end{equation}
The following estimate is a consequence of  \eqref{sinisa10revision}, \eqref{revision11111}:
\begin{equation} \label{julian1000}
\bigl|\bigl(m_1^{(1)}, m_2^{(1)}\bigr)^\top\bigr| \leq C|\chi|\|f 
\|_{L^2(Q, \mathbb{C}^3)}. 
\end{equation} 
Next, we define 
\begin{equation*}
({\rm sym}\nabla)^{*}A\,{\rm sym}\nabla {\mathfrak u}_2^{(2)}=-({\rm sym}\nabla)^{*}A\Xi\bigl(\chi, m_1^{(1)},m_2^{(1)}\bigr),\quad\quad\int_Q {\mathfrak u}_2^{(2)}=0,        
\end{equation*}
and, as a consequence of \eqref{julian1000}, we have 
\begin{equation*}
\bigl\|{\mathfrak u}_2^{(2)}\bigr\|_{H^1(Q, {\mathbb C}^3)} \leq C|\chi|^2\|f
\|_{L^2(Q, {\mathbb C}^3)}.
\end{equation*} 

Furthermore, we define $\mathfrak{u}_3\in H^1_{\#}(Q, \mathbb{C}^3)$ such that
\begin{equation*}
\begin{aligned} 
({\rm sym}\nabla)^{*}A\,{\rm sym}\nabla  {\mathfrak u}_3&= {\rm i}\Bigl(X^{*}A\,{\rm sym}\nabla \bigl({\mathfrak u}_2^{(1)}+{\mathfrak u}_2^{(2)}\bigr)-({\rm sym}\nabla)^{*}A\bigl\{X\bigl({\mathfrak u}_2^{(1)}+{\mathfrak u}_2^{(2)}\bigr)\bigr\}
+X^{*} A \Xi\bigl(\chi,m_1^{(1)},m_2^{(1)}\bigr)\Bigr)\\[0.7em] 
&
+{\rm i}X^{*}A(X\mathfrak{u}_1)
-|\chi|^2\bigl(m_1^{(1)},m_2^{(1)},0\bigr)^\top-|\chi|^2 \mathfrak{u}_1,\quad\quad\quad \int_Q {\mathfrak u}_3=0.
\end{aligned} 
\end{equation*}
The right-hand side of the above equation for ${\mathfrak u}_3$
yields zero when tested with vectors $D=(0,0,D_3)^\top$, by the symmetry properties (\ref{symmetryuagain2}). To see that it vanishes when tested with vectors $(D_1,D_2,0)^\top,$ we use the fact that
\begin{align*}
\int_Q A \bigl(\sym \nabla {\mathfrak u}_2^{(2)}&+\Xi\bigl(\chi,m_1^{(1)},m_2^{(1)}\bigr) \bigr): \overline{{\rm i}X(D_1,D_2,0)^\top}
=A^{\textrm{hom},2}_{\chi}\bigl(m_1^{(1)},m_2^{(1)}\bigr)^\top\cdot\overline{(D_1,D_2)^\top} \\[0.3em] 
&=-\int_Q A(\sym \nabla \mathfrak{u}_2^{(1)}+{\rm i}X\mathfrak{u}_1): \overline{{\rm i}X(D_1,D_2,0)^\top}
-|\chi|^2\bigl(m_1^{(1)},m_2^{(1)},0\bigr)^\top.
\end{align*}
Finally, we obtain the estimate
\begin{equation*}
\|{\mathfrak u}_3\|_{H^1(Q, {\mathbb C}^3)} \leq C|\chi|^3\|f
\|_{L^2(Q, {\mathbb C}^3)}.
\end{equation*} 
\\
{\it Step 3.}
Define the approximate solution 
$$
{\mathfrak U}=\bigl(m_1+m_1^{(1)},m_2+m_2^{(1)},0\bigr)^\top+{\mathfrak u}_1+{\mathfrak u}_2^{(1)}+{\mathfrak u}_2^{(2)} +{\mathfrak u}_3, $$
which clearly satisfies 
\begin{equation*}
\begin{aligned}
({\rm sym}\nabla+{\rm i}X)^{*}A ({\rm sym}\nabla+{\rm i}X){\mathfrak U}+|\chi|^2 {\mathfrak U}=|\chi|^2 \overline{e_\chi}f
+R_4,\\[0.75em]
\qquad \| R_4
\|_{H^{-1}_{\#}(Q, {\mathbb C}^3)} \leq C|\chi|^4\|f
\|_{L^2(Q, {\mathbb C}^3)}.
\end{aligned}
\end{equation*}
It follows that the error  $z:=u\overline{e_\chi} 
-{\mathfrak U}$ satisfies
\begin{equation*}   ({\rm sym}\nabla+{\rm i}X)^{*}A ({\rm sym}\nabla+{\rm i}X)z+|\chi|^2 z=R_4,
\end{equation*} 
and hence, in the same way as before, we obtain (see the argument between (\ref{revision10000}) and (\ref{korrre1}))
\begin{equation*} 
\|z\|_{H^1(Q, {\mathbb C}^3)} \leq C|\chi|^2\|f
\|_{L^2(Q, {\mathbb C}^3)}.
\end{equation*}
It follows that
\begin{equation}
\begin{aligned}
\|\overline{e_\chi}  u_\alpha-m_{\alpha}
\|_{H^1(Q, {\mathbb C}^3)} & \leq C|\chi|\|f
\|_{L^2(Q, {\mathbb C}^3)} ,\qquad\alpha=1,2,
\\[0.7em]
\| \overline{e_\chi}  u_3
\|_{H^1(Q, {\mathbb C}^3)} & \leq  C|\chi|\|f
\|_{L^2(Q, {\mathbb C}^3)},
\\[0.6em]
\bigl\|\overline{e_\chi}  u_\alpha-\bigl(m_{\alpha}+m_{\alpha}^{(1)}\bigr)
-(\mathfrak{u}_1)_{\alpha}\bigr\|_{H^1(Q, {\mathbb C}^3)} & \leq C|\chi|^2\|f
\|_{L^2(Q, {\mathbb C}^3)} ,\qquad\alpha=1,2,
\\[0.6em]
\|\overline{e_\chi}  u_3-(\mathfrak{u}_1)_{3}\|_{H^1(Q, {\mathbb C}^3)} & \leq C|\chi|^2\|f
\|_{L^2(Q, {\mathbb C}^3)}. 
\end{aligned}
\label{korrre1oo}
\end{equation}
\begin{remark} 
\label{f3_omitted}
1. One can set $\widecheck{f\,}\!\!_3=0$ when deriving
the first two estimates in (\ref{korrre1oo}), 
by virtue of the inequalities (\ref{additional_est}) and the fact that $m_1,$ $m_2$ 
do not depend on $\widecheck{f\,}\!\!_3,$ see (\ref{sinisa3*second}).

2. The method of this section is also applicable to the problem of the ``bulk'' 3D elasticity, and thus by a similar analysis we can recover the result of \cite{BirmanSuslina}. As already mentioned, the novelty in the plate case is the presence of eigenvalue of order $|\chi|^4.$ This requires further and more complex analysis, which is presented in Section 5.1.1.
\end{remark}

\subsection{General elasticity  tensor} 
\label{gen_est_section}

In this section we do not impose the symmetry conditions (\ref{symmetry}) and, guided by the results of Section \ref{structure}, consider two different scalings for the operator ${\mathcal A}_\chi$ with respect to the quasimomentum $\chi.$

\subsubsection{First scaling}
\label{first_scaling_section}

As the asymptotic procedure of Section \ref{planar} shows, in the case of the first scaling, which corresponds to the eigenvalue of order $|\chi|^4$ of ${\mathcal A}_\chi$ ({\it cf.} also Remark \ref{expl2}), it is convenient to scale the horizontal components of the ``Floquet-transformed" force density $f$, see (\ref{sinisa1000revision}), so our analysis below yields the estimates (\ref{korrre1oo2}), similar to (\ref{korrre1}). 
In Section \ref{gen_case_sec} we interpret these estimates in terms of the original body forces $F^h$ ({\it i.e.} the inverse Floquet transform of the ``unscaled" $f$) for the proof of Theorem \ref{main_result_general} ({\it cf.} Theorem \ref{main_result_first}, where for the case of the first invariant subspace we keep the general scaling $\varepsilon^{-\delta}$ in the first two components of the forces.) 

\vskip 0.2cm
\noindent\textbf{Asymptotic equation.} An approximating problem for \eqref{scalingjedan} takes the form 
\begin{equation}
\begin{aligned}
&\vert\chi\vert^{-4}A^{\textrm{hom}}_{\chi}m\cdot\overline{d} 
+\int_Q (m_1-{\rm i}\chi_1 x_3 m_3, m_2-{\rm i}\chi_2 x_3m_3,m_3)^\top \cdot \overline{(d_1-{\rm i}\chi_1 x_3 d_3, d_2-{\rm i}\chi_2 x_3d_3,d_3)^\top}
\\[0.3em] 
&=\int_Q|\chi|^{-1}\bigl(f\,\!\!_1, f\,\!\!_2\bigr)^\top \cdot \overline{e_\chi(d_1-{\rm i}\chi_1 x_3 d_3,d_2 -{\rm i}\chi_2 x_3d_3)^\top}+\int_Q f\,\!\!_3 \cdot \overline{e_\chi d_3} \quad\quad\forall d\in \mathbb{C}^3.
\end{aligned}
\label{korekcija11111}
\end{equation}
In the same way as in Remark \ref{korekcijajed}, it can be shown that (\ref{korekcija11111}) is equivalent to the identity
\begin{equation}
\label{marita10000}  
\bigl(\vert\chi\vert^{-4}A^{\textrm{hom}}_{\chi}+1\bigr)m\cdot\overline{d} 
=\int_Q|\chi|^{-1}\bigl(f\,\!\!_1, f\,\!\!_2\bigr)^\top \cdot \overline{e_\chi(d_1-{\rm i}\chi_1 x_3 d_3,d_2 -{\rm i}\chi_2 x_3d_3)^\top}+\int_Q f\,\!\!_3 \cdot \overline{e_\chi d_3} \quad\quad\forall d\in{\mathbb C}^3,
\end{equation}
in the sense that it does not affect any of the four inequalities \eqref{korrre1oo2}.
Using \eqref{muic1}, 
we easily obtain the estimates
\begin{equation}
\bigl|(m_1,m_2)^\top\bigr|\leq C|\chi| \|f\|_{L^2(Q, \mathbb{C}^3)},\qquad\quad
|m_3| \leq C \|f\|_{L^2(Q, \mathbb{C}^3)}.
\label{muic4}
\end{equation}
	
\vskip 0.2cm
\noindent{\bf Approximation error estimates for the first scaling.} 
\textbf{\it Step 1.}
In order to determine the ``corrector'' term ${\mathfrak u}_2$, we solve 
\begin{equation} 
({\rm sym}\nabla)^{*}A ({\rm sym}\nabla)  {\mathfrak u}_2=-({\rm sym}\nabla)^{*}A\bigl(\Xi(\chi, m_1,m_2)-{\rm i}x_3 \Upsilon(\chi,m_3)\bigr),\ \ \ \ 
{\mathfrak u}_2 \in H^1_{\#}(Q, \mathbb{C}^3), \quad\int_{Q} {\mathfrak u}_2=0,   
\label{otto1revision0000}       
\end{equation}
so that, due to 
\eqref{muic4},
the estimate
\begin{equation} 
\label{revision10020000}
\|{\mathfrak u}_2\|_{H^1(Q, {\mathbb C}^3)} \leq C\vert \chi \vert^2\|f\|_{L^2(Q, {\mathbb C}^3)}
\end{equation}
holds. 

Next, we define $\mathfrak{u}_3^{(1)} \in H^1_{\#}(Q, \mathbb{C}^3)$ as the solution to
\begin{equation}
\label{marita10002}
\begin{aligned}
({\rm sym}\nabla)^{*} &A\,{\rm sym}\nabla{\mathfrak u}_3^{(1)} ={\rm i}\bigl\{X^*A\,{\rm sym}\nabla  {\mathfrak u}_2
-({\rm sym}\nabla)^{*}A(X{\mathfrak u}_2)+X^*A\bigl(\Xi(\chi, m_1,m_2)-{\rm i}x_3 \Upsilon(\chi,m_3)\bigr)\bigr\}\\[0.6em] 
& +|\chi|^3\overline{e_\chi}(f_1,f_2, 0)^\top
-|\chi|^4(m_1-{\rm i}\chi_1 x_3 m_3,m_2-{\rm i}\chi_2 x_3 m_3,0)^\top,\quad
\quad\int_{Q} {\mathfrak u}_3^{(1)}=0.
\end{aligned}
\end{equation}
Notice that, due to \eqref{korekcija11111}, \eqref{otto1revision0000}, the right-hand side in the definition of ${\mathfrak u}_3^{(1)}$  vanishes when tested with constant vectors 
$(D_1,D_2,D_3)^\top$, where we again follow the reasoning of (\ref{psi_id}).
Due to  \eqref{muic4}, \eqref{revision10020000}, we also have the estimate
\begin{equation*}
\bigl\|\mathfrak{u}_3^{(1)}\bigr\|_{H^1(Q,\mathbb{C}^3)} \leq C|\chi|^3 \|f\|_{L^2(Q, \mathbb{C}^3)}. 
\end{equation*}
\vskip 0.2cm
\noindent{\it Step 2.} We update $m \in \mathbb{C}^3$ with $m^{(1)}=\bigl(m^{(1)}_1, m^{(1)}_2, m^{(1)}_3\bigr)^\top\in \mathbb{C}^3$ such that
\begin{equation}
\begin{aligned}
&A^{\textrm{hom}}_{\chi} m^{(1)}\cdot\overline{d}+|\chi|^4\int_Q\bigl(m_1^{(1)}-{\rm i}\chi_1 x_3 m_3^{(1)},m_2^{(1)}-{\rm i}\chi_2 x_3 m_3^{(1)}, m_3^{(1)}\bigr)^\top\cdot \overline{(d_1-{\rm i}\chi_1 x_3 d_3,d_2-{\rm i}\chi_2 x_3 d_3,d_3)^\top}\\[0.5em]
&=-\int_QA(\sym \nabla \mathfrak{u}_3^{(1)}+{\rm i}X\mathfrak{u}_2): \overline{\bigl(\Xi(\chi, d_1,d_2)-{\rm i}x_3 \Upsilon (\chi,d_3) \bigr)}
\qquad\forall d=(d_1, d_2,d_3)^\top\in{\mathbb C}^3.
\end{aligned}
\label{marita10001}
\end{equation}
It is straightforward to see that
\begin{equation*}
\bigl|\bigl(m_1^{(1)}, m_2^{(1)}\bigr)^\top\bigr|\leq C|\chi|^2 \|f\|_{L^2(Q,\mathbb{C}^3)}, \quad \bigl|m_3^{(1)}\bigr|\leq C|\chi| \|f\|_{L^2(Q,\mathbb{C}^3)}. 
\end{equation*}
Furthermore, we define ${\mathfrak u}_3^{(2)}\in H^1_{\#}(Q, \mathbb{C}^3)$ as the solution to the problem
\begin{equation*} 
({\rm sym}\nabla)^{*}A\,{\rm sym}\nabla {\mathfrak u}_3^{(2)}=-({\rm sym}\nabla)^{*}A\bigl(\Xi\bigl(\chi, m_1^{(1)},m_2^{(1)}\bigr)-{\rm i}x_3\Upsilon\bigl(\chi,m_3^{(1)}\bigr)\bigr), \quad 
\int_{Q} {\mathfrak u}_3^{(2)}=0,    
\end{equation*}
so that, in particular, the following bound holds:
\begin{equation*}
\bigl\| \mathfrak{u}_3^{(2)}\bigr\|_{H^1(Q,\mathbb{C}^3)}\leq C|\chi|^3 \|f\|_{L^2(Q, \mathbb{C}^3)}.
\end{equation*}
Next we define ${\mathfrak u}_4^{(1)} \in H^1_{\#}(Q, \mathbb{C}^3)$ as the solution to
\begin{equation}
\label{josss2}
\begin{aligned}
&({\rm sym}\nabla)^{*}A\,{\rm sym}\nabla{\mathfrak u}_4^{(1)}= {\rm i}\Bigl(X^*A\,{\rm sym}\nabla\bigl({\mathfrak u}_3^{(1)}+{\mathfrak u}_3^{(2)}\bigr)
-({\rm sym}\nabla)^{*}A\bigl\{X\bigl({\mathfrak u}_3^{(1)}+{\mathfrak u}_3^{(2)}\bigr)\bigr\}\\[0.7em]
&+X^*A\bigl(\Xi\bigl(\chi, m_1^{(1)},m_2^{(1)}\bigr)-{\rm i}x_3\Upsilon\bigl(\chi, m_3^{(1)}\bigr)\bigr)\Bigr)
+{\rm i}X^{*} A(X\mathfrak{u}_2)
+|\chi|^4\overline{e_\chi}(0,0,f_3)^\top
\\[0.6em]
&-|\chi|^4\bigl(m_1^{(1)}+(\mathfrak{u}_2)_{1}-{\rm i}\chi_1 x_3 m_3^{(1)},m_2^{(1)}+(\mathfrak{u}_2)_{2}-{\rm i}\chi_2 x_3 m_3^{(1)},m_3\bigr)^\top,
\qquad\quad\int_{Q} {\mathfrak u}_4^{(1)}=0.
\end{aligned}
\end{equation}
As before, the right-hand side of \eqref{josss2} vanishes when tested with constant vectors, 
in view of the identity \eqref{ljulj1}. Thus \eqref{josss2} has a unique solution, and 
\begin{equation*}
\bigl\| \mathfrak{u}_4^{(1)}\bigr\|_{H^1(Q,\mathbb{C}^3)}\leq C|\chi|^4 \|f\|_{L^2(Q, \mathbb{C}^3)}.
\end{equation*}
\vskip 0.2cm
\noindent{\it Step 3.}  We again update $m \in \mathbb{C}^3$ with $m^{(2)} \in \mathbb{C}^3$ such that 
\begin{equation*}
\begin{aligned}
&A^{\textrm{hom}}_{\chi} m^{(2)}\cdot\overline{d}+|\chi|^4\int_Q\bigl(m_1^{(2)}-{\rm i}\chi_1 x_3 m_3^{(2)},m_2^{(2)}-{\rm i}\chi_2 x_3 m_3^{(2)}, m_3^{(2)}\bigr)^\top\cdot \overline{(d_1-{\rm i}\chi_1 x_3 d_3,d_2-{\rm i}\chi_2 x_3 d_3,d_3)^\top}\\[0.8em]
&=-\int_Q A\bigl(\sym \nabla \mathfrak{u}_4^{(1)}+{\rm i}X\mathfrak{u}_3^{(1)}+{\rm i}X\mathfrak{u}_3^{(2)}\bigr): \overline{\bigl(\Xi(\chi, d_1,d_2)-{\rm i}x_3 \Upsilon (\chi,d_3) \bigr)}
\quad\quad\forall d=(d_1, d_2,d_3)^\top\in{\mathbb C}^3,
\end{aligned}
\end{equation*}
and, in particular, the following estimates hold:
\begin{equation*}
\bigl|\bigl(m_1^{(2)},m_2^{(2)}\bigr)^\top\bigr|\leq C|\chi|^3 \|f\|_{L^2(Q,\mathbb{C}^3)}, \quad\bigl|m_3^{(2)}\bigr| \leq C|\chi|^2 \|f\|_{L^2(Q,\mathbb{C}^3)}.
\end{equation*}
Furthermore, we define ${\mathfrak u}_4^{(2)} \in H^1_{\#}(Q, \mathbb{C}^3)$ such that
\begin{equation*} 
({\rm sym}\nabla)^{*}A\,{\rm sym}\nabla{\mathfrak u}_4^{(2)}=-({\rm sym}\nabla)^{*}A\bigl(\Xi\bigl(\chi, m_1^{(2)},m_2^{(2)}\bigr)-{\rm i}x_3\Upsilon\bigl(\chi,m_3^{(2)}\bigr)\bigr), \quad
\quad\int_{Q} {\mathfrak u}_4^{(2)}=0,  
\end{equation*}
and, particular, 
\begin{equation*}
\bigl\| \mathfrak{u}_4^{(2)}\bigr\|_{H^1(Q,\mathbb{C}^3)}\leq C|\chi|^4 \|f\|_{L^2(Q, \mathbb{C}^3)},
\end{equation*}
as well as ${\mathfrak u}_5^{(1)} \in H^1_{\#}(Q, \mathbb{C}^3)$ such that 
\begin{align*}
&({\rm sym}\nabla)^{*}A\,{\rm sym}\nabla{\mathfrak u}_5^{(1)} = {\rm i}\Bigl(X^*A\,{\rm sym}\nabla\bigl({\mathfrak u}_4^{(1)}+{\mathfrak u}_4^{(2)}\bigr)
-({\rm sym}\nabla)^{*}A\bigl\{X\bigl({\mathfrak u}_4^{(1)}+{\mathfrak u}_4^{(2)}\bigr)\bigr\}\\[0.6em]
&+X^*A\bigl(\Xi\bigl(\chi, m_1^{(2)},m_2^{(2)}\bigr)-{\rm i}x_3\Upsilon\bigl(\chi, m_3^{(2)}\bigr)\bigr)\Bigr)
+{\rm i}X^{*} A\bigl\{X\bigl(\mathfrak{u}_3^{(1)}+\mathfrak{u}_3^{(2)}\bigr)\bigr\}\\[0.5em] 
&-|\chi|^4\bigl(m_1^{(2)}+\bigl(\mathfrak{u}_3^{(1)}+\mathfrak{u}_3^{(2)}\bigr)_{1}-{\rm i}\chi_1 x_3 m_3^{(2)},m_2^{(2)}+\bigl(\mathfrak{u}_3^{(1)}+\mathfrak{u}_3^{(2)}\bigr)_{2}-{\rm i}\chi_2 x_3 m_3^{(2)},m_3^{(1)}\bigr)^\top,\quad 
\quad\int_{Q} {\mathfrak u}_5^{(1)}=0,
\end{align*}
and hence
\begin{equation*}
\bigl\| \mathfrak{u}_5^{(1)}\bigr\|_{H^1(Q,\mathbb{C}^3)}\leq C|\chi|^5 \|f\|_{L^2(Q, \mathbb{C}^3)}.
\end{equation*}
\vskip 0.2cm
\noindent{\it Step 4.} We update $m \in \mathbb{C}^3$ with $m^{(3)} \in \mathbb{C}^3$ 
in the same way as above by defining $\mathfrak{u}_5^{(2)}, \mathfrak{u}_6 \in H^1_{\#}(Q, \mathbb{C}^3)$. \\
\vskip 0.2cm
\noindent{\it Step 5.} 
In the same way as in Section \ref{section_prvaasimptotika}, it follows that
\begin{equation}
\begin{aligned}
\bigl\|\overline{e_\chi}  u_\alpha-(m_{\alpha}-{\rm i}\chi_\alpha m_3 x_3)
\bigr\|_{H^1(Q, {\mathbb C})} & \leq C|\chi|^2\|f\|_{L^2(Q, {\mathbb C}^3)} ,\qquad\alpha=1,2,
\\[0.5em]
\| \overline{e_\chi}  u_3-m_3
\|_{H^1(Q, {\mathbb C})} & \leq  C|\chi|\|f\|_{L^2(Q, {\mathbb C}^3)},
\\[0.5em]
\bigl\|\overline{e_\chi}  u_\alpha-\bigl(m_{\alpha}+m_{\alpha}^{(1)}-{\rm i}\chi_\alpha (m_3+m_3^{(1)})x_3+(\mathfrak{u}_2)_{\alpha}\bigr)\bigr\|_{H^1(Q, {\mathbb C})} & \leq C|\chi|^3\|f\|_{L^2(Q, {\mathbb C}^3)} ,\qquad\alpha=1,2,
\\[0.5em]
\bigl\|\overline{e_\chi}  u_3-\bigl(m_3+m_3^{(1)}\bigr)\bigr\|_{H^1(Q, {\mathbb C})} & \leq C|\chi|^2\|f\|_{L^2(Q, {\mathbb C}^3)}. 
\end{aligned}
\label{korrre1oo2}
\end{equation}

\subsubsection{Second scaling} 
\label{second_scaling_section}
\textbf{Asymptotic equation.} An approximating problem for \eqref{scalingdva} has the form
\begin{equation}
\begin{aligned} 
\vert\chi\vert^{-2}A^{\textrm{hom}}_{\chi}m\cdot\overline{d} 
&+\int_Q (m_1-{\rm i}\chi_1 x_3 m_3, m_2-{\rm i}\chi_2 x_3m_3,m_3)^\top \cdot \overline{(d_1-{\rm i}\chi_1 x_3 d_3, d_2-{\rm i}\chi_2 x_3d_3,d_3)^\top}
\\[0.3em] &=\int_Qf
 \cdot \overline{e_\chi(d_1-{\rm i}\chi_1 x_3 d_3,d_2 -{\rm i}\chi_2 x_3d_3,d_3)^\top}\qquad\forall d=(d_1, d_2,d_3)^\top\in{\mathbb C}^3,
\end{aligned}
\label{korekcija100pon0} 
\end{equation}
and the following estimate holds:
\begin{equation} \label{korekcija10pon0}
|m| \leq C\|f\|_{L^2(Q,\mathbb{C}^3)}.
\end{equation} 
Setting 
\[
d_1=d_2=0,\qquad d_3= m_3-\int_Q \overline{e_{\chi}} f_3, 
\]
in (\ref{korekcija100pon0}), we obtain 
\begin{equation*} 
\left|m_3-\int_Q \overline{e_{\chi}} f_3\right| \leq C|\chi|\|f\|_{L^2(Q,\mathbb{C}^3)}.
\end{equation*} 
In the same way as in Remark \ref{korekcijajed} it is shown that this equation is equivalent to 
\begin{equation}
\bigl(\vert\chi\vert^{-2}A^{\textrm{hom}}_{\chi}+1\bigr)m\cdot\overline{d} 
=\int_Qf
\cdot \overline{e_\chi(d_1-{\rm i}\chi_1 x_3 d_3, d_2-{\rm i}\chi_2 x_3 d_3,d_3)^\top}
\qquad\quad\forall d=(d_1, d_2,d_3)^\top\in{\mathbb C}^3.
\label{marita10003} 
\end{equation}
\\
\textbf{Approximation error estimates for the second scaling.} 
{\it Step 1.} 
We define ${\mathfrak u}_1 \in H^1_{\#} (Q, \mathbb{C}^3)$ as the solution to
\begin{equation}\label{revision1003000pon0} 
({\rm sym}\nabla)^{*}A\,{\rm sym}\nabla {\mathfrak u}_1=-({\rm sym}\nabla)^{*}A\bigl( \Xi(\chi, m_1,m_2)-{\rm i}x_3\Upsilon(\chi,m_3) \bigr),\quad {\mathfrak u}_1 \in H^1_{\#}(Q, \mathbb{C}^3), \quad\int_Q {\mathfrak u}_1=0,          
\end{equation} 
and infer from \eqref{korekcija10pon0} that 
\begin{equation*}
\|{\mathfrak u}_1\|_{H^1(Q, {\mathbb C}^3)} \leq C  |\chi|\|f\|_{L^2(Q, {\mathbb C}^3)}. 
\end{equation*}
Next we define ${\mathfrak u}_2^{(1)}\in H^1_{\#} (Q, \mathbb{C}^3)$ that satisfies 
\begin{equation}
\begin{aligned} 
& ({\rm sym}\nabla)^{*}A\,{\rm sym}\nabla  {\mathfrak u}_2^{(1)}={\rm i}\bigl\{X^{*}A\,{\rm sym}\nabla {\mathfrak u}_1-({\rm sym}\nabla)^{*}A(X{\mathfrak u}_1)+X^{*} A \bigl(\Xi (\chi,m_1,m_2)-{\rm i}x_3\Upsilon(\chi,m_3)\bigr)\bigr\}
\\[0.5em] 
&-|\chi|^2 (m_1-{\rm i}\chi_1x_3 m_3,m_2-{\rm i}\chi_2x_3 m_3,0)^\top +|\chi|^2 \overline{e_\chi}f
-|\chi|^2\biggl(0,0,\int_Q\overline{e_\chi}f_3\biggr)^{\top},
\quad \int_Q {\mathfrak u}_2^{(1)}=0.
\end{aligned} 
\label{sinisa101000}
\end{equation}

It follows from \eqref{korekcija100pon0} and \eqref{revision1003000pon0} that the right-hand side of \eqref{sinisa101000} vanishes when tested with constant vectors, and 
\begin{equation*}
\bigl\|{\mathfrak u}_2^{(1)}\bigr\|_{H^1(Q, {\mathbb C}^3)} \leq C|\chi|^2\|f\|_{L^2(Q, {\mathbb C}^3)}.
\end{equation*}
\\
{\it Step 2.} 
We update $m \in \mathbb{C}^3$ with $m^{(1)} \in \mathbb{C}^3,$ which we define to satisfy
\begin{equation}
\begin{aligned}
&A^{\textrm{hom}}_{\chi} m^{(1)}\cdot \overline{d}+|\chi|^2\int_Q\bigl(m_1^{(1)}-{\rm i}\chi_1 x_3 m_3^{(1)},m_2^{(1)}-{\rm i}\chi_2 x_3 m_3^{(1)},m_3^{(1)}\bigr)^\top\cdot \overline{(d_1-{\rm i}\chi_1 x_3 d_3,d_2-{\rm i}\chi_2 x_3 d_3,d_3)^\top}
\\[0.5em] &=-\int_Q A(\sym \nabla \mathfrak{u}_2^{(1)}+{\rm i}X\mathfrak{u}_1): \overline{\bigl(\Xi (\chi,d_1,d_2)-{\rm i}x_3\Upsilon(\chi,d_3)\bigr)}
 \quad\quad\forall d=(d_1, d_2,d_3)^\top\in{\mathbb C}^3.
\end{aligned}
\label{sinisa3*secondoooo}
\end{equation}
First, setting $d=m^{(1)}$ in (\ref{sinisa3*secondoooo}), we obtain  
\begin{equation}
	\bigl|m^{(1)}\bigr|\leq C |\chi| \|f\|_{L^2 (Q,\mathbb{C}^3)},
	\label{first_m_est}
\end{equation}	
Second, setting $d_1=d_2=0,$ $d_3=m_3^{(1)}$ in \eqref{sinisa3*secondoooo} and using (\ref{first_m_est}), we additionally obtain 
\begin{equation*}
	 \bigl|m_3^{(1)}
	\bigr| \leq C |\chi|^2\|f\|_{L^2(Q, \mathbb{C}^3)}. 
\end{equation*}	


Next, we define ${\mathfrak u}_2^{(2)} \in H^1_{\#}(Q, \mathbb{C}^3)$ so that
\begin{equation*}
({\rm sym}\nabla)^{*}A\,{\rm sym}\nabla {\mathfrak u}_2^{(2)}=-({\rm sym}\nabla)^{*}A\bigl( \Xi\bigl(\chi, m_1^{(1)},m_2^{(2)}\bigr)-{\rm i}x_3\Upsilon\bigl(\chi,m_3^{(1)}\bigr) \bigr),
\quad\int_Q {\mathfrak u}_2^{(2)}=0,          
\end{equation*} 
and, in particular, the estimate
\begin{equation*}
\bigl\|\mathfrak{u}_2^{(2)}\bigr\|_{H^1(Q, \mathbb{C}^3)} \leq C |\chi|^2 \|f\|_{L^2(Q,\mathbb{C}^3)}
\end{equation*}
holds.

Finally, consider ${\mathfrak u}_3 \in H^1_{\#}(Q, \mathbb{C}^3)$ such that
\begin{equation}
\begin{aligned} 
&({\rm sym}\nabla)^{*}A\,{\rm sym}\nabla{\mathfrak u}_3= {\rm i}\Bigl(X^{*}A\,{\rm sym}\nabla\bigl({\mathfrak u}_2^{(1)}+{\mathfrak u}_2^{(2)}\bigr)-({\rm sym}\nabla)^{*}A\bigl\{X\bigl( {\mathfrak u}_2^{(1)}+{\mathfrak u}_2^{(2)}\bigr)\bigr\}\\[0.7em]
&\hspace{+24ex} +X^{*} A \bigl(\Xi\bigl(\chi,m_1^{(1)},m_2^{(1)}\bigr)-{\rm i}x_3\Upsilon\bigl(\chi,m_3^{(1)}\bigr)\bigr)\Bigr)
+{\rm i}X^{*} A(X \mathfrak{u}_1)
\\[0.2em]  
&-|\chi|^2 \left(m_1^{(1)}-{\rm i}\chi_1x_3 m_3^{(1)}+({\mathfrak u}_1)_1, m_2^{(1)}-{\rm i}\chi_2x_3 m_3^{(1)}
+({\mathfrak u}_1)_{2}, m_3+({\mathfrak u}_1)_{3}-\int_Q\overline{e_\chi}f_3
\right)^\top,
\quad \int_Q {\mathfrak u}_3=0.
\end{aligned} 
\label{julian2}
\end{equation}
It is easy to see that, as a consequence of \eqref{korekcija100pon0}, \eqref{sinisa101000} and \eqref{sinisa3*secondoooo}, the right-hand side of \eqref{julian2} vanishes when tested with constant vectors. 
 Thus \eqref{julian2} has a unique solution, and 
$$ 
\|{\mathfrak u}_3\|_{H^1(Q, {\mathbb C}^3)} \leq C|\chi|^3\|f\|_{L^2(Q, {\mathbb C}^3)}.  
$$
\vskip 0.2cm
\noindent{\it Step 3.} Similarly to Section \ref{section_drugaasimptotika} we infer that for all $\chi\in Q'_{\rm r}$
\begin{equation*}
\begin{aligned}
\|\overline{e_\chi}  u-m\|_{H^1(Q, {\mathbb C^3})} & \leq C|\chi||\|f\|_{L^2(Q, {\mathbb C}^3)},
\\[0.6em]
\bigl\|\overline{e_\chi}  u_\alpha-\bigl(m_{\alpha}+m_{\alpha}^{(1)}-{\rm i}\chi_\alpha m_3 x_3+(\mathfrak{u}_1)_{\alpha}\bigr)\bigr\|_{H^1(Q, {\mathbb C})} & \leq C|\chi|^2|\|f\|_{L^2(Q, {\mathbb C}^3)} ,\qquad\alpha=1,2,
\\[0.4em]
\bigl\|\overline{e_\chi}  u_3-\bigl(m_3+m_3^{(1)}+(\mathfrak{u}_1)_{3}\bigr)\bigr\|_{H^1(Q, {\mathbb C})} & \leq C|\chi|^2|\|f\|_{L^2(Q, {\mathbb C}^3)}. 
\end{aligned}
\end{equation*}
\begin{remark} 
\label{rem53}
1.	The above procedure allows us to continue the asymptotic expansion up to any order in $|\chi|.$ 
Since we terminated the asymptotic procedure at the second corrector, the above estimates are optimal for the eigenspace corresponding to the eigenvalue of ${\mathcal A}_\chi$ of order $|\chi|^2.$ 
Optimality for the eigenspace of  the eigenvalue 
of order $|\chi|^4$ would require higher-order correctors. These are non-standard in the theory of homogenisation and can be obtained by achieving an error of order $|\chi|^4$ in the mentioned subspace, as an analysis similar to that presented in the next chapter shows.


2. The reason for not including the standard corrector $({\mathfrak u}_2)_3$
 in the last estimate in (\ref{korrre1}) is that already the $|\chi|^2$-error requires to include a non-standard corrector $m_3^{(1)},$ and incorporating $({\mathfrak u}_2)_3$ would only make sense if $m_3^{(2)}$ were included as well (thereby resulting in a $|\chi|^3$ error for the third component), which we have refrained from doing.
 
3. In summary, our method not only produces optimal operator-norm estimates for the difference between the solutions to the original equation and the standard plate equations, but it also provides a way to compute correctors that improve these estimates (up to the errors that result from neglecting the eigenspaces the higher-order spectrum of order one).
\end{remark}
\section{Norm-resolvent and energy estimates for the infinite plate}
\label{normresolvent}
Here we interpret the error estimates obtained above in terms the original family of elasticity operators ${\mathcal A}^\varepsilon,$ $\varepsilon>0,$ in $L^2(\Pi^h , {\mathbb R}^3),$ $\Pi^h ={\mathbb R}^2\times(-h/2, h/2),$ given by the differential expressions ({\it cf.} (\ref{original_identity}), where $h=\varepsilon$)
$$ 
({\rm sym}\nabla)^{*}A(x_1/\varepsilon, x_2/\varepsilon)\,{\rm sym}\nabla.  
$$
This will complete the proofs of our results formulated in Section \ref{formulation}. For each value $\gamma\ge-2$ we use the representation (\ref{vonNeumann1}) in Section \ref{strategynovelties}. 
We shall discuss separately the cases of planar-symmetric and general elasticity tensors. 
In addition to the proofs of the main results,
we will address the $L^2\to H^1$ estimates (``energy estimates") and higher-order $L^2\to L^2$ estimates.  

Notice that while the $|\chi|$-uniform estimates we have proved in Section \ref{asymptotic_proc} do not yield small errors for finite values of $|\chi|,$ they will suffice for the proof of the stated operator-norm estimates, as finite values of $\chi$ make a controllably small contribution into the resolvent asymptotics. This will be evident from the optimality analysis with respect to $\chi$ in each of the two parts of Section \ref{sym_case_res} and in Section \ref{gen_case_sec}.

\subsection{Case of planar-symmetric elasticity tensor}

\label{sym_case_res}

Following a convention similar to that of Section \ref{apriori_sec}, we attach the overscripts $\,\widehat{}\,$ and $\widecheck{\phantom{a}}$ to the force components that are even and odd in $x_3,$ respectively. 

\subsubsection{First invariant subspace}
\label{app_first}


\noindent\textbf{Proof of $L^2\to L^2$ estimate (\ref{first_subspace_est}).} It follows from the analysis of Section \ref{structure}, that in the for each $\chi\in Q'_{\rm r},$ the largest eigenvalue of the ``fibre" $\left(\varepsilon^{-\gamma-2}  \mathcal{A}_\chi+I\right)^{-1}$
is of order 
$$
\left(\max\left\{\varepsilon^{-\gamma-2}|\chi|^4,1\right\}\right)^{-1},
$$ 
and all other eigenvalues are at most of order $\bigl(\max\left\{\varepsilon^{-\gamma-2},1\right\}\bigr)^{-1}\sim\varepsilon^{\gamma+2},$ $\varepsilon\to0.$
In deriving the sought approximation for $(\varepsilon^{-\gamma}{\mathcal A}^\varepsilon+I)^{-1},$ we neglect the spectral projections of 
$(\varepsilon^{-\gamma-2} \mathcal{A}_\chi+I)^{-1}$ (see (\ref{vonNeumann1})) onto its eigenspaces corresponding to the latter eigenvalues,
which results in an overall error of order $\varepsilon^{\gamma+2-\delta}.$ 

Furthermore, it follows from the first two estimates in \eqref{korrre1} that in the eigenspace of $(\varepsilon^{-\gamma-2} \mathcal{A}_\chi+I)^{-1}$  corresponding to its largest eigenvalue, the leading order of the approximation error 
is given by 
\begin{equation}
C|\chi|^p \left(\max\left\{\varepsilon^{-\gamma-2}|\chi|^4,1\right\}\right)^{-1} \max\bigl\{\varepsilon^{-\delta}|\chi|,1\bigr\},
\label{two_orders}
\end{equation}
where $p=2$ and $p=1$ for the first two components and 
 third component, respectively. Indeed, for each $\varepsilon>0,$ $\chi\in Q'_{\rm r}\setminus\{0\},$ we write
\[
\varepsilon^{-\gamma-2}  \mathcal{A}_{\chi}+I=\varepsilon^{-\gamma-2}|\chi|^4\bigl(|\chi|^{-4} \mathcal{A}_{\chi} \bigr)+I
\]
and notice that for all $\eta>0$ the function 
\begin{equation}
f_{\varepsilon,\chi}(\zeta):=\bigl(\varepsilon^{-\gamma-2}|\chi|^4\zeta+1\bigr)^{-1},\quad  \Re(\zeta)>0,
\label{function_f}
\end{equation}
is bounded in the half-plane $\{\zeta: \Re(\zeta)>\eta\}$ by the expression $\left(\max\{\varepsilon^{-\gamma-2}|\chi|^4\eta,1 \}\right)^{-1},$ which depends on $\eta$ but is independent of 
$\varepsilon.$ 
Furthermore, we use
the Riesz integral representation
\begin{equation}
 \label{marita1} 
{\mathcal P}_1^\chi\left(\varepsilon^{-\gamma-2}  \mathcal{A}_{\chi}+I\right)^{-1}{\mathcal P}_1^\chi=\frac{1}{2\pi{\rm i}} \oint_{\Gamma}f_{\varepsilon,\chi} (\zeta) \left(\zeta I-|\chi|^{-4} \mathcal{A}_{\chi}\right)^{-1}d\zeta,
\end{equation} 
where ${\mathcal P}_1^\chi$ is the projection onto the eigenspace of the lowest eigenvalue of ${\mathcal A}_\chi.$
Here we choose $\rho_2>0$ and a contour $\Gamma\subset\{\zeta: \Re(\zeta)>\eta\}$ 
with the following three properties satisfied for all $\chi\in Q'_{\rm r}\setminus\{0\},$ $|\chi|<\rho_2$ (see Lemma \ref{revlemma}): 

a) The domain bounded by $\Gamma$ contains the lowest eigenvalue of the operator $|\chi|^{-4}\mathcal{A}_{\chi};$ 

b) The same domain contains  none of the higher eigenvalues of $|\chi|^{-4}\mathcal{A}_{\chi};$ 

c) The distance from $\Gamma$ to the spectrum of  $|\chi|^{-4} \mathcal{A}_{\chi}$ is bounded below by a positive constant.

The claim concerning the error bound (\ref{two_orders}) follows by approximating the resolvent under the integral sign 
in \eqref{marita1}, uniformly on $\Gamma,$ using the first two estimates in \eqref{korrre1} and Remark \ref{korekcijajed}. The existence of a value $\rho_2$ with the required properties follows from the bounds (\ref{bbb0})--(\ref{bbb4}), and  $\rho_2$ in turn determines an appropriate choice of $\eta$ above. (Notice that analogous estimates ensure that for a suitable choice of $\rho_2$ the contour $\Gamma$ possesses the properties a, c above in relation to the eigenvalue of ${\mathcal A}_\chi^{{\rm hom},1},$ whenever $|\chi|<\rho_2.$) For quasimomenta $\chi$ such that $|\chi|>\rho_2$ the required bound holds automatically.

It is easily seen that the maximum in (\ref{two_orders}) is attained for 
$\varepsilon^{-\gamma-2}|\chi|^4 \sim 1.$ 
Therefore, the total error is of order
$$
\varepsilon^{p(\gamma+2)/4} \max\bigl\{\varepsilon^{(\gamma+2)/4-\delta},1\bigr\},
$$ 
and the estimate (\ref{first_subspace_est}) follows by applying the inverse Floquet transform.


\vskip 0.2cm

\noindent\textbf{$L^2\to H^1$ estimate.} To obtain an approximation for the resolvents $(\varepsilon^{-\gamma}{\mathcal A}^\varepsilon+I)^{-1}$  as mappings from 
$L^2(\Pi^h , {\mathbb R}^3)$ to $H^1(\Pi^h , {\mathbb R}^3),$  we use the third and fourth estimates in \eqref{korrre1} and include the corrector $\mathfrak{u}_2$ in the first two components.\footnote{Notice that neglecting the spectral projection onto the eigenspace of an order-one eigenvalue of ${\mathcal A}_\chi$ 
 leads to an error of order $\varepsilon^{-1}$ in the energy norm.} To this end, for each $\chi\in Q_{\rm r}'$ consider the 
bounded operator
\begin{equation}
 {\mathfrak K}_1^{\rm b}: L^2(\Pi)\ni U\mapsto\varepsilon^{-2}{\mathcal F }_\varepsilon^*{\mathfrak S}_{\rm b}^1\left(\varepsilon^{-\gamma-2} A^{\textrm{hom},1}_{\chi}+I \right)^{-1}\!{\mathcal F}_\varepsilon{U},
 \label{correctorb}
 \end{equation}
where ${\mathfrak S}_{\rm b}^1$ takes $m_3 \in \mathbb{C}$ to the solution of \eqref{sinisa1001}.
 Notice that ${\mathfrak K}^{\rm b}_1$ is the standard corrector in the theory of homogenisation, see {\it e.g.} \cite{BLP}. The overall error of neglecting the spectral projections of $(\varepsilon^{-\gamma}{\mathcal A}^\varepsilon+I)^{-1}$ onto the eigenspaces corresponding to higher eigenvalues of $(\varepsilon^{-\gamma-2}{\mathcal A}_\chi+I)^{-1}$ in the direct integral (\ref{vonNeumann1}) is of order $\varepsilon^{\gamma+1-\delta}$ and, with the corrector 
${\mathfrak K}^{\rm b}_1$  included,
the $H^1$ error in each fibre\footnote{Recall that under the Floquet transform, the operator of differentiation with respect to $x_\alpha,$ $\alpha=1,2,$ is mapped to 
$\varepsilon^{-1}e_\chi(\partial_{y_{\alpha}}+{\rm i}\chi_\alpha)\overline{e_\chi}.$}
is bounded by
\begin{equation}
C\max\bigl\{|\chi|^p,\varepsilon^{-1}|\chi|^{p+1}\bigr\} \left(\max\left\{\varepsilon^{-\gamma-2}|\chi|^4,1\right\}\right)^{-1} \max\bigl\{\varepsilon^{-\delta} |\chi|,1\bigr\},
\label{two_ordersH1}
\end{equation}
where $p=2$ and $p=1$ for the first two components and 
 third component, respectively. As in the case of (\ref{two_orders}), the maximum in (\ref{two_ordersH1}) is attained for
$\varepsilon^{-\gamma-2}|\chi|^4  \sim 1,$ 
and therefore the overall approximation error is of order
\[
\max\bigl\{\varepsilon^{p(\gamma+2)/4},\varepsilon^{(p+1)(\gamma+2)/4-1}\bigr\} \max\bigl\{\varepsilon^{(\gamma+2)/4-\delta},1\bigr\}. 
\]
As a result, there exists $C>0$ such that
\begin{equation*}
  \begin{aligned}
  &\Bigl\|R^hP_{\alpha} \left(\varepsilon^{-\gamma} \mathcal{A}^{\varepsilon}+I\right)^{-1}\bigl(\varepsilon^{-\delta}\widecheck{F\,}\!\!_1,\varepsilon^{-\delta}\widecheck{F\,}\!\!_2, \widehat{F}_3\bigr)^\top
  \\[0.5em]& \hspace{+5ex}
  +\Re\Bigl\{\left(\varepsilon x_3\partial_{\alpha}\left(  \varepsilon^{-\gamma+2}\mathcal{A}^{{\rm hom}, 1} +I \right)^{-1}-\varepsilon^2P_{\alpha} {\mathfrak K}^{\rm b}_1\right)\bigl( SR^h\widehat{F}_3+\varepsilon^{-\delta}S_1R^h\widecheck{F\,}\!\!_1 +\varepsilon^{-\delta}S_2R^h\widecheck{F\,}\!\!_2\bigr)\Bigr\}\Bigr\|_{H^1(\Pi)} 
  \\[0.5em] 
  &  \hspace{+10ex}  
  \leq C\max\bigl\{\varepsilon^{(\gamma+2)/2}, \varepsilon^{3(\gamma+2)/4-1}\bigr\}
  \max\bigl\{\varepsilon^{(\gamma+2)/4-\delta},1\bigr\}
  \Bigl\|R^h\bigl(\widecheck{F}_1,\widecheck{F}_2, \widehat{F}_3\bigr)^\top\Bigr\|_{L^2(\Pi ,\mathbb{R}^3)} , \quad  \alpha=1,2, \\[1.3em]
  &\Bigl\|R^hP_3 \left(\varepsilon^{-\gamma} \mathcal{A}^{\varepsilon}+I\right)^{-1}\bigl(\varepsilon^{-\delta}\widecheck{F\,}\!\!_1,\varepsilon^{-\delta}\widecheck{F\,}\!\!_2, \widehat{F}_3\bigr)^\top\\[0.5em]
  &\hspace{+10ex}
 -\Re\Bigl\{\left(  \varepsilon^{-\gamma+2}\mathcal{A}^{{\rm hom}, 1} +I \right)^{-1}
  \bigl( SR^h\widehat{F}_3+\varepsilon^{-\delta}S_1R^h\widecheck{F\,}\!\!_1 +\varepsilon^{-\delta}S_2R^h\widecheck{F\,}\!\!_2\bigr)\Bigr\}\Bigr\|_{H^1(\Pi)} \\[0.3em] 
  &  \hspace{+20ex}  \leq C \max\bigl\{\varepsilon^{(\gamma+2)/4}, \varepsilon^{(\gamma+2)/2-1}\bigr\} 
  \max\bigl\{\varepsilon^{(\gamma+2)/4-\delta},1\bigr\}\Bigl\|R^h\bigl(\widecheck{F\,}\!\!_1,\widecheck{F\,}\!\!_2, \widehat{F}_3\bigr)^\top\Bigr\|_{L^2(\Pi , \mathbb{R}^3)}. 
  \end{aligned}
  \end{equation*}
\\
\textbf{Higher-order $L^2\to L^2$ estimate.} To obtain a higher-order approximation in $L^2(\Pi^h , \mathbb{R}^3)$, one defines a suitable corrector ${\mathfrak K}^{\rm b}_2,$ using \eqref{sinisa1000revision}, \eqref{sinisa1004}, \eqref{sinisa1010}, and \eqref{marita2}, for all $\chi\in Q_{\rm r}'$.  This corrector appears to be unknown in the homogenisation theory, although its derivation is similar to the elliptic argument of \cite{BirmanSuslina}.
We leave the details to the interested reader, noting that the associated error is bounded by 
\begin{equation}
C|\chi|^p\left(\max\left\{\varepsilon^{-\gamma-2}|\chi|^4,1\right\}\right)^{-1} \max\bigl\{\varepsilon^{-\delta} |\chi|,1\bigr\}.
\label{l2max}
\end{equation} 
where $p=3$ and $p=2$ for the first two components and 
third component, respectively. (Recall that the error of neglecting higher eigenvalues is of order $\varepsilon^{\gamma+2-\delta}$ in terms of the resolvent $(\varepsilon^{-\gamma}{\mathcal A}^\varepsilon+I)^{-1}.$) It is easily seen that the maximum in (\ref{l2max}) is attained for  
$\varepsilon^{-\gamma-2}|\chi|^4\sim 1,$
and hence the overall error is of order
$$
\varepsilon^{p(\gamma+2)/4} \max\bigl\{\varepsilon^{(\gamma+2)/4-\delta},1\bigr\}.
$$  
As a result, there exists $C>0$ such that
\begin{equation*}
  \begin{aligned}
  &\Bigl\|R^hP_{\alpha} \left(\varepsilon^{-\gamma} \mathcal{A}^{\varepsilon}+I\right)^{-1}\bigl(\varepsilon^{-\delta}\widecheck{F\,}\!\!_1,\varepsilon^{-\delta}\widecheck{F\,}\!\!_2, \widehat{F}_3\bigr)^\top\\[0.8em]
  &\hspace{+8ex}+\Re\Bigl\{\left(\varepsilon x_3\partial_{\alpha}\left(  \varepsilon^{-\gamma+2}\mathcal{A}^{{\rm hom}, 1} +I \right)^{-1}-\varepsilon^2P_{\alpha} {\mathfrak K}^{\rm b}_1\right)\bigl(SR^h\widehat{F}_3+ \varepsilon^{-\delta}S_1R^h\widecheck{F\,}\!\!_1+\varepsilon^{-\delta}S_2R^h\widecheck{F\,}\!\!_2 \bigr)\\[0.8em]
  &\hspace{+16ex}
  +\varepsilon x_3\partial_{\alpha}{\mathfrak K}^{\rm b}_2 R^h\bigl(\widecheck{F\,}\!\!_1,\widecheck{F\,}\!\!_2, \widehat{F}_3\bigr)^\top\Bigr\}\Bigr\|_{L^2(\Pi)}\\[0.8em]
  &\hspace{+24ex}\leq C \varepsilon^{3(\gamma+2)/4} 
  \max\bigl\{\varepsilon^{(\gamma+2)/4-\delta},1\bigr\}\Bigl\|R^h(\widecheck{F}_1,\widecheck{F}_2, \widehat{F}_3)^\top\Bigr\|_{L^2(\Pi ,\mathbb{R}^3)} , \quad  \alpha=1,2,
  \end{aligned}
  \end{equation*}	 
  \begin{equation*}
  \begin{aligned}
  &\Bigl\|R^hP_3 \left(\varepsilon^{-\gamma}\mathcal{A}^{\varepsilon}+I\right)^{-1}\bigl(\varepsilon^{-\delta}\widecheck{F\,}\!\!_1,\varepsilon^{-\delta}\widecheck{F\,}\!\!_2, \widehat{F}_3\bigr)^\top\\[0.9em]
  &\hspace{+12ex}-\Re\Bigl\{\left( \left(\varepsilon^{-\gamma+2}\mathcal{A}^{{\rm hom}, 1} +I \right)^{-1}-\varepsilon^2P_3 {\mathfrak K}^{\rm b}_1  \right) \bigl(SR^h\widehat{F}_3+\varepsilon^{-\delta}S_1 R^h\widecheck{F\,}\!\!_1+\varepsilon^{-\delta}S_2R^h\widecheck{F\,}\!\!_2  \bigr)
 \\[0.9em] 
 & \hspace{+24ex} 
  - {\mathfrak K}^{\rm b}_2R^h\bigl(\widecheck{F\,}\!\!_1,\widecheck{F\,}\!\!_2, \widehat{F}_3\bigr)^\top\Bigr\}\Bigr\|_{L^2(\Pi)}\\[0.8em]
  &\hspace{+36ex}\leq C \varepsilon^{(\gamma+2)/2} \max\bigl\{\varepsilon^{(\gamma+2)/4-\delta},1\bigr\}\Bigl\|R^h(\widecheck{F\,}\!\!_1,\widecheck{F\,}\!\!_2, \widehat{F}_3)^\top\Bigr\|_{L^2(\Pi ,\mathbb{R}^3)}.
  \end{aligned}
  \end{equation*}

\subsubsection{Second invariant subspace}
\label{app_second}

\noindent\textbf{Proof of $L^2\to L^2$ estimate (\ref{second_subspace_est}).} For each $\chi\in Q_{\rm r}'$ in the direct integral (\ref{vonNeumann1}),
the largest two eigenvalues of 
$\left(\varepsilon^{-\gamma-2} \mathcal{A}_\chi+I\right)^{-1}$ 
are of order 
\begin{equation}
\left(\max\left\{\varepsilon^{-\gamma-2}|\chi|^2,1\right\}\right)^{-1},
\label{largest_order}
\end{equation}
and all other eigenvalues are at least of order 
$\left(\max\left\{\varepsilon^{-\gamma-2},1\right\}\right)^{-1}\sim\varepsilon^{\gamma+2},$ $\varepsilon\to0.$
 In the approximation on each fibre we neglect the spectral projections of $(\varepsilon^{-\gamma-2} \mathcal{A}_\chi+I)^{-1}$ onto the eigenspaces corresponding to the latter 
eigenvalues, which results in an error of order $\varepsilon^{\gamma+2}$ for the operators $(\varepsilon^{-\gamma}{\mathcal A}^\varepsilon+I)^{-1}.$ 
Furthermore, the expression (\ref{largest_order}) and the first two estimates in \eqref{korrre1oo} imply that in the eigenspace corresponding to the largest two eigenvalues of  $(\varepsilon^{-\gamma-2} \mathcal{A}_\chi+I)^{-1},$ the approximation error is of order
\[
|\chi|\left(\max\left\{\varepsilon^{-\gamma-2}|\chi|^2,1\right\}\right)^{-1}.
\] 
As in Section \ref{app_first}, introducing, for each $\varepsilon>0,$ $\chi\in Q'_{\rm r}\setminus\{0\},$ the function 
\begin{equation}
g_{\varepsilon,\chi}(\zeta):=\bigl(\varepsilon^{-\gamma-2}|\chi|^2\zeta+1\bigr)^{-1},\quad  \Re(\zeta)>0,
\label{function_g}
\end{equation}
and integrating the expression $g_{\varepsilon,\chi} (\zeta)\left(\zeta I-|\chi|^{-2} \mathcal{A}_{\chi}\right)^{-1}$ over  a contour in the $\zeta$-plane that contains both order-one eigenvalues of $|\chi|^{-2} \mathcal{A}_{\chi},$
{\it cf.} (\ref{marita1}), we infer that the overall error is of order
$\varepsilon^{(\gamma+2)/2},$ and the estimate (\ref{second_subspace_est}) follows.
\\
\vskip 0.01cm
\noindent\textbf{$L^2\to H^1$ estimate.} Similarly to the analysis of the $L^2\to H^1$ error in the case of the first invariant subspace, for each $\chi\in Q_{\rm r}'$ we define the corrector 
 \begin{equation}
 {\mathfrak K}_1^{\rm m}: L^2(\Pi ,\mathbb{R}^2)\ni U\mapsto\varepsilon^{-1}{\mathcal F}_\varepsilon^*{\mathfrak S}^{\rm m}_1\left(\varepsilon^{-\gamma-2} A^{\textrm{hom},2}_{\chi}+I \right)^{-1}\!{\mathcal F}_\varepsilon U,
 \label{correctors}
 \end{equation}
where ${\mathfrak S}^{\rm m}_1$ is the solution operator for \eqref{revision1003}.
The resulting error is of order 
$$
\max\bigl\{|\chi|, \varepsilon^{-1}|\chi|^2\bigr\} \left(\max\left\{\varepsilon^{-\gamma-2}|\chi|^2,1\right\}\right)^{-1},
$$
so that 
\begin{equation*}
 \begin{aligned}
\Biggl\| R^h\left(\varepsilon^{-\gamma} \mathcal{A}^{\varepsilon}+I\right)^{-1}\bigl(\widehat{F}_1,\widehat{F}_2, \widecheck{F\,}\!\!_3\bigr)^\top&-\left(\begin{array}{c}\Re\Bigl\{\left(\left(  \varepsilon^{-\gamma}\mathcal{A}^{{\rm hom}, 2} +I \right)^{-1}+\varepsilon {\mathfrak K}^{\rm m}_1 \right)SR^h\bigl(\widehat{F}_1,\widehat{F}_2 \bigr)^\top\Bigr\}\\[0.5em]0\end{array}\right)\Biggr\|_{H^1(\Pi , \mathbb{R}^3)}\\[0.7em] 
 &\leq C \max\bigl\{\varepsilon^{(\gamma+2)/2}, \varepsilon^{\gamma+1}\bigr\} \Bigl\|R^h\bigl(\widehat{F}_1,\widehat{F}_2, \widecheck{F\,}\!\!_3\bigr)^\top\Bigr\|_{L^2(\Pi ,\mathbb{R}^3)}.
 \end{aligned} 
 \end{equation*}
Note that the cost of neglecting the projections 
onto the eigenspaces of of  $(\varepsilon^{-\gamma-2} \mathcal{A}_\chi+I)^{-1}$ corresponding to lower eigenvalues is of order $\varepsilon^{\gamma+1}$ in terms of the operator $(\varepsilon^{-\gamma}{\mathcal A}^\varepsilon+I)^{-1}.$
\\
\vskip 0.01cm
\noindent\textbf{Higher-order $L^2\to L^2$ estimate.} 
Similarly to the rationale of Section \ref{app_first}, for all $\chi\in Q_{\rm r}'$ we define a corrector ${\mathfrak K}^{\rm m}_2$ using \eqref{sinisa3*second}, \eqref{revision1003}, \eqref{sinisa101pon2}, \eqref{sinisa3*secondoo}, and obtain a fibre-wise error of order 
$$
|\chi|^2 \left(\max\left\{\varepsilon^{-\gamma-2}|\chi|^2,1\right\}\right)^{-1}.
$$
In terms of the original operator family ${\mathcal A}^\varepsilon,$
this yields the estimate
\begin{equation*}
\begin{aligned}
\Bigg\|R^h\left(\varepsilon^{-\gamma} \mathcal{A}^{\varepsilon}+I\right)^{-1}\bigl(\widehat{F}_1,\widehat{F}_2, \widecheck{F\,}\!\!_3\bigr)^\top&-\Re\Biggl\{\left(\begin{array}{c}\left(\left(\varepsilon^{-\gamma}\mathcal{A}^{{\rm hom}, 2} +I \right)^{-1}+\varepsilon {\mathfrak K}^{\rm m}_1 \right)SR^h\bigl(\widehat{F}_1, \widehat{F}_2\bigr)^\top\\[0.6em]0\end{array}\right)\\[0.7em]
&+{\mathfrak K}^{\rm m}_2R^h\bigl( \widehat{F}_1,  \widehat{F}_2, \widecheck{F}\!\!_3\bigr)^\top\Biggr\}\Bigg\|_{L^2(\Pi, \mathbb{R}^3)}
\leq C \varepsilon^{\gamma+2} \Bigl\|R^h\bigl(\widehat{F}_1,\widehat{F}_2, \widecheck{F\,}\!\!_3\bigr)^\top\Bigr\|_{L^2(\Pi,\mathbb{R}^3)}.
\end{aligned} 
\end{equation*}

\begin{remark} 
\label{expl3}  
1. For the case of a bounded plate of thickness $h$ which is homogeneous and isotropic, the eigenvalues  of orders $h^2$ and $1$ were approximated in \cite{Dauge}.
We find that among our estimates the most interesting are those for $\gamma=2$, $\delta=1$ and $\gamma=\delta=0.$ (The former is the scaling studied in \cite{ciarlet} and earlier in \cite{Miara}, while the second corresponds to the unscaled forces and unscaled time in the related evolution problem, see also a discussion in \cite{BuzChVZ} and in Section \ref{intro_section} of the present paper.) Furthermore, the scaling we adopt for the horizontal forces is standard for the derivation of stationary and evolution equations for plates, see \cite{ciarlet}. Our analysis shows that while this scaling does not influence the asymptotic behaviour of spectral projections, it will be important in the study of the evolution equations, which we postpone to a future publication. Moreover, it can be seen as natural from the computational point of view.

Since in the present paper we are interested in analysing asymptotic regimes that do not admit a limit operator and rather necessitate generalising the approach towards ``operator asymptotics" ({\it cf.} high contrast setups \cite{ChCoARMA}, \cite{CherErshKis}, where similar situations occur), it is reasonable to admit arbitrary values of $\gamma$ and $\delta$ and express the approximation error in terms of these parameters. In a future work we will discuss the consequences of these estimates for evolution problems. When dealing with the derivation of evolution problems for linear plate theory, one usually scales the mass density or time, see \cite{Raoult}. As a consequence, there is no inertia term for the horizontal components in the limit evolution equations. However, as also shown by our analysis, the information obtained by these limit equations is partial, since {\it e.g.} in the second invariant subspace (in case of a planar-symmetric elasticity tensor), one would not need to scale the density (or time) to obtain the limit equations ({\it i.e.} in-plane waves propagate on finite, unscaled, time intervals).

2. In the case of the second invariant subspace our resolvent estimates for ${\mathcal A}^\varepsilon$ are of the ``optimal" order $\varepsilon^{\gamma+2}.$ This is due to the fact that the error of approximating the operators ${\mathcal A}_\chi$ by an effective operator on the two-dimensional subspace corresponding to the two minimal eigenvalues is of the same order as the error of neglecting the eigenspaces that correspond to eigenvalues of order $\varepsilon^{-2}.$ Notice that any further approximation should include subspaces of oscillating eigenfunctions, which would take us outside the scope of the standard homogenised operator. 

For the first invariant subspace for the operators ${\mathcal A}^\varepsilon$, our estimates for ${\mathcal A}_\chi$ in the approximation subspace corresponding to the eigenvalue of order $\varepsilon^{-2}|\chi|^4$ are not optimal. In order to regain optimality, in addition to ${\mathfrak K}^{\rm b}_2$ we would have to include other ``non-standard"  homogenisation correctors,
see also Remark \ref{rem53}. A non-standard homogenisation corrector for a higher-order $L^2\to L^2$  estimate has previously  appeared in the literature \cite{BirmanSuslina_corrector}. However, the presence of more than one non-standard corrector is novel. 

3. Notice that for  $\gamma=0$ the error of the estimate of the third component in the $H^1$ norm in the first invariant subspace is of order one. This is, however, optimal if one does not want to include non-standard correctors, since in this case the $H^1$ norms of the third components of the original and approximating solutions are of lower order than $\varepsilon^0.$ Notice also that adding non-standard correctors can improve $L^2 \to H^1$ estimates, similarly to the way it improves $L^2 \to L^2$ estimates.  
\end{remark}

\subsection{General case: Theorem \ref{main_result_general}}

\label{gen_case_sec}


\noindent\textbf{Proof of $L^2\to L^2$ estimate (\ref{general_est}).} Arguing as in Section \ref{sym_case_res}, we notice that the largest eigenvalue of the fibre operator
$\left(\varepsilon^{-\gamma-2} \mathcal{A}_\chi+I\right)^{-1}$ 
is of order $\left(\max\left\{\varepsilon^{-\gamma-2}|\chi|^4,1\right\}\right)^{-1}\!\!,$ the following two are of  order 
$\left(\max\left\{\varepsilon^{-\gamma-2}|\chi|^2,1\right\}\right)^{-1}\!\!,$ 
and all remaining eigenvalues are at most of order $\varepsilon^{\gamma+2}.$ 
In the approximation we neglect the projections onto the eigenspaces of the latter eigenvalues, which results in an error of order $\varepsilon^{\gamma+2}$ for the operators $(\varepsilon^{-\gamma}{\mathcal A}^\varepsilon+I)^{-1}.$  

Under two different scalings for the fibres ${\mathcal A}_\chi,$ we obtain resolvent estimates in two eigenspaces, determined by the eigenvalues of orders $|\chi|^4$ and
 $|\chi|^2$. The operator 
$|\chi|^{-4} \mathcal{A}_{\chi}$ has one eigenvalue of order one, the others being at least of order $|\chi|^{-2},$ while $|\chi|^{-2} \mathcal{A}_{\chi}$ has one eigenvalue of order $|\chi|^2$ and two eigenvalues of order one, while the others are at least of order $|\chi|^{-2}$.  
Integrating the functions $f_{\varepsilon,\chi} (\zeta) \left(\zeta I-|\chi|^{-4} \mathcal{A}_{\chi}\right)^{-1},$ see (\ref{function_f}), and $g_{\varepsilon,\chi} (\zeta)\left(\zeta I-|\chi|^{-2} \mathcal{A}_{\chi}\right)^{-1},$ see (\ref{function_g}), over respective contours in the $\zeta$-plane 
that contain the leading eigenvalues, {\it cf.} (\ref{marita1}), allows us to obtain error estimates on each of the related subspaces ---
the overall error is then given by the maximum of the two. 

Proceeding with the above plan, for all quasimomenta $\chi$ such that 
$|\chi|\le\rho_3$ for some $\rho_3>0$ that can be defined {\it a priori}, we choose contours $\Gamma_1$ and $\Gamma_2,$ as follows. The contour $\Gamma_1$ contains the smallest eigenvalue of the operator $|\chi|^{-4} \mathcal{A}_{\chi}$ and no other eigenvalues (the nearest one being of order $|\chi|^{-2}$), while the contour $\Gamma_2$ contains two eigenvalues of order one of the operator $|\chi|^{-2}\mathcal{A}_{\chi}$ and no other eigenvalues (of which the smallest is of order $|\chi|^2$; all the others are at least of order $|\chi|^{-2}$). We  then write
\begin{align}
(\mathcal{P}_1^\chi+\mathcal{P}_2^\chi)&\bigl(\varepsilon^{-\gamma-2}  \mathcal{A}_{\chi}+I\bigr)^{-1}(\mathcal{P}_1^\chi+\mathcal{P}_2^\chi)= 
\mathcal{P}_1^\chi\left(\varepsilon^{-\gamma-2}  \mathcal{A}_{\chi}+I\right)^{-1}\mathcal{P}_1^\chi+\mathcal{P}_2^\chi\left(\varepsilon^{-\gamma-2}  \mathcal{A}_{\chi}+I\right)^{-1}\mathcal{P}_2^\chi\nonumber\\[0.6em]
&= \frac{1}{2\pi{\rm i}} \oint_{\Gamma_1}f_{\varepsilon,\chi} (\zeta) \bigl(\zeta I-|\chi|^{-4} \mathcal{A}_{\chi}\bigr)^{-1}d\zeta
 + \frac{1}{2\pi{\rm i}} \oint_{\Gamma_2}g_{\varepsilon,\chi} (\zeta) \bigl(\zeta I-|\chi|^{-2} \mathcal{A}_{\chi}\bigr)^{-1}d\zeta,\label{integrals}
\end{align}   
where $\mathcal{P}_1^\chi$ is the projection on the one-dimensional eigensubspace of the operator $\mathcal{A}_{\chi}$ whose eigenvalue is of order $|\chi|^4$ and $\mathcal{P}_2^\chi$ is the projection on the two-dimensional eigenspace corresponding to the eigenvalues of order $|\chi|^2$, {\it cf.} (\ref{marita1}). Finally apply the result of Section \ref{first_scaling_section} to the second term in (\ref{integrals}) and the result of Section \ref{second_scaling_section} to the second term. This completes the proof of (\ref{general_est}). 
\\
\vskip 0.01cm
\noindent\textbf{$L^2\to H^1$ estimate.} Similarly to the above ({\it cf.} (\ref{correctorb}), (\ref{correctors}) for analogous formulae introduced in the analysis of the bending and membrane subspaces), we define the corrector 
$$ 
 {\mathfrak K}_1: L^2(\Pi,\mathbb{R}^3)\ni U=\varepsilon^{-1}{\mathcal F}_\varepsilon^*{\mathfrak S}_1\left(\varepsilon^{-\gamma-2}A^{\textrm{hom}}_{\chi}+I \right)^{-1}{\mathcal F}_\varepsilon{U}, \qquad \chi\in Q_{\rm r}',
 $$
where ${\mathfrak S}_1$ is the solution operator for \eqref{otto1revision0000}.
 This yields the estimates
 \begin{align*}
 &\Bigl\|R^h P_{\alpha}\left(\varepsilon^{-\gamma} \mathcal{A}^{\varepsilon}+I\right)^{-1}F
 \\[0.5em] 
 &-\left(\left(P_{\alpha}-\varepsilon \partial_{\alpha}P_3\right)\left(  \varepsilon^{-\gamma}\mathcal{A}^{\rm hom}+I \right)^{-1}+\varepsilon  P_{\alpha} {\mathfrak K}_1\right)\bigl(SR^hF_1, SR^hF_2, SR^hF_3+S_1R^hF_1+S_2R^h F_2\bigr)^\top\Bigr\|_{H^1(\Pi)}\\[0.5em]
 & \hspace{+20ex}\leq C\max\bigl\{\varepsilon^{(\gamma+2)/2},\varepsilon^{\gamma+1},\varepsilon^{3(\gamma+2)/4-1}\bigr\} \bigl\|R^h F\bigr\|_{L^2(\Pi, \mathbb{R}^3)}, \quad \alpha=1,2, \\[0.6em]
 &\Bigl\|R^h P_{3}\left(\varepsilon^{-\gamma} \mathcal{A}^{\varepsilon}+I\right)^{-1}F
 \\[0.6em] 
 &-P_3\left(\left(  \varepsilon^{-\gamma}\mathcal{A}^{\rm hom}+I \right)^{-1}+\varepsilon {\mathfrak K}_1\right)\bigl(SR^hF_1, SR^hF_2, SR^hF_3+S_1R^hF_1+S_2R^hF_2\bigr)^\top\Bigr\}\Bigr\|_{H^1(\Pi)}\\[0.6em]
 & \hspace{+20ex}\leq C  \max\bigl\{\varepsilon^{(\gamma+2)/4}, \varepsilon^{(\gamma+2)/2-1}\bigl\|R^hF\bigr\|_{L^2(\Pi ,\mathbb{R}^3)}. 
 \end{align*}
\\
\vskip -0.4cm
\noindent\textbf{Higher-order $L^2\to L^2$ estimate.} Using \eqref{marita10000}, \eqref{otto1revision0000}, \eqref{marita10002}, \eqref{marita10001}, \eqref{marita10003}--\eqref{sinisa3*secondoooo}, 
for each $\chi\in Q_{\rm r}'$ we define a corrector ${\mathfrak K}_2$ ({\it cf.} Sections \ref{app_first}, \ref{app_second} for similar correctors ${\mathfrak K}_2^{\rm b},$ ${\mathfrak K}_2^{\rm m}$ for estimates in the bending and membrane subspaces) so that
\begin{equation*}
\begin{aligned}
 &\Bigl\| R^hP_{\alpha}\left(\varepsilon^{-\gamma} \mathcal{A}^{\varepsilon}+I\right)^{-1}F
 \\[0.8em] 
 &-\Bigl(\bigl(P_{\alpha}-\varepsilon \partial_{\alpha}P_3\bigr)\bigl( \varepsilon^{-\gamma}\mathcal{A}^{\rm hom}+I\bigr)^{-1}+\varepsilon  P_{\alpha} {\mathfrak K}_1\Bigr)\bigl(SR^hF_1,  SR^hF_2, SR^hF_3+S_1R^hF_1+S_2R^hF_2  \bigr)^\top\\[0.7em]
 & \hspace{+15ex}-(P_{\alpha}-\varepsilon x_3P_3){\mathfrak K}_2 R^h F\Bigr\|_{L^2(\Pi)}
 \leq C   \varepsilon^{3(\gamma+2)/4}\bigl\|R^h F\bigr\|_{L^2(\Pi ,\mathbb{R}^3)}, \quad \alpha=1,2, 
 \\[0.6em]
 &\Bigl\| R^h P_{3}\left(\varepsilon^{-\gamma} \mathcal{A}^{\varepsilon}+I\right)^{-1}F
 \\[0.8em] 
 &-P_3\left(\left(\varepsilon^{-\gamma}\mathcal{A}^{\rm hom}+I \right)^{-1}+\varepsilon {\mathfrak K}_1\right)\bigl(SR^hF_1, SR^hF_2, SR^hF_3+S_1R^hF_1+S_2R^h F_2  \bigr)^\top-P_3{\mathfrak K}_2R^hF\Bigr\|_{L^2(\Pi)}\\[0.5em] 
 & \hspace{+15ex}
 \leq C  \varepsilon^{(\gamma+2)/2} \bigl\|R^h F\bigr\|_{L^2(\Pi,\mathbb{R}^3)}. 
 \end{aligned}
 \end{equation*}

\section{Concluding remarks}
\label{concl_sec}

 \label{rgamma} 
\noindent 1. To adapt our results to the case when $h=h(\varepsilon) \sim \varepsilon$, {\it i.e.} when there exists $\alpha,\beta>0$, independent of $\varepsilon$, such that 
 $$ \alpha \leq \delta_{\varepsilon}:=h(\varepsilon)/\varepsilon\leq \beta \qquad \forall \varepsilon>0, $$
one has to change the definition of $\mathcal{L},$ $\mathcal{A}^{{\rm hom}, 1},$ $\mathcal{A}^{\rm hom},$ see (\ref{L_definition}), (\ref{Ahoms}), as follows: 
 \begin{equation*}
 \begin{aligned} 
 \mathcal{L}(M_1,M_2) : (M_1,M_2):= \inf_{\psi \in H^1_{\#}(Q,\mathbb{R}^3)}\int_Q A\bigl({\mathfrak I}( M_1-x_3M_2)&+{\rm sym}\nabla^{\delta_{\varepsilon}} \psi\bigr):\bigl( {\mathfrak I}(M_1-x_3 M_2)+{\rm sym}\nabla^{\delta_{\varepsilon}} \psi\bigr),\\[0.1em]
 &M_1,M_2 \in \mathbb{R}^{2 \times 2}, 
 \end{aligned}
 \end{equation*}
 $$ \mathcal{A}^{\textrm{hom},1}:=\delta_{\varepsilon}^2\bigl(\widetilde{\nabla}^2\bigr)^{*} \mathcal{L}^1 \widetilde{\nabla}^2, \qquad  \mathcal{A}^{\textrm{hom}}:=\bigl(\textrm{sym}\widetilde{\nabla}, h \widetilde{\nabla}^2\bigr)^{*} \mathcal{L}\bigl(\textrm{sym}\widetilde{\nabla}, h \widetilde{\nabla}^2\bigr), 
 $$
where, as before, 
 $\nabla^{\delta_{\varepsilon}}=(\partial_{1}, \partial_{2}, \delta_{\varepsilon}^{-1}\partial_{3})$ is the 
``$x_3$-rescaled" gradient.
  	
	In this more general case, the constants on the right-hand sides of the estimates in Theorem \ref{main_result_first}, Theorem \ref{main_result_second}  and Theorem \ref{main_result_general} depend on $\alpha, \beta$. Moreover, in the expressions on the left-hand sides of the estimates containing the term $\varepsilon x_3$ (related to the Kirchhoff-Love ansatz) in Theorem \ref{main_result_first}, Theorem \ref{main_result_general}  as well as in Section \ref{app_first}, Section \ref{gen_case_sec}, this term should be replaced by $hx_3.$ 
	 These modifications follow from an analogue of Lemma \ref{neN_est}, which in turn is based on the existence of  $C(\alpha,\beta)>0$ such that
\begin{equation}
 	\|u\|_{H^1(Q, \mathbb{C}^3)}  \leq  C(\alpha,\beta)\bigl( \|u\|_{L^2(Q,\mathbb{C}^3)}+\bigl\|\textrm{sym}\nabla^{\delta_{\varepsilon}} u\bigr\|_{L^2(Q, \mathbb{C}^3)}\bigr) \quad \forall u \in H^1(Q, \mathbb{C}^3),
\label{C_one}
\end{equation}	
 as well as
 \begin{equation}
 \begin{aligned} 
 \|u_1- a_1\delta_{\varepsilon}x_3-dy_2-c_1 \|_{H^1(Q, {\mathbb C})} &\leq C(\alpha,\beta)\bigl\| \textrm{sym}\nabla^{\delta_{\varepsilon}} u\bigr\|_{L^2(Q, {\mathbb C}^{3\times3})},\\[0.6em]
 \|u_2- a_2\delta_{\varepsilon} x_3+dy_1-c_2 \|_{H^1(Q, {\mathbb C})} &\leq C(\alpha,\beta)\bigl\| \textrm{sym} \nabla^{\delta_{\varepsilon}} u\bigr\|_{L^2(Q, {\mathbb C}^{3\times3})}, \\[0.6em]
 \|u_3+ a_1y_1+a_2y_2- c_3 \|_{H^1(Q, {\mathbb C})} &\leq C(\alpha,\beta)\bigl\| \textrm{sym}\nabla^{\delta_{\varepsilon}} u\bigr\|_{L^2(Q, {\mathbb C}^{3\times3})}
 \end{aligned}
 \label{C_two}
 \end{equation}
 for all $u \in H^1(Q, \mathbb{C}^3),$ where 
 \begin{align*}
 a_\alpha:=&\int_Q\bigl(\delta_{\varepsilon}^{-1}\partial_3 u_\alpha-\partial_\alpha u_3\bigr), \quad \alpha=1,2,
 \qquad d:=\int_Q(\partial_2 u_1-\partial_1 u_2),\\[0.4em]  
c_1:=&\int_Q u_1-d, \qquad c_2:=\int_Q u_2+d, \qquad c_3:=\int_Q u_3+a_1+a_2.
 \end{align*}   
In order to proceed with the corresponding version of the asymptotic expansion in Section \ref{asymptotic_proc}, the following Korn-type inequality can also be proved with $C(\alpha, \beta)>0:$
\begin{equation}
\biggl\|u-\int_Q u\biggr\|_{H^1(Q,\mathbb{C}^3)} \leq  C(\alpha,\beta) \bigl\|\textrm{sym} \nabla^{\delta_{\varepsilon}}u\bigr\|_{L^2(Q, \mathbb{C}^{3\times 3})} 
\qquad \forall u \in H^1_{\#}(Q, \mathbb{C}^3). 
\label{periodicKorn}
\end{equation}
The value of $C(\alpha, \beta)$ such that 
(\ref{C_one})--(\ref{periodicKorn}) hold can be obtained by controlling uniformly the constants in the Korn inequalities on cuboids $Q_{\rm r} \times (-\varkappa/2, \varkappa/2)$, $\varkappa\in[\alpha, \beta].$ The way to ensure such control, in turn, follows from 
 \cite[Theorem 1.3]{Griso}, see also 
\cite[Section 2.4]{OShY}.

\vskip 0.3cm
\noindent 2. One of the advantages of our method is that all constants appearing on the 
right-hand sides of the error estimates can be evaluated explicitly, as in \cite{ChCoARMA}, \cite{CherErshKis}. These constants depend on: the Poincar\'{e} and Korn constants for the cube, the norm of the $H^1\hookrightarrow L^2$ trace operator for the cube, 
and the constant $C_1$ in the inequality
\begin{equation}
\|\nabla u \|_{L^2(Q_{\rm r}, \mathbb{C}^2)} \leq C_1\|\textrm{sym} \nabla u\|_{L^2(Q_{\rm r}, \mathbb{C}^{2 \times 2})}\quad \forall u \in H^1_{\chi}(Q_{\rm r}, \mathbb{C}^2),\quad\chi\in Q_{\rm r}'.
\label{C1_estimate}
\end{equation}
All these quantities are used in the proof of Lemma \ref{neN_est} or its analogue for the general case $h\sim\varepsilon$ (where for the latter we replace $\nabla$ by $\nabla^{\delta_\epsilon}$).  Furthermore, for the analysis of Section \ref{asymptotic_proc} we require the existence of $C_2>0$ such that
\begin{equation}
\|\nabla u \|_{L^2(Q,\mathbb{C}^{3 \times 3})} \leq C_2 \|\textrm{sym} \nabla u \|_{L^2(Q,\mathbb{C}^{3 \times 3})}\quad \forall u \in H^1_{\#} (Q, \mathbb{C}^3),
\label{C2_estimate}
\end{equation}
which together with the Poincar\'{e} inequality provides the necessary versions of (\ref{C_one}) and (\ref{periodicKorn}). 
The constant $C_1$ in (\ref{C1_estimate}) can be obtained by Fourier transform, 
while $C_2$ in (\ref{C2_estimate}) evaluated on the basis of $C(\alpha, \beta)$ in (\ref{C_one})---(\ref{periodicKorn}), by using the analysis of Lemma \ref{neN_est}. 

\vskip 0.3cm

\noindent 3. The method of the present work can be adapted for obtaining order-sharp operator-norm asymptotic estimates for rod-like and shell-like structures, see \cite{CDVZ}. Finally, one can combine the method of the present paper with the strategy of \cite{ChDO}, who developed an approach to order-sharp norm-resolvent estimates for multi-dimensional homogenisation problems with respect to arbitrary periodic Borel measures (including measures with non-trivial singular components). A brief reflection will convince the reader that the case of plate-like structures containing singular components can thereby be analysed to the same order of approximation error, and for the same range of energy and force density scalings as in the present work. 

\section*{Acknowledgements}
KC is grateful for the support of
the Engineering and Physical Sciences Research Council (EPSRC): Grant EP/L018802/2 ``Mathematical foundations of metamaterials: homogenisation, dissipation and operator theory''. IV has been supported by the Croatian Science Foundation under Grant agreement No.~9477 (MAMPITCoStruFl) and Grant agreement No. IP-2018-01-8904 (Homdirestroptcm). KC would also like to thank the Isaac Newton Institute for Mathematical Sciences, Cambridge, for support and hospitality during the programme ``The mathematical design of new materials", where work on this paper was undertaken. This work was supported by EPSRC grant no EP/R014604/1.


\begin{thebibliography}{9}
	
	\bibitem{BLP}
	Bensoussan, A., Lions, J.-L., Papanicolaou, G. C., 1978. {\it Asymptotic Analysis for Periodic Structures,} North-Holland.
	
	\bibitem{BP} Bakhvalov, N., Panasenko, G., 1989. {\it Homogenisation: averaging processes in periodic media.} 
	Mathematics and its Applications (Soviet Series), {\bf 36}. Kluwer Academic Publishers Group, Dordrecht.
	
	\bibitem{BirmanSolomyak}
Birman, M. S., and Solomjak, M. Z., 1987. {\it Spectral theory of selfadjoint operators in {H}ilbert space.} Mathematics and its Applications (Soviet Series). D. Reidel Publishing Co., Dordrecht.
	
	 \bibitem{BirmanSuslina} Birman, M. Sh., Suslina, T. A.,  2004. Second order periodic differential operators. Threshold properties and homogenisation. {\it St. Petersburg Math. J.} {\bf 15}(5), 639--714.
	 
	  \bibitem{BirmanSuslina_corrector} Birman, M. Sh., Suslina, T. A., 2006. Averaging of periodic elliptic differential operators taking a corrector into account. {\it St. Petersburg Math. J.} {\bf 17} (6), 897--973.
	
	\bibitem{BirmanSuslina_hyperbolic}
	Birman, M. Sh., Suslina, T. A., 2009. Operator error estimates for the averaging of nonstationary periodic equations. {\it St. Petersburg Math. J.} {\bf 20}(6), 873--928.
	
	\bibitem{BFF}
	 Braides, A., Fonseca, I., Francfort, G., 2000. 3D-2D asymptotic analysis for inhomogeneous thin films. {\it Indiana Univ. Math. J.} {\bf 49} (4), 1367--1404.
	
	\bibitem{BV}
	Bukal, M., Vel\v{c}i\'{c}, I., 2017. On the simultaneous homogenization and dimension reduction in elasticity and locality of $\Gamma$-closure. {\it Calc. Var. Partial Differential Equations} {\bf 56}(3), Paper No. 59, 41 pp.
	
           \bibitem{BuzChVZ} Bu\v{z}an\v{c}i\'{c} M., Cherednichenko K., Vel\v{c}i\'{c} I., and \v{Z}ubrini\'{c} J. Spectral and evolution analysis of elastic plates in high-contrast regime, preprint: \url{https://arxiv.org/pdf/2105.05597.pdf}. 
 
           \bibitem{Caill} Caillerie, D., Nedelec, J. C., 1984. Thin periodic and elastic plates. {\it Math. Method. Appl. Sci.} {\bf 6} (1), 159--191. 


	\bibitem{CDVZ}
	Cherednichenko, K., Vel\v{c}i\'{c}, I., and \v{Z}ubrini\'{c}, J., 2020. Sharp operator-norm asymptotics for thin elastic rods with rapidly oscillating periodic properties, {in preparation.}
	
	\bibitem{CC}  
	Cherdantsev, M., Cherednichenko, K., 2015. Bending of thin periodic plates. {\it Calc. Var. Partial Differential Equations} {\bf 54} (4), 
	4079--4117.
	
	
	\bibitem{ChCoARMA}
	Cherednichenko, K. D, Cooper, S., 2016. Resolvent estimates for high-contrast elliptic problems with periodic coefficients. {\it Archive for Rational Mechanics and Analysis} {\bf 219}(3), 1061--1086.
	
	\bibitem{CherErshKis} Cherednichenko, K. D., Ershova, Yu. Yu., and Kiselev, A. V., 2020. Effective behaviour of critical-contrast PDEs: micro-resonances, frequency conversion, and time dispersive properties. I. {\it Communications in Mathematical Physics,} {\bf 375}, 1833--1884.
	
	\bibitem{ChDO}
	Cherednichenko, K., D'Onofrio, S., 2018. Operator-norm convergence estimates for elliptic homogenisation problems on periodic singular structures. {\it Journal of Mathematical Sciences} {\bf 232}(4), 558--572.
	
	\bibitem{ciarlet}
	Ciarlet, P. G., 1997. {\it Mathematical Elasticity, Volume II: Theory of Plates,} North Holland. 
	
	\bibitem{Cioranescu_Damlamian_Griso} 
	Cioranescu, D., Damlamian, A., Griso, G., 2018. {\it The Periodic Unfolding Method,} Springer.
		
	\bibitem{CD}
	Cioranescu, D., Donato, P., 1999. {\it An Introduction to Homogenization,} Oxford University Press.	
		
		
	\bibitem{Quasi_Cooper}
	Cooper, S., 2018. Quasi-periodic two-scale homogenisation and effective spatial dispersion in high-contrast media. {\it Calc. Var. Partial Differential Equations} {\bf 57}:76.

            \bibitem{Dam} Damlamian, A.,  Vogelius, M., 1987. Homogenization limits of the equations of
elasticity in thin domains, {\it SIAM J. Math. Anal.}  {\bf 18} (2), 435--451.
	
	\bibitem{Davies}
	Davies, B., 1995. {\it Spectral Theory and Differential Operators,} Cambridge University Press.
	
	\bibitem{Dauge}  Dauge, M., Djurdjevic, I., Faou, E., R\"{o}ssle, A., 1999. Eigenmode asymptotics in thin elastic plates. {\it J. Math. Pures Appl.} {\bf 78}(9), 925--964.
	
	\bibitem{Gelfand}
	Gel'fand, I. M., 1950. Expansion in characteristic functions of an equation with periodic coefficients. (Russian) {\it Doklady Akad. Nauk SSSR (N.S.)} {\bf 57}, 1117--1120.
	
	\bibitem{Griso_2006} Griso, G., 2006. Interior error estimate for periodic homogenisation. {\it Anal. Appl.,} {\bf 4}(1), 61--79.
	
	\bibitem{Griso}
	Griso, G., 2008. Decompositions of displacements of thin structures, {\it J. Math. Pures Appl. (9)} 89(2), 199--223.
	
	\bibitem{HNV}
	 Hornung, P., Neukamm, S., Vel\v{c}i\'{c}, I., 2014 Derivation of a homogenized nonlinear plate theory from 3d elasticity. {\it Calc. Var. Partial Differential Equations} {\bf 51} (3-4), 677--699.
	
        
        \bibitem{Kenig} Kenig, C. E., Lin, F., Shen, Z., 2012. Convergence rates in $L^2$ for elliptic homogenization problems. {\it Archive for Rational Mechanics and 
        Analysis} {\bf 203}(3), 1009-1036.
	
        \bibitem{Kuchment} Kuchment, P., 1993. {\it Floquet Theory for Partial Differential Equations}, Birkh\"{a}user.	
        
        \bibitem{Meshkova_hyperbolic_full} 
        Meshkova, Yu. M., 2018. On operator estimates for homogenization of hypebolic systems with periodic coefficients, 42 pp., arXiv:1705.02531.
        
        \bibitem{Meshkova_hyperbolic_Math_Notes}
        Meshkova, Yu. M., 2019. On the homogenization of periodic hyperbolic systems. {\it Math. Notes} {\bf 105}(5--6), 929--934.
              
        \bibitem{Meshkova_Suslina}
        Meshkova, Yu. M., Suslina, T. A., 2016. Homogenization of initial boundary value problems for parabolic systems with periodic coefficients. {\it Appl. Anal.} {\bf 95}(8), 1736--1775.
        
        \bibitem{Miara} 
        Miara, B., 1994. Justification of the asymptotic analysis of elastic plates. I. The linear case, {\it Asymptotic Anal.} {\bf 8}, 259--276.
        
        \bibitem{NV_vK}
        Neukamm, S. Vel\v{c}i\'{c}, I., 2013. Derivation of a homogenized von-K\'{a}rm\'{a}n plate theory from 3D nonlinear elasticity. {\it Math. Models Methods Appl. Sci.} {\bf 23}(14), 2701--2748.
        
	\bibitem{OShY}
	Oleinik, O. A.,  Shamaev, A. S., Yosifian, G. A., 1992. {\it Mathematical Problems in Elasticity and Homogenization,} Amsterdam: North-Holland.
	
	\bibitem{Panasenko_book}
	Panasenko, G., 2005. {\it Multi-Scale Modelling for Structures and Composites.} Springer, Dordrecht.
	
	\bibitem{Pastukhova}
	Pastukhova, S. E., 2012. Approximations of the exponential of an operator with periodic coefficients. Problems in mathematical analysis. No. 62. {\it J. Math. Sci. (N.Y.)} {\bf 181}(5), 668--700.
	
	\bibitem{Raoult}
       Raoult, A., 1985. Construction d'un modele d'\'evolution de plaques avec termes d'inertie de rotation, {\it Annali di Matematica Pura ed Applicata} {\bf139}, 361--400.
	
	\bibitem{ReedSimon}
	Reed, M., Simon, B., 1978. {\it Methods of Modern Mathematical Physics IV. Analysis
of Operators}, Academic Press, New York.	
	
	 \bibitem{Suslina_parabolic}
	 Suslina, T. A., 2007. Homogenization of a periodic parabolic Cauchy problem. Nonlinear equations and spectral theory, 201--233, {\it Amer. Math. Soc. Transl. Ser. 2} {\bf 220}, {\it Adv. Math. Sci.} {\bf 59,} Amer. Math. Soc., Providence, RI.
	 
	\bibitem{Suslina_parabolic_corrector}
	 Suslina, T., 2010. Homogenization of a periodic parabolic Cauchy problem in the Sobolev space $H^1({\mathbb R}^d).$ {\it Math. Model. Nat. Phenom.} {\bf 5}(4), 390--447.
	
	\bibitem{Suslina_Dirichlet}
	Suslina, T. A., 2013. Homogenization of the Dirichlet problem for elliptic systems: $L^2$-operator error estimates. {\it Mathematika} {\bf 59}(2), 463--476.
	
	\bibitem{Suslina_Neumann}
	 Suslina, T., 2013. Homogenization of the Neumann problem for elliptic systems with periodic coefficients. {\it SIAM J. Math. Anal.} {\bf 45}(6), 3453--3493.
	
	
	\bibitem{Suslina_two_parametric}
	Suslina, T. A., 2014. Approximation of the resolvent of a two-parameter quadratic operator pencil near the lower edge of the spectrum. {\it St. Petersburg Math. J.} {\bf 25} (5), 869--891.
	
	\bibitem{Suslina_perforated}
	Suslina, T. A., 2018. Spectral approach to homogenization of elliptic operators in a perforated space. {\it Ludwig Faddeev memorial volume,} 481--537, World Sci. Publ., Hackensack, NJ.
	
	
	
	\bibitem{ZhikovPastukhova}
       Zhikov, V. V., Pastukhova, S. E., 2005. On operator estimates for some problems in homogenization theory. {\it Russ. J. Math. Phys.} {\bf 12} (4), 515--524.
	
	

	
	
	
	
\end{thebibliography}
\end{document}